\documentclass[12pt,reqno]{article}
\usepackage[utf8]{inputenc}
\usepackage{enumerate}
\usepackage{lmodern}
\usepackage[T1]{fontenc}
\usepackage{verbatim}
\usepackage{textcomp}

\usepackage[english]{babel}
\usepackage[a4paper,vmargin={3.5cm,3.5cm},hmargin={2.5cm,2.5cm}]{geometry}
\usepackage[font=sf, labelfont={sf,bf}, margin=1cm]{caption}
\usepackage[pdftex]{hyperref}
\usepackage[pdftex]{color,graphicx}
\usepackage{xcolor}

\usepackage{amsmath,amsfonts,amssymb,amsthm,mathrsfs,mathtools,bbm}
\usepackage{empheq}
\newtheorem{theorem}{Theorem}[section]
\newtheorem{corollary}[theorem]{Corollary}
\newtheorem{lemma}[theorem]{Lemma}
\newtheorem{proposition}[theorem]{Proposition}

\theoremstyle{remark}
\newtheorem{remark}[theorem]{Remark}
\theoremstyle{definition}

\def\RR{\mathbb{R}}
\def\R{\mathbb{R}}

\def\NN{\mathbb{N}}

\def\PP{\mathbb{P}}

\def\EE{\mathbb{E}}

\newcommand{\1}{{\bf 1}}
\def\al{\alpha}
\def\be{\beta}

\def\la{\lambda}

\newcommand{\xxi}{\xi}
\newcommand{\cP}{\mathcal{P}}
\newcommand{\cX}{\mathcal{X}}
\newcommand{\cC}{\mathcal{C}}

\newcommand{\cY}{\mathcal{Y}}
\newcommand{\ytil}{\tilde{y}}

\def\b0{\mathbf{0}}

\def\bN{\mathbf{N}}
\def\dtv{d_{\mathrm{TV}}}

\def\d2{d_2}

\newcommand{\bea}{\begin{eqnarray}}
\newcommand{\eea}{\end{eqnarray}}
\newcommand{\bean}{\begin{eqnarray*}}
\newcommand{\eean}{\end{eqnarray*}}
\newcommand{\eps}{\varepsilon}

\newcommand{\Po}{{\cal P}}
\renewcommand{\emptyset}{\varnothing}

\newcommand{\fmax}{f_{\rm max}}
\newcommand{\Cor}{\mathsf{Cor}}
\newcommand{\A}{\mathsf{A}}
\newcommand{\tod}{\overset{d}\longrightarrow}
\newcommand{\toP}{\overset{\PP}\longrightarrow}

\renewcommand{\eta}{\cP}
\newcommand{\Gum}{{\mathsf{Gu}}}
\newcommand{\PRV}{{\mathsf{Po}}}

\numberwithin{equation}{section}

\DeclareMathOperator\diam{diam}
\newcommand{\dist}{{\rm dist}}
\begin{document}
\title{\bf Fluctuations of the connectivity threshold and largest  nearest-neighbour link}
%\thanks{Supported by EPSRC grant EP/T028653/1 }
\author{
	Mathew D. Penrose$^1$ and Xiaochuan Yang$^1$  \\
{\normalsize{\em University of Bath and Brunel University London}} 
}

\date{\today}

\maketitle
\setcounter{footnote}{1}
 \footnotetext{Supported by EPSRC grant EP/T028653/1 }

\begin{abstract}   
	Consider a random uniform sample of $n$ points in a compact region $A$ of Euclidean $d$-space, $d \geq 2$, with a smooth or (when $d=2$) polygonal boundary.  Fix $k \in \NN$.  Let $T_{n,k}$  be the threshold $r$ at which the geometric graph on these $n$ vertices with distance parameter $r$ becomes $k$-connected.  We show that if $d=2$ then $n (\pi/|A|) T_{n,1}^2 - \log n$ is asymptotically standard Gumbel.  For $(d,k) \neq (2,1)$, it is $n (\theta_d/|A|) T_{n,k}^d - (2-2/d) \log n - (4-2k-2/d) \log \log n$ that converges in distribution to a nondegenerate limit, where $\theta_d$ is the volume of the unit ball.  The limit is Gumbel with scale parameter 2 except when $(d,k)=(2,2)$ where the limit is two component extreme value distributed.  The different cases reflect the fact that boundary effects are more more important in some cases than others.  We also give similar results for the largest $k$-nearest neighbour link $U_{n,k}$ in the sample, and show $T_{n,k}=U_{n,k}$ with high probability.  We provide estimates on rates of convergence and give similar results for Poisson samples in $A$.  Finally, we give similar results even for non-uniform samples, with a less explicit sequence of centring constants.
%\noindent{\bf AMS 2010 Classification:}
\end{abstract}
\tableofcontents

%LATER - {\bf Things still to do: Non-convex polygons(?)}

%Maybe `Limit distribution of the connectivity threshold' a better title?

\section{Introduction}

\subsection{Overview and motivation} 
This paper is concerned with the threshold at which the random geometric graph
becomes connected.  This graph is defined as follows.
Let $d \in \NN$.
Given a finite set $\cX \subset \R^d$, and $r>0$, the {\em geometric graph}
$G(\cX,r)$ has vertex set $\cX$, 
with an edge drawn between any two vertices at Euclidean distance
at most
$r$ from each other. We say $G(\cX,r)$ is \textit{$k$-connected} if it
is not possible to disconnect the graph by removing $k-1$ or fewer vertices.
(in particular, 1-connectivity is the same as connectivity). 
The {\em $k$-connectivity threshold} of $\cX$ is the number
\begin{align*}
	M_k(\cX) := \inf\{r>0: G(\cX,r) \mbox{ is $k$-connected} \}.
\end{align*}
An alternative characterisation of $M_1(\cX)$ is in terms
of  {\em minimal spanning
tree} (MST). A MST on $\cX$ is a tree spanning $\cX$ that minimises the 
total (Euclidean) length of the edges. It is not hard to see
that $M_1(\cX)$ equals the longest edge length of a MST on $\cX$.

For the {\em random} geometric graph,
the vertex set $\cX$
is given by the set of points of a Poisson point process $\eta_n$
in $\RR^d$ with intensity measure $n\nu$, where $\nu$ is a probability
measure on $\R^d$ with probability  density function $f:\RR^d\to [0,\infty)$,
and $n\in(0,\infty)$ is the mean number of vertices.
Alternatively, for $n \in \NN$ we can take $\cX = \cX_n$, 
where $\cX_n$ denotes 
a binomial point process 
whose points are $n$ independent random $d$-vectors
with common density $f$.
Since the vertices are placed randomly in $\RR^d$,  the thresholds
$M_k(\cX_n)$ and
\begin{align}
	M_{n,k} := M_k(\eta_n)
	= \inf\{r>0: G(\eta_n,r) \mbox{ is $k$-connected} \}
	\label{e:defM_n}
\end{align}
are  random variables.

In this paper we investigate the limiting behaviour of 
the connectivity threshold $M_{n,k}$ and $M_k(\cX_n)$  for large $n$ and fixed $k\in\NN$.
We assume throughout that $d \geq 2$.
We consider a broad class of measures $\nu$, subject to the
{\em working assumption} that  $\nu$ has compact support $A 
\subset \R^d$, and its density $f$ is continuous 
and bounded away from zero on $A$.
As  $n$ grows, the spacing between vertices becomes smaller, 
so one expects to have $M_{n,k}\to 0$ as $n\to \infty$.
A simple consideration by computing typical spacing of vertices leads to the belief that $M_{n,k}$ should decay more slowly than $n^{-1/d}$, in the sense that
$n M_{n,k}^d$ should tend to infinity in probability. 
In the special case where $\nu$ is the uniform distribution on $[0,1]^d$,
it is known \cite{Pen97,Pen99a,Pen}, that
there is an explicit sequence of centring constants $(a_n)_{n \geq 1}$
satisfying $a_n \to \infty$ as $n \to \infty$ such that 
\begin{align}\label{e:2nd_order}
nM_{n,k}^d -  a_n   \overset{d}\longrightarrow X; ~~~~~~
	nM_{k}(\cX_n)^d -  a_n   \overset{d}\longrightarrow X,
\end{align}
where $X$ is an explicit non-degenerate random variable. Clearly
\eqref{e:2nd_order} 
is what is needed to determine $\lim_{n \to \infty}
\PP[M_{n,k} \leq r_n]$ for any sequence $(r_n)_{n \geq 1}$ such that the
limit exists.

In the present paper, we show that (\ref{e:2nd_order}) holds
for suitable $a_n$ and $X$, for a broad class of measures $\nu$.
While one might perhaps anticipate that the limiting behaviour of the form
\eqref{e:2nd_order}
would carry over from the uniform distribution on $[0,1]^d$ to more general
$\nu$ satisfying our working assumption, proving this seems to be considerably
harder that one might naively expect, and 
in the last 20 years or so there has been limited progress in proving such
results.  It turns out that
even in the uniform case where $f$ is constant on $A$,
boundary effects are important because
the `most isolated' vertex is likely to lie near the boundary
when $d \geq 3$. Thus 
the formula for $a_n$, even when $\nu$ is uniform on $[0,1]^d$,
is quite complicated (see \eqref{e:cubelim} below) due to having
to consider all of the
different kinds of faces making up the boundary, and does
not necessarily provide much insight into the appropriate choice
of centring constants for other $A$. In the non-uniform case, 
determining appropriate constants $a_n$ is even harder because they depend 
in a delicate way on how $f$ approaches its minimum, both in
the interior and on the boundary of $A$.
 
 In this work we chiefly consider the case where $\partial A$ is smooth. 
 In the uniform case we determine an explicit sequence of centring constants
 $a_n$ such that (\ref{e:2nd_order}) holds for suitable $X$. In the non-uniform
 case we are still (in most cases) able to derive \eqref{e:2nd_order}
 on taking $a_n$ to be the median of the distribution  of $nM_{n,k}^d$.
 Part of our proof involves approximating $A$ 
 with a polyhedral set $A_n$ with the
 spacing between vertices of $A_n$ decreasing slowly as $n$ becomes large.
 This technique was developed recently for certain random coverage
 problems in \cite{Pen22,HPY23}, and its availability
 is one reason why this problem is more tractable 
 now than it was before.
 
 We also address the case where $d=2$ and $A$
 is polygonal; this case could be of importance for some
 applications, and it turns out that
 %boundary effects are less complicated for $d=2$ than for $d \geq 3$. 
 the effects of corners are asymptotically negligible for $d = 2$.
 We hope to deal with the case
 of polytopal $A$ in dimension $d \geq 3$ in future work.

 Understanding the connectivity threshold is important for a
 variety of applications. In telecommunications, the vertices could
 represent mobile transceivers and one might be interested in whether
 the network of transceivers is connected; see e.g. \cite{GK99}.
 In topological data analysis (TDA), detecting connectivity is a fundamental step for inspecting all other higher
dimensional topological features 
(here the dimension of the ambient space may be very high). See
also applications in machine learning
(clustering), statistical tests (e.g. for outliers), spatial
epidemics or forest fires (see the description in \cite{Pen97})
%todo:{\bf (more refs/details on applications here?)}.   

Note that (\ref{e:2nd_order}) implies the weaker statement
that the sequence $(nM_{n,k}^d -a_n)_{n \geq 1}$ is tight as $n \to \infty$,
and even in the (rather exceptional) cases where we cannot prove
(\ref{e:2nd_order}), we shall prove this weaker statement.
Tightness in turn implies $nM_{n,k}^d/a'_n \to 1$ in probability as
$n \to \infty$, for any sequence $(a'_n)_{n \geq 1}$ satisfying
$a'_n/a_n \to 1$ as $n \to \infty$.
Another direction of research (not followed
in the present paper) is to improve this
to  almost sure convergence, i.e. a strong law of large numbers (SLLN),
 under the natural coupling of $(\cX_n, n \geq 1)$:
\begin{align}\label{e:1st_order}
	\frac{nM_{k}(\cX_n)^d}{a'_n} \overset{a.s.}\longrightarrow 1,
\end{align}
or to establish (\ref{e:1st_order}) for some $a'_n$
even in cases when (\ref{e:2nd_order})
is not known; see \eqref{e:SLLN_smooth}, 
\eqref{e:SLLN_cube}
below.
%In previous work, \cite{Pen99b, PYpoly} provide a SLLN in
%the case of $A$ having a smooth boundary, or $A$ polytopal,
%with $\nu$ not necessarily uniform on $A$,
%but no result along the lines
%of \eqref{e:2nd_order}.
\subsection{The largest $k$-nearest-neighbour link}
\label{ss:LNNL}

Given $x \in \R^d$ and $r >0$, 
we denote the closed Euclidean ball of radius $r$, centred at $x$,
 by either  $B_r(x)$ or $B(x,r)$.
 Given finite $\cX \subset \R^d$ we define 
 the {\em largest $k$-nearest-neighbour link} $L_k(\cX)$ of $\cX$,
 by
\begin{align}
	 L_k(\cX)
	:= \begin{cases}
		\max_{x\in\cX} \left(
	%	\min_{\cY \subset \cX \setminus \{x\}: |\cY| = k}
	%	\left( \max_{y \in \cY } \|y-x\| \right) 
		\inf\{r >0: \cX(B_r(x)) > k\}
		\right)
		&{\rm if} ~ |\cX| \geq k+1 ,
	%	\\
	%	\max_{x\in\cX} \left(
	%	\min_{\cY \subset \cX \setminus \{x\}: |\cY| = k}
	%	\left( \max_{y \in \cY } \|y-x\| \right) \right) 
	%	&{\rm if} ~ |\cX| \geq k+1 ,
		\\
		0 & {\rm otherwise,}
	\end{cases} 
	\label{e:defL_n}
\end{align}
where $|\cX|$ denotes the number of elements of $\cX$ and
$\cX(\cdot):= |\cX \cap \cdot|$ denotes the counting measure associated with 
$\cX$.
Note that if $|\cX| \geq k+1$ and $x \in \cX$, 
then $\inf \{r >0: \cX(B_r(x)) >k\}$
%$ 
%		\min_{\cY \subset \cX \setminus \{x\}: |\cY| = k}
%\left( \max_{y \in \cY } \|y-x\|  \right)$
is the distance from $x$ to its $k$-nearest neighbour in $\cX$.
Note also that $L_k(\cX) \leq M_k(\cX)$ since
if $r < L_k(\cX)$ and $|\cX| \geq k+1$, then
$G(\cX,r)$ has at least one vertex with degree less than $k$ and therefore is not 
$k$-connected, so $r \leq M_k(\cX)$.

 Our analysis of $M_{n,k}$ and $M_k(\cX_n)$ will involve first investigating
 $L_{n,k}:= L_k(\eta_n)$ and $ L_k(\cX_n)$.
{\em A priori}, it is not obvious that $L_k(\cX_n)$ is a sharp
lower bound for $M_k(\cX_n)$;
nevertheless, it is known in some cases (see Section \ref{ss:lit})
that $M_k(\cX_n)$ enjoys the same limiting behaviour as $L_k(\cX_n)$, and even
sometimes that
\begin{align}
	\lim_{n \to \infty} \PP[ L_k(\cX_n) = M_k(\cX_n) ] = 1.
\label{e:LM}
\end{align}
Equation (\ref{e:LM}), when true, 
says that with probability tending to 1 as $n \to \infty$
the point set $\cX_n$ has the following property:
If we start with no edges between the vertices of $\cX_n$,
and then add edges one by one in order of increasing Euclidean length
until we arrive at a $k$-connected graph, then
just before the addition of the last edge we still have a 
 vertex of degree less than $k$; if $k=1$ then
just before the addition of the last edge we have exactly two
components, one of which is a singleton.

The largest $k$-nearest neighbour link 
$L_k(\cX_n)$
(or $L_k(\eta_n)$)
is of interest in its own right. To quote \cite{DH89}, it
`comes up in almost all discussions of computational complexity
involving nearest neighbours'. See
e.g. \cite{Pen97} for further motivation.
As with $M_{n,k}$, its limiting behaviour 
has previously been studied on the torus, and the unit cube,
and only at the level of a SLLN for regions with smooth or polytopal boundary.
By providing a limiting distribution for $L_k(\cX_n)$ for regions
with smooth or polygonal boundary,
we here add significantly to this body of work, too.

\subsection{Literature review}
\label{ss:lit}

Before stating our main results, let us give a literature review on this topic.
Under our working assumption (WA), we use throughout the notation 
\begin{align}
 f_0 := \inf_{x \in A} f(x); ~~~~~  f_1 : = \inf_{x \in  \partial A} f(x);
 ~~~~~
\fmax := \sup_{x \in A} f(x).
	\label{e:f0f1fmax}
\end{align}
Note that $0 < f_0 \leq f_1 \leq \fmax < \infty$ under our WA.
Let  $\theta_d$
denote the volume of a $d$-dimensional unit ball. 
 i.e. $ \theta_d := \pi^{d/2}/\Gamma (1+d/2)$,

In the case where 
$A$ is the flat torus $\mathbb{T}^d$ of any dimension, it is known
\cite[Theorem 13.6]{Pen} that under the natural coupling of $(\cX_n, n \geq 1)$
we have
\begin{align*}
	\lim_{n \to \infty} \left( 
	\frac{\theta_d n (M_k(\cX_n))^d}{\log n} \right) =
	\frac{1}{f_0} \quad a.s.
\end{align*}

The dimensionality and the density play a crucial role when one considers compact sets with boundaries. More precisely, if 
$A$ has a smooth boundary, it is proved
for $k=1$ in  \cite{Pen99c,Pen99b}, and for general $k$ in \cite{Pen}, that   
\begin{align}
	\label{e:SLLN_smooth}
	\lim_{n \to \infty} \left( \frac{\theta_d n (M_k(\cX_n))^d}{\log n} \right)  
	= \lim_{n \to \infty} \left( \frac{\theta_d n (L_k(\cX_n))^d}{\log n}
	\right)  
	=
	\max\Big(\frac{1}{f_0}, \frac{2-2/d}{f_1} \Big) \quad a.s.
\end{align}
When
$A$ is a convex polytope, it is proved in \cite{PYpoly} that
 \begin{align}\label{e:SLLN_cube}
	 \lim_{n \to \infty} \left(
	 \frac{ n (M_k(\cX_n))^d}{\log n} \right)
	 = \lim_{n \to \infty} \left(
	 \frac{ n (L_k(\cX_n))^d}{\log n} \right)
	 = \max_{\varphi \in \Phi^*(A) } 
	 \Big(\frac{ D(\varphi)}{f_\varphi \rho_\varphi d}  \Big) 
	 \quad a.s.
\end{align}
where $\Phi^*$ denotes the collection of all faces of $\varphi$ of all
dimensions (including $A$ itself, considered as a face of dimension 
$d$), and where $D(\varphi)$ is the dimension of face $\varphi$, 
and where $f_\varphi$ denotes the infimum of $f$ over face $\varphi$
and $\rho_\varphi$ is the angular volume of face $\varphi$.

Less is known about the fluctuations of $nM_{k}(\cX_n)^d - a_n$.
Weak limit results of the type \eqref{e:2nd_order} are proved for two cases in \cite{Pen97, Pen99a}. 
The first case is when $f$ is uniform on $\mathbb{T}^d$ for any $d \geq 2$.
In this case, by \cite[Corollary 13.20]{Pen}, one has 
\begin{align*}
	\theta_d n (M_k(\cX_n))^d - \log n
	- (k-1) \log \log n+ \log ((k-1)!)
	\overset{d}\longrightarrow  \Gum,
\end{align*}
where $\Gum$ denotes a standard Gumbel random variable,
i.e.  one with cumulative distribution function 
$F(x)= \exp(-e^{-x}), x \in \R$.
(For $a \in \R, b >0$ the random variable
$b \Gum + a$ is said to be Gumbel distributed
with scale parameter $b$ and location parameter $a$.)

The second case is when $f$ is uniform on $[0,1]^d$.
For this case,
we describe only the results for $k=1$ from \cite{Pen}
but the
case of general $k$ is also treated there.
When $f$ is uniform on $[0,1]^d$, one has
by \cite[Corollary 13.21]{Pen} that
\begin{align}
	2^{2-d} \theta_d n (M_1(\cX_n))^d -
	(2/d)\log n +
	(d -3 + 2/d) \log \log n
	\nonumber
	\\
	+ \log \left(
	\Big(
	\frac{2^{2-2/d}}{d(d-1)} \Big) (\theta_d d)^{3-d-2/d} 
	\theta_{d-1}^{d-2}   \right)
	\overset{d}\longrightarrow  \Gum.
	\label{e:cubelim}
\end{align}
 Similar results hold
for $L_k(\cX_n)$; see \cite[Theorems 8.3 and 8.4]{Pen}.
Moreover it is known that (\ref{e:LM}) holds.
Also these results hold
for $\eta_n$ instead of $\cX_n$. \\

So far as we know, there is no weak limit result for other shaped boundaries
or for non-uniform distributions
(until now). In the special case where
$d=2$ and $f$ is uniformly distributed in a  disk,
\cite{GK99} gives a partial result in the direction of a weak limit.

The main results of this paper considerably generalise previous findings and deepen the understanding of the connectivity threshold in terms
of the geometry of $A$. In the uniform case, we also provide a bound
on the rate of
convergence that is new for all shapes under the WA. \\

We end this section by mentioning some related results. It is natural to ask
what happens if we drop the working assumption and take the support of $f$ to
be unbounded. Penrose \cite{Pen98} found that the scaling is completely
different in the case of standard Gaussian density in $\RR^d$; see also Hsing and Rootz\'en \cite{HR05} for a significant extension in dimension two, where a class of elliptically contoured distributions are included, e.g.  Gaussian densities with correlated coordinates. 
Gupta and Iyer \cite{GI10} give a limiting distribution and SLLN
for $L_{n,1}$ for a class of radially
symmetric densities with unbounded support, including cases where
$L_{n,1}$ (and hence also $M_{n,1}$) does not even tend to zero.

\subsection{Main results}
\label{ss:mainres}

Throughout this paper, $c$ and $c'$ denote positive finite
constants whose values may vary from line to line and do not depend
on $n$.
Also if $n_0 \in (0,\infty)$ and 
$f(n),g(n)$ are two functions, defined for all $n \geq n_0$ 
with $g(n) >0$ for all $n \geq n_0$,
the notation $f(n)= O(g(n))$ as $n \to \infty$
means that $\limsup_{n \to \infty} (|f(n)|/g(n)) < \infty$.
If also $f(n) >0$ for all $n \geq n_0$, we use notation
$f(n)= \Theta(g(n))$ to mean that both $f(n)=O(g(n))$
and $g(n)= O(f(n))$.

Given $d,k \in \NN$, define the constant
\bea
c_{d,k} := \theta_{d-1}^{-1}  \theta_d^{1-1/d} (2- 2/d)^{k-2 + 1/d} 2^{1-k}
/(k-1)!
\label{e:defcd}
\eea
In this paper, given $A \subset \R^d$, 
we say that $A$ has {\em $C^2$ boundary}
(or for short: $\partial A \in C^2$)
if for each $x\in\partial A$,
the topological boundary of $A$,
there exists a rigid motion $\rho_x$ of
$\R^d$, an open set $U_x\subset \RR^{d-1}$ and a $C^2$ function
$f_x :\RR^{d-1}\to\RR$ such that $\rho_x(A \cap U) = \rho_x(U)
\cap \mathrm{epi}(f_x) $,
where $\mathrm{epi}(f_x):= \{(u,z) \in \R^{d-1} \times \R: z \ge f_x(u)\}$,
the closed epigraph of $f_x$.

For compact $A \subset \R^d$ with $C^2$ or polytopal boundary,
let $|A|$ denote the volume (Lebesgue measure) of $A$, and
$|\partial A|$ the perimeter of $A$, i.e. the $(d-1)$-dimensional
Hausdorff measure of $\partial A$.
Also define
\bea
\sigma_A := \frac{|\partial A|}{|A|^{1-1/d}}.
\label{e:defsigA}
\eea
%todo: decide whether change sigma_A for something else. there are sigma's down the line for a possible shape of discretisation of small clusters. maybe ok.
Note that $\sigma_A^{d}$ is sometimes called the
{\em isoperimetric ratio} of $A$, and is at least $d^d \theta_d$
by the isoperimetric inequality.

%We write $\overline{A^o}$ for the closure of the interior of $A$.
%We shall always assume that $\overline{A^o}= A$.

%TODO: check again whether convexity is necessary in 2d polygonal case
\begin{theorem}[Weak limits in the uniform case] \label{t:smooth}
	Suppose  either that $d \geq 2 $ and $A$ a compact subset of
	$\RR^d$ with $C^2$ boundary,
	%and with $\overline{A^o} = A$,
	or that $d=2$ and $A$ is
	a convex polygon.  Let  $f \equiv f_0 \1_A$ with $f_0=|A|^{-1}$.
	Let $\beta \in \R$.  Then, if $d =2$, we have as $n \to \infty $ that
	\begin{align}
		\PP[ nf_0 \pi M_{1}(\cX_n)^2
		-\log n \le  \be ] = \exp \Big( -   
		\frac{\sigma_A \pi^{1/2} e^{-\beta/2}}{2 (\log n)^{1/2}} \Big)
		e^{-e^{-\beta}}
		+O( (\log n)^{-1} );
		\label{e:MXweak}
		\\
		\PP[ nf_0 \pi M_{n,1}^2 -\log n \le  \be ] = \exp \Big( -   
		\frac{\sigma_A \pi^{1/2} e^{-\beta/2}}{2 (\log n)^{1/2}} \Big)
		e^{-e^{-\beta}}
		+O( (\log n)^{-1} ).
		\label{e:Mnweak}
\end{align} 
	Also, given $k \in \NN$,  set
	\begin{align*}
		u_{n,k} & := \PP[ n \theta_d f_0 M_{n,k}^d - (2-2/d) \log n +
		(4-2k-2/d) \log \log n \leq \beta];
		\\
	u'_{n,k} & := \PP[ n \theta_d f_0 M_{k}(\cX_n)^d - (2-2/d) \log n +
		(4-2k-2/d) \log \log n \leq \beta].
	\end{align*}
	If $d= 2$ then as $n \to \infty$
	\begin{align}
		u_{n,2} = \exp \Big(- \frac{\pi^{1/2}\sigma_A e^{-\beta/2}
		\log \log n}{8 \log n}  - \frac{e^{-\beta} \log \log n}{\log n}
		\Big)
		\exp \Big(- e^{-\beta}- 
		\frac{\pi^{1/2} \sigma_A e^{-\beta/2}}{4} \Big)
		\nonumber \\
		+ O\big(\frac{1}{\log n} \big),
		~~~~~~~~~
		\label{e:limu22}
	\end{align}
	and likewise for $u'_{n,2}$.
	If $d \geq 3$, or if $d=2, k \geq 3$
we have as $n \to \infty$ that
\begin{align}
	\label{e:limudhi}
	u_{n,k} = \exp \Big( - \frac{c_{d,k} \sigma_A e^{-\beta/2}
	(k-2+1/d)^2 \log \log n}{(1-1/d)\log n} 
	\Big) \exp(- c_{d,k} \sigma_A e^{-\beta/2} ) +
	O\big( \frac{1}{\log n} \big), 
\end{align}
	and likewise for $u'_{n,k}$. 
	Also \eqref{e:LM} holds, and all of 
	 the above results 
	hold with $L_{n,k}$ (resp.  $L_{k}(\cX_n)$)
	instead of $M_{n,k}$ (resp.  $M_{k}(\cX_n)$).
\end{theorem}
%	todo:{\bf (Add to this:   Non-convex polygons?)}

\begin{remark}
\begin{enumerate}[i)]
	\item The statements about $L_{n,k} $ and $ L_{k}(\cX_n)$ in this theorem 
	are spelt out  in Corollaries
	\ref{c:nnlink2d} and \ref{c:nnlink3d+}. 
	\item These results imply certain convergence in distribution results.
	Namely, $n \theta_d f_0 M_{n,k}^d$, suitably centred,
	is asymptotically Gumbel distributed with scale parameter $1$ when $d=2, k=1$
	but with scale parameter $2$ when $d \geq 3$ or $d=2,k \geq 3$.
	When $d=2,k=2$ the limiting distribution is a so-called
	{\em two-component extreme value distribution} (TCEV), that
	is, a probability distribution with cumulative distribution
	function (cdf) given by the product
	of two Gumbel cdfs with different scale parameters,
	in this case 1 and 2.
	The terminology TCEV was introduced in the hydrology literature
	\cite{Rossi}.
	\item 
	We have included a multiplicative correction factor
	in each of \eqref{e:MXweak}--\eqref{e:limudhi},
	namely the first factor in the right hand side in
	each case,
	because this factor tends to 1 very slowly, especially
	in \eqref{e:MXweak} and 
	\eqref{e:Mnweak} where the correction factor
	is $1+ O((\log n)^{-1/2})$ so for moderately large
	values of $n$ the limiting
	Gumbel distribution with cdf $e^{-e^{-\beta}}$ is
	not very close to the centred distribution of
	$n f_0 \theta_dM_{n,1}^d$.
	If $ d \geq 3$ it is possible to give some extra terms
	in the correction and improve the error bound to
	$O((\log n)^{\eps -2})$. 
%	todo:{\bf [More details somewhere?]} 
	\item The error bounds above are for fixed $\beta$ but we would
	need to make them uniform in $\beta$ for an error bound in the
	Kolmogorov distance  between probability distributions.
	We do
	not address this in this paper. 
\item It seems likely that our results carry over to the case where
	$d=2$ and $A$ is a non-convex polygon. Our main
		reason for restricting attention to polygons that are convex its
		that some of our arguments in Section \ref{s:RelateLM}
		are based on results from \cite{PYpoly} that
		are stated there only for convex polytopes.
\end{enumerate}		
\end{remark}

We now give a result for the general non-uniform case; that
is, we still use our WA on $f$, but drop the stronger
assumption that $f$ is constant on $A$. 
Recall $f_0, f_1$ defined at \eqref{e:f0f1fmax}.
In this more general
case, subject to the condition 
$f_1 \neq f_0(2-2/d)$,
we still provide a result along the lines of
(\ref{e:2nd_order}), but now, instead of using the explicit centring constants
$a_n = (2-2/d)\log n - (4-2k-2/d) {\bf 1}\{d \geq 3~ {\rm or} ~ k 
\geq 2\} \log \log n$
as in Theorem \ref{t:smooth},
we take $a_n$ to be the median of the distribution $n M_n^d$. 
In the case $f_1= f_0(2-2/d)$ we prove only the weaker result
that  our sequence of centred random variables is tight.

Given a random variable $X$, 
let $\mu(X):= \inf\{x \in \R:\PP[X \leq x] \geq 1/2\} $,  the median 
of  the distribution of $X$. 
%Let $\Gum$ be a standard Gumbel random variable.
Note that $\mu(\Gum) = -\log(\log 2)$,
so for $\alpha >0$, the random variable
$\alpha (\Gum+\log (\log 2))$ has a Gumbel distribution
with median 0 and with scale parameter $\alpha$.

\begin{theorem}[Weak limit in the non-uniform case] 
	\label{t:nonunif}
	Suppose our working assumption applies,   either with
	$d \geq 2 $ and $A$ a compact subset of
	$\RR^d$ with $C^2$ boundary,
	%and with $\overline{A^o} =A$,
	or with $d=2$ and $A$ 
	a convex  polygon. Let $k \in \NN$.

	(i)
	If $f_1 > f_0(2-2/d)$, then  as $n \to \infty$,
\begin{align}
	n M_{k}(\cX_n)^d - n \mu(M_k(\cX_n))^d \tod (\theta_d f_0)^{-1}
	(\Gum+ \log \log 2);
	\label{e:limnonu2}
	\\
	n L_{k}(\cX_n)^d - n \mu(L_k(\cX_n))^d \tod (\theta_d f_0)^{-1}
	(\Gum+ \log \log 2);
	\label{e:limnonu4}
	\\
	n M_{n,k}^d - n \mu(M_{n,k})^d \tod (\theta_d f_0)^{-1}
	(\Gum+ \log \log 2);
	\label{e:limnonu1}
	\\
	n L_{n,k}^d - n \mu(L_{n,k})^d \tod (\theta_d f_0)^{-1}
	(\Gum+ \log\log 2).
	\label{e:limnonu3}
\end{align}

	(ii)
	If $f_1 < f_0(2-2/d)$, then  as $n \to \infty$,
\begin{align}
	n M_{k}(\cX_n)^d - n \mu(M_{k}(\cX_n))^d \tod (2/(\theta_d f_1))
	(\Gum+ \log \log 2);
	\label{e:limnonu6}
	\\
	n L_{k}(\cX_n)^d - n \mu(L_{k}(\cX_n))^d \tod (2/(\theta_d f_1))
	(\Gum+ \log \log 2).
	\label{e:limnonu8}
	\\
	n M_{n,k}^d - n \mu(M_{n,k}^d) \tod (2/(\theta_d f_1))
	(\Gum+ \log \log 2);
	\label{e:limnonu5}
	\\
	n L_{n,k}^d - n \mu(L_{n,k})^d \tod (2/(\theta_d f_1))
	(\Gum+ \log \log 2);
	\label{e:limnonu7}
\end{align}
	(iii) In all cases, 
including when $f_1 = f_0(2-2/d)$, \eqref{e:LM} holds, and also
	the family of random variables 
	$(n (M_{n,k}^d- \mu(M_{n,k})^d))_{n \geq 1}$ is tight.
	%todo:{\bf [$n$ running through only integers?]}.
	Likewise the collection of random variables
	$(n (L_{n,k}^d- \mu(L_{n,k})^d))_{n \geq 1}$ is tight,
	as are the sequences 
	$(n (M_{k}(\cX_n)^d- \mu(M_{k}(\cX_n))^d))_{n \geq 1}$ and
	$(n (L_{k}(\cX_n)^d- \mu(L_{k}(\cX_n))^d))_{n \geq 1}$.

	%todo:{\bf  Non-convex polygons? }
\end{theorem}

Before proceeding to proofs, we give a rough calculation indicating
why, in the uniform case with $f \equiv f_0{\bf 1}_A$, we might expect to see
qualitative differences between the cases with $d=2$ and 
$k =1 $ or $k=2$, and other cases, as seen in Theorem \ref{t:smooth}.
Suppose we take a sequence of distance parameters
$r_n$ with 
$n f_0 \theta_d r_n^d = \log n + (k-1) \log \log n +c $
for some constant $c$.
Given $r_n$,  let $F^o_{n}$, $F_{n}^\partial$  be
the number  of vertices of degree less than $k$ in the interior of
$A$, respectively near the boundary of $A$. We give
a rough calculation
suggesting that for $d\geq 3$ we have
$\EE [ F_{n}^\partial] \gg \EE [ F_{n}^o]$ so the boundary region dominates,
while for $d=2$, it depends on the value of $k$ whether the interior or
boundary region dominates.
Firstly,
$$
\EE [ F_{n}^o] \approx n ((n f_0 \theta_d r_n^d)^{k-1}/(k-1)!)  
e^{- n f_0 \theta_d r_n^d}
\sim (e^{-c})/(k-1)!.
$$
For $F_{n}^\partial$, note that for small positive
$s$ the volume of the intersection of $A$ with
a ball of radius $r_n$ centred 
at distance $sr_n$ from $\partial A$ is about $(\theta_d/2) r_n^d+
\theta_{d-1} sr_n^d$, suggesting
\begin{align*}
	\EE [ F_{n}^\partial ] &
	\approx nr_n (f_0 \theta_d nr_n^d/2)^{k-1} (|\partial A|/(k-1)!)
	e^{-n f_0 \theta_d r_n^d/2}
	\int_0^\infty e^{-n f_0 \theta_{d-1} s r_n^d} ds
	\\
	& \approx {\rm const.} \times 
	n^{(1/2)-(1/d)} (n r_n^d)^{(1/d)  + k-2} (\log n)^{(1-k)/2}
	\\
	& \approx {\rm const.} \times 
	n^{(1/2)-(1/d)} (\log n)^{(1/d) + ((k-3)/2)}.
\end{align*}
If $d \geq 3$ this tends to infinity (regardless of $k$). 
Thus the boundary effects dominate in this case; we should
choose a slightly bigger $r_n$ to make $\EE[F_{n}^\partial]$
tend to a constant, and then $\EE[F_{n}^o]$ will tend to zero.

If $d=2$, the last expression for $\EE[F_{n}^\partial]$ tends to zero if 
$k=1$ and to infinity if $k \geq 3$, so the interior contribution
dominates when $k=1$ but the boundary contribution dominates when $k \geq 3$.
When $k=2$ the interior and boundary effects are of comparable size.

Having chosen $r_n$ so that $\EE[F_{n}^o+ F_{n}^\partial]$
tends to a constant,  we shall use Poisson approximation to
show that $\PP[L_{n}^o \leq r_n]$ tends to a non-trivial constant,
and then some percolation arguments to show the same limit
holds for $\PP[M_{n,1} \leq r_n]$. \\

The rest of the paper is organised as follows. After the preparation of geometrical ingredients in Section 2, we prove Poisson approximation for the number of $k$-isolated vertices in Section 3, asymptotic equivalence of $L_{n,k}$ and $M_{n,k}$ in Section 4, the weak law in the nonuniform case (Theorem \ref{t:nonunif}) in Section 5 and finally the weak law in the uniform case (Theorem \ref{t:smooth}) in Section 6. 

\section{Geometrical preliminaries}

In this section, we prepare some geometrical ingredients for later use.
Let $A$ be a compact subset of $\RR^d$ with $d \geq 2$.

 Given $B,C \subset \R^d$, set $B \oplus C := \{x+y: x \in B,y \in C\}$.
 Let $o$ denote the origin in $\R^d$.
Given $x \in \R^d$, we write $B+x$ for  $B \oplus \{x\}$.

 Given $s >0$, and $\Gamma \subset A$,
we write $\Gamma^{(s)}$ for $(\Gamma \oplus B_s(o))
\cap A$, the set of points in $A$ distant at most $s$ from $\Gamma$.
%todo:({\bf or just for $\Gamma \oplus B_x(o)$? This is different notation
%from \cite{Pen22} but maybe that's ok.})
Also we set $\diam(\Gamma):= \sup_{x,y \in \Gamma} \|y-x\|$,
or zero if $\Gamma= \emptyset$.

We write $A^{(-s)}$ for  $ A \setminus (\partial A)^{(s)}$,
the set of points  in $A$ distant more than $s$ from the
boundary $\partial A$ of $A$.

When $A$ is polygonal, 
we denote by $\Cor$ the set of corners of $A$. 

\begin{lemma}
	\label{lemgeom3}
	Suppose $A$ has a
	$C^2$ boundary.
	Let $\eps \in (0,1]$. Then:
	
	(i)
	For all small enough $r >0$ we have
	\begin{align}
		|B_r(x) \cap A| \geq ((\theta_d/2)+ (\theta_d \eps/4))r^d, 
	~~~~~
	\forall~  x \in A^{(-\eps r)}.
		\label{e:fornonun}
	\end{align}
	 
(ii)	 There exists $\delta > 0$ and $r_0 > 0$  such that 
	 if $0 < r < s < 2r < r_0$, then
	\begin{align}
	|A \cap B_s(x) \setminus B_r(x)| \leq ( (\theta_d/2)+ \eps)(s^d -r^d),
	 ~~~~~ \forall x \in (\partial A)^{(\delta s)}.
		\label{e:annu_bdy}
	\end{align}

\end{lemma}
\begin{proof}
	Clearly \eqref{e:fornonun} holds for $x \in A^{(-r)}$.

	Let $x \in (\partial A)^{(r)} \cap A^{(-K_0 r^2)}$ with
	$K_0$ to be chosen later. Without loss of
	generality (after a translation and rotation),
	we can assume that the closest point of $\partial A$ to $x$  
	lies at the origin, and $x = he_{d}$ for some $h > 0 $
	(where $e_d$ is the $d$th coordinate vector),
	and for some convex open $V \subset \R^d$ with $o \in V$
	and some open convex neighbourhood $U$ of the
	origin in $\R^{d-1}$, 
%	
	%Assume
	%$\partial A$ is the graph of a function $\phi:U\to \R$
	%with $U$ an open convex neighbourhood of the origin
	% in $\R^{d-1}$ (locally this
	%is true after a rotation and translation).
%we say that $A$ has {\em smooth boundary}
%if for each $x\in\partial A$, there exists a rigid motion $\rho_x$ of
%$\R^d$, an open set $U_x\subset \RR^{d-1}$ 
	and some $C^2$ function
	$\phi :\RR^{d-1}\to\RR$ we have that $A \cap V
	= V \cap \mathrm{epi}(\phi) $.
%where $\mathrm{epi}(f_x):= \{(u,z) \in \R^{d-1} \times \R: z \ge f_x(u)\}$,
%the closed epigraph of $f_x$.

	Since $z=o$ is the closest point in $\partial A$ to $x$,
	we must have $\nabla \phi(o) =o$. By a compactness argument,
	we can also assume
	$\sum_{i=1}^d \sum_{j=1}^d |\partial^2_{ij} \phi| \leq K/(9d^2)$
	on $U$ for some constant $K$ (depending on $A$).

	Now suppose $u \in \R^{d-1}$ with $\|u \| \leq 3r$
	(assume $r$ is small enough that all such $u$ lie
	in $U$). By the
	Mean Value theorem
	$\phi(u) =  u \cdot \nabla \phi(w)$ for some $w \in [o,u]$,
	and for $1 \leq i \leq d$, $\partial_i \phi(w) =
	w \cdot \nabla \partial_i \phi(v)$ for some
	$v \in [o,w]$. Hence 
	\bea
	|\phi(u)| \leq  (K/9)\|u\|^2 \leq
	K r^2, ~~~ \forall u \in \R^{d-1} ~{\rm with}~ \|u\| \leq 3r.
	\label{0127b}
	\eea

	For $a >0$ set $g(a):= |B_1(o) \cap (\R^{d-1} \times [0,a])|$.
	Then $g(a)/a$ is decreasing in $a$, so for $0 \leq a \leq 1$
	we have $g(a)/a \geq g(1) = \theta_d/2$.

	Let $\pi: \R^d \to \R^{d-1} $ denote projection onto
	the first $d-1$ coordinates and let $h:\R^d \to \R$ denote 
	projection onto the last coordinate ($h$ stands for
	`height').
	Take $K_0= 2K$.
	%Writing $h(x)$ for the $d$-th coordinate
	%(`height') of $x$, 
	Then $h(x) = d(x,\partial A) \geq K_0 r^2$.  Also $h(x) \leq r$.
	For $z \in B_r(x) \cap (\R^{d-1} \times
	[Kr^2,\infty))$ we have $\|\pi (z)\| \leq r$ so that
	by (\ref{0127b}) we have
	$|\phi(\pi(z)) | \leq K r^2 \leq h(z)$, and hence $z \in A$.
	Therefore
	\begin{align}
		|B_r(x) \cap A| & \geq |B_r(x) \cap (\R^{d-1}\times [Kr^2,\infty))|
	%\nonumber \\ 
	%	& \geq (\theta_d/2) r^d 
	\nonumber \\ 
		%& \geq (\theta_d/2) r^d +\eps_0 ( h(x) - Kr^2) r^{d-1}
		& \geq (\theta_d/2) r^d + r^d g((h(x) -K r^2)/r)
	\nonumber \\ 
		& \geq (\theta_d/2) r^d + (\theta_d/2) r^{d-1}
		(h(x) -K r^2)
	\nonumber \\ 
		& \geq (\theta_d/2) r^d + (\theta_d  h(x) /4) r^{d-1},
	\label{0127a}
	\end{align}
	%where the last line is because $Kr^2 = K_0r^2/2 \leq h(x)$.
	where the last line is because $Kr^2 = K_0r^2/2 \leq h(x)/2$.
	%and $\eps_0$ is a constant depending only on $d$.

	Let  $\eps >0$. Provided $r$ is small enough we 
	have $\eps r \geq K_0 r^2$,  so that
	 $ A^{(-\eps r)} \subset A^{(-K_0r^2)}$. Hence 
	for $x \in (\partial A)^{(r)} \cap A^{(-\eps r)}$ 
	we have \eqref{0127a}, and therefore since $h(x) \geq \eps r$,
	$$
	|B_r(x) \cap A|
	%\geq (\theta_d/2) r^d +\theta_d ( h(x) /4) r^{d-1}
		 \geq ((\theta_d/2) + (\theta_d \eps /4))r^d.
	$$
	%so (\ref{e:fornonun}) holds
	%on taking $\delta = \theta_d \eps/4$. This completes the proof of (i).
Thus we have (i).
	 \begin{figure}
\center
\includegraphics[width=0.49\textwidth, trim= 0 10 5 20]{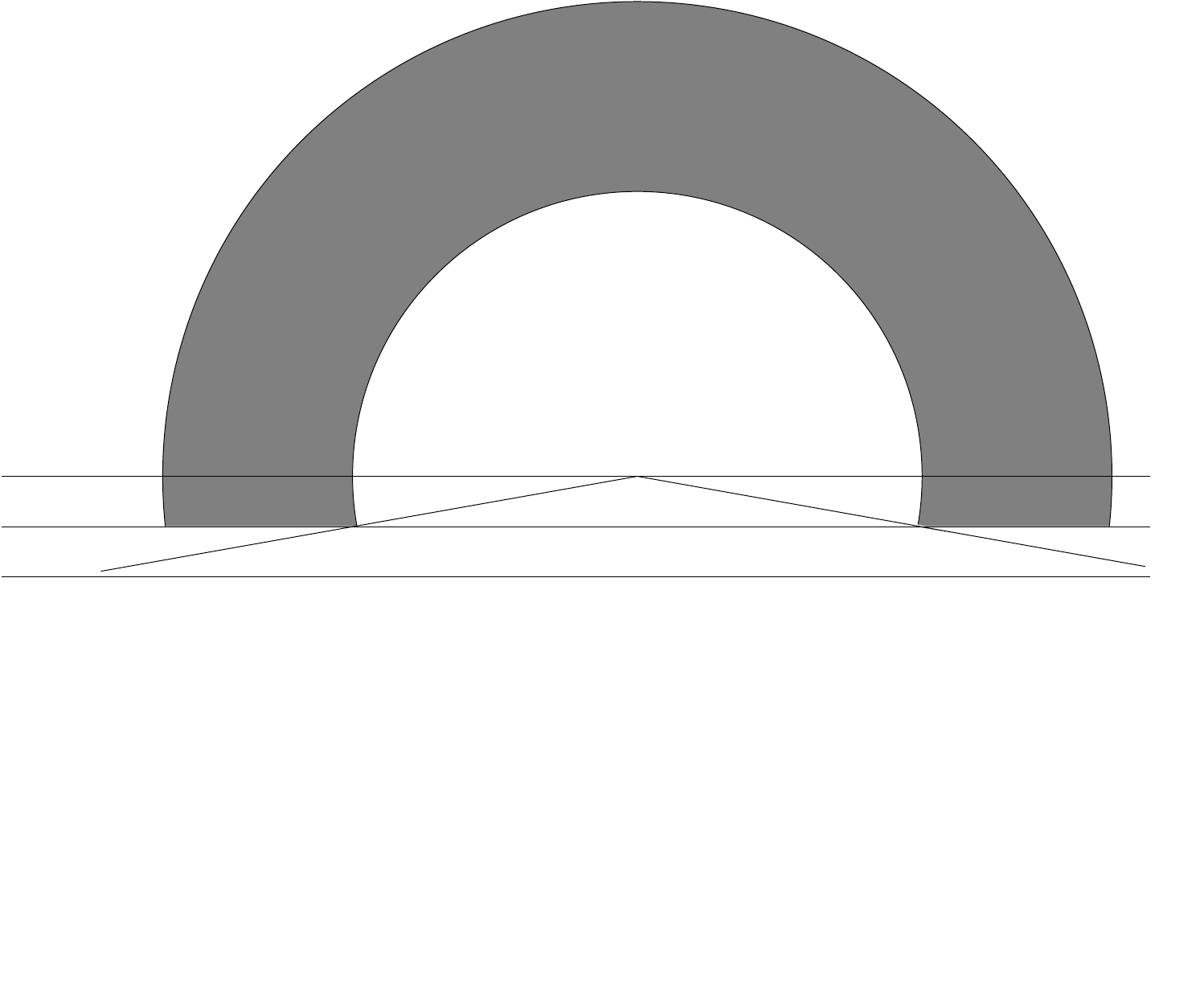}
\vskip -2.5cm
\caption{\label{partAnnu}
The horizontal lines are at height $0$, $-2 \delta s$ and $-4\delta s$.
The circles are of radius $r,s$ centred on the origin.
		 }
\end{figure}

For part (ii),
let $\delta >0$, to be chosen later.
Suppose that $s >0$ and
$x \in (\partial A)^{(\delta s)}$.
Let $r \in (s/2,s)$. Provided $s$ is small enough,
by  \eqref{0127b}  we have $\phi(u) \geq - \delta s$ for all $u \in \R^{d-1}$
with $\|u \| \leq 3s$. Therefore
 $$
 A \cap B_s(x) \setminus B_r(x) \subset (B_s(x) \setminus B_r(x)) \cap  
 (\R^{d-1} \times [ h(x) - 2 \delta s, \infty)  ) .
 $$
 Define the set 
 $S_{\delta}:= \{x \in B_1(o): \pi_d( \|x\|^{-1} x ) \geq-4 \delta\}.$
 Then since $s \leq 2r$, as shown in Figure \ref{partAnnu},
\begin{align*}
	|A \cap B_s(x) \setminus B_r(x)| \leq (s^d- r^d) |S_\delta| .
\end{align*}
%
% Therefore
%\begin{align}
% |A \cap B_s(x) \setminus B_r(x) | 
% \leq (\theta_d /2)(s^d - r^d) + 2 \delta  s \theta_{d-1} (s^{d-1} -r^{d-1}).
%	\label{e:slice_annu}
%\end{align}
% Since we are taking $r \geq s/2$, we have that 
% $$
% s^{d} -r^{d}  = (s-r) (s^{d-1} + s^{d-2} r + \cdots + r^{d-1} )
% \geq (s-r) d 2^{1-d}
%  s^{d-1}
% $$
% and similarly
% $$
% s^{d-1} -r^{d-1}  = (s-r) (s^{d-2} + s^{d-3} r + \cdots + r^{d-2} )
% \leq (s-r) d s^{d-2}.
% $$
% Therefore $s(s^{d-1}-r^{d-1}) \leq 2^{d-1} (s^{d} -r^d)$. Hence
% by \eqref{e:slice_annu},
% $$
% |A \cap B_s(x) \setminus B_r(x)| \geq (\theta_d /2)(s^d - r^d)
% + \delta  \theta_{d-1} 2^d (s^d -r^d),
% $$
On taking $\delta $ small enough so that $|S_\delta| \leq
\theta_d ((1/2) + \eps)$, 
% and on taking $\delta  = 2^{-d}\theta_{d-1}^{-1} \eps$,
 we obtain 
 (\ref{e:annu_bdy}), completing the proof of part (ii).
\end{proof}

Given $x \in A$,
set $\dist(x,\partial A):= \inf_{z \in \partial A}\{\|z-x\|\}$. 

\begin{lemma}
	\label{lemgeom1a} Suppose $A \subset \R^d$ is
	compact with $C^2$ boundary.
	%and with $\overline{A^o}=A$.

	(i) Given $\eps > 0$, there  exists 
	$r_0 = r_0(d,A, \eps)>0$  such that
	\bea
	|  B_r(x) \cap A| \geq ((\theta_d/2)-\eps)  r^d,
	~~~~ \forall x \in A, r \in (0,r_0).
	\label{e:0430a}
	\eea
	(ii) There is a constant $r_1 = r_1 (d,A)$, such that if
	$r \in (0,r_1)$ and $x, y \in A$
	with  $\|y-x \| \leq 2r$ and
	 $\dist(x,\partial A) \leq \dist(y,\partial A)$, then
	\bea
	|   A \cap B_r(y) \setminus B_r(x)| \geq  8^{-d}\theta_{d-1}
	r^{d-1} \|y-x\|.
	\label{e:diffball}
	\eea
	(iii)
	 Given $\eps > 0$, there  exists 
	$r_2 = r_2(d,A, \eps)>0$  such that
	\bean
	| A \cap  B_s(x) \setminus B_r(x) | \geq ((\theta_d/2)-\eps)  
	(s^d-r^d),
	~~~~ \forall x \in A, s \in (0,r_2), r \in (0,s).
	%\label{e:0430a}
	\eean
\end{lemma}
\begin{proof} 
	For (i) and (ii),
	see \cite[Lemma 3.2]{HPY23}. For (iii),
	as in 
	the proof of Lemma \ref{lemgeom3},
	it suffices to consider $x \in A$ such that
	the closest point of $\partial A$ to $x$ is at $o$
	with $x = he_d$ for some $h \in (0,s]$, and moreover
	$\partial A $ near $o$  is the graph of a $C^2$ function
	$\phi: U \to \R$ with $U$ an open neighbourhood of the
	origin in $\R^{d-1}$ and with
	$\sum_{i=1}^d \sum_{j=1}^d |\partial^2_{ij} \phi |\leq K/(9d^2)$
	on $U$ for some constant $K$ depending only on $A$.

	Then, as at \eqref{0127b},
	%in the proof of \cite[Lemma 3.2]{HPY23}, 
	provided
	$r_1$ is small enough we have
	$|\phi(u)| \leq (K/9) \|u\|^2$ for $\|u\| < r_1$.
	Therefore given $\delta >0$, if also $\|u \| <  \delta/K$
	we have $|\phi(u) | \leq \delta \|u\|$.
	%and therefore
	%with $h(\cdot)$ denoting projection onto the $d$th coordinate,
	%$$
	%\{(u,h): \|u \leq 9 \delta/K, \|u \| \leq h \leq \sqrt{r^2-u^2}
	%\}
	%\subset A
	%$$

	Choose $\delta $ as follows. Given $a>0$,  define the set
	$\Lambda_a \subset \R^{d}$ by
	$$
	\Lambda_a:= \{(u,z): u \in \R^{d-1}, z \in \R, z \geq a\|u\|, \|u\|^2
	+ z^2 \leq 1\}.
	$$
	Then $|\Lambda_a | \uparrow \theta_d/2$ as $a \downarrow 0$
	so we can and do choose $\delta >0$ 
	such that $|\Lambda_\delta| \geq (\theta_d/2) - \eps$.

	For $0 \leq r < s$, given $x$ as above define the set 
	$
	S:=  ( (s \Lambda_\delta) \setminus (r \Lambda_\delta)) + x.  
	$
	Then
	$$
	|S| = | (s \Lambda_\delta) \setminus (r \Lambda_\delta)|
	= (s^d -r^d)|\Lambda_\delta| \geq ((\theta_d/2)-\eps)(s^d -r^d).
	$$
	Now suppose also that $s <  \delta /K$.
	For $y = (v,h') \in S$,
	we have $\|v \| \leq s$ and 
	$$
	h' \geq (h'-h) \geq \delta \|v\| \geq \phi(v),
	$$
	so $y \in A$ and $S \subset A$. 
	Also $S \subset B_s(x) \setminus B_r(x)$.
	This gives the result, with $r_2= \delta/K$.
\end{proof}

Recall that $\Cor$ denotes the set of corners of $A$ when $A$ is polygonal.

\begin{lemma}\label{l:poly_diff_ball}
	Assume $d=2$ and $A$ is polygonal, then there exist $K>0$ and $r_1>0$ depending on $A$ such that for all $r\in(0,r_1)$, $x,y\in A\setminus \Cor^{(Kr)}$ with $\dist(x,\partial A)\le \dist(y,\partial A)$ and $\|x-y\|\le 3r$, the lower bound \eqref{e:diffball} holds.
\end{lemma}
\begin{proof} 
	
%	First we examine the situation where $x$ is not too close to the corners of $A$. 
	Let $r_1$ be small enough such that
	non-overlapping edges of $A$ are distant at least $8r_1$ from each other. 
	Consider $x \in A\setminus \Cor^{(Kr)}$ with $r< r_1$ where
	$K$ is made explicit later.
	We can assume that the corner of $A$ closest to $x$
	is formed by edges $e,e'$ meeting at the origin with angle 
	$\alpha\in (0,2\pi) \setminus \{\pi\}$. 
	We claim that, provided $K>4+ 8/|\sin \alpha|$,
	the disk $B(x,4r)$ intersects at most one of the two edges.
	Indeed, if it intersects both edges,
	then taking $w\in B(x,4r) \cap e, w'\in B(x,4r) \cap e'$
	we have $\|w-w'\|\le 8r$; hence $\dist(w,e')\le 8r $.
	Then, $\|w\| \leq \dist(w,e')/|\sin \alpha|\le 8r /|\sin \alpha|$. 
	However, $\|w\|\ge (K-4)r$ by the triangle inequality,
	so we arrive at a contradiction. 
	We have
	thus shown that any ball of radius $4r$ with centre
	distant at least $Kr$ from the corners of $A$ cannot intersect two edges at the same time, where 
	$K = 5 +(8/\min_i |\sin \alpha_i|)$ and $\{\alpha_i\}$ are
	the angles of  the  corners of $A$.  
	
	We have $B(x,r)\cup B(y,r)\subset B(x,4r)$; hence, the argument leading to Lemma \ref{lemgeom1a}-(ii), namely \cite[Lemma 3.2]{HPY23}, gives the estimate \eqref{e:diffball} in this case too. 
\end{proof}

\begin{lemma}
	\label{l:Med1}
	Let $\eps \in (0,1]$. Then
	for all $r >0 $ and
	all compact $B \subset \R^d$ with $ \diam B \geq \eps r$
	we have $| B \oplus B_r(o) | \geq |B| + \theta_d(1 + 2^{-d-1}d^{-d}
	\eps^d)
	r^d$. 
\end{lemma}
\begin{proof}
	By scaling, it suffices to show that for all compact $B \subset \R^d$
	with $\diam B  \geq \eps $, we have 
	$| B \oplus B_1(o) | \geq |B| + \theta_d (1+ 2^{-d-1}d^{-d}
	\eps^d)  $. 

	 Let $B \subset \R^d$
	 with $\eps \leq \diam B < \infty$.
	 Without loss of generality we can assume $\diam(\pi_1(B)) 
	 \geq \eps/d$, where $\pi_1$ denotes projection onto
	 the first coordinate.

	 Let $x$ be a  left-most point of $B$,
	 $y$ a right-most point of $B$ and $u$ a top-most point of $B$.
	 Here `left' and `right' refer to ordering using the first coordinate
	 and `top' refers to ordering using the last coordinate.
	 Let $H^+$ be the right half of $B_1(y)$ and $H^-$ the left-half
	 of $B_1(x)$. Let $D:= B_{\eps/(2d)}(u+ (0,\ldots,0,\eps/(2d)))$, and
	 let $D^+$ and $ D^-$ be the left half and right half of $D$,
	 respectively.
	 Then the interiors of $H^+$ and of $H^-$ are disjoint
	 from $B$ and from each other, and the interior of either
	 $D^+$  or $D^-$ (or both) is disjoint from all of 
	 $B, H^+$ and $H^-$. 
	 Therefore since $H^+$, $H^-$ and $D$ are
	 all contained in $B \oplus B_1(o)$, we obtain
	 that
	 $$
	 |B \oplus B_1(o) | \geq |B|+ \theta_d + (\theta_d 
	 2^{-d-1}d^{-d})\eps^d,
	 $$
	 as required.
\end{proof}

\begin{lemma}
	\label{l:Med2} 
	Suppose $\partial A \in C^2$.
	Let $\rho, \eps \in \R$ with $0 < \eps < \rho$.
	Then there exist
	$\delta = \delta(d,\rho,\eps) >0$,
	%and $\delta' = \delta'(d,\rho,\eps) >0$,
	and $r_0 = r_0 (d, \rho, \eps, A)$,
	such that for all $r\in (0,r_0)$ and
	all compact $B \subset A$ with $\eps r \leq \diam B \leq \rho r$
	we have 
	\begin{align}
		\label{e:MedLB}
		| (B \oplus B_r(o)) \cap A | \geq |B| +
%		todo: could use the notation B^{(r)} but maybe ok
		((\theta_d/2) + \delta) r^d,
	\end{align}
		and also, letting $x_0$ denote a closest point of $B$ 
	to $\partial A$, we have
	 \begin{align}
		\label{e:MedLB2}
		| (B \oplus B_r(o)) \cap A | \geq |B| +  |B_r(x_0)\cap A| + 
		2  \delta r^d. 
	\end{align}
	%}
\end{lemma}
%\xy{revised}
\begin{proof}
It suffices to prove \eqref{e:MedLB2}. Indeed, if 
	\eqref{e:MedLB2} holds for some $\delta$ and $r_0$, then
	%taking $\delta = \delta'/2$,
	using
	\eqref{e:MedLB2} and
	Lemma \ref{lemgeom1a}-(i) readily yields
	(\ref{e:MedLB}) for some new, possibly smaller, choice of $r_0$.

	Without loss of generality we may assume $\eps < 1 < \rho $.
	Let $r>0$, and let $B \subset A$ be compact with $\eps r
	\leq \diam (B) \leq \rho r$.
	If $B \subset A^{(-r)}$ we can use Lemma \ref{l:Med1}
	%the preceding lemma,
	so it suffices to consider the case where 
	$B \cap (\partial A)^{(r)} \ne\emptyset$.

	Let $x_0$ be a closest point of
	$B $ to $\partial A$.
	Without loss of
	generality (after a rotation and translation),
	we can assume that the closest point of $\partial A$ to $x_0$  
	lies at the origin, and $x_0 = he_{d}$ for some $h \in [0,r] $,
	and that within some neighbourhood of the origin,
	$A$ coincides with the closed epigraph of a
	function $\phi:U\to \R$
	with $U$ an open convex neighbourhood of the origin
	 in $\R^{d-1}$.
	 %(locally this is true after a rotation and translation).
	As in the proof of Lemma \ref{lemgeom3},
	we can find $K = K(d,A) \in (1,\infty)$ such that
	\bea
	|\phi(u)| \leq  (K/9)\|u\|^2 \leq
	K \rho^2 r^2, ~~~ \forall u \in \R^{d-1} ~{\rm with}~
	\|u\| \leq 2\rho r.
	\label{0127b3}
	\eea
	Assume $r \leq  \eps/(144 d K \rho^2)$. 
	Let $\pi: \R^d \to \R^{d-1} $ denote projection onto
	the first $d-1$ coordinates, and
	for $1 \leq i \leq d$,
	let $\pi_i : \R^d \to \R$ denote projection onto the
	$i$th coordinate.
	Define the set $H^-$ (slightly less than half a ball of radius $r$:
	see Figure \ref{Usteer}) by 
	\begin{align}
	H^-:= \{z\in B_r(x_0):  \pi_d(z) < \pi_d(x_0) - K\rho^2 r^2 \}. 
		\label{e:S0def}
	\end{align}
%`	Now we consider \eqref{e:MedLB2}. First suppose $\diam \pi_d(B)\ge \eps  r/d$ and let $x^+,x^-,y$ be as before in this case. 	
	%Since $x_0$ is a closest point of $B$ to $\partial A$,
	For all $w \in B$ we have $\| \pi(w)\| \leq \|w-x_0\| \leq \rho r$,
	so by \eqref{0127b3},
	\begin{align}
	\pi_d(x_0)=\dist(x_0,\partial A)\le \dist(w,\partial A)\le \pi_d(w) 
	+  |\phi(\pi(w))| \le \pi_d(w)+K\rho^2 r^2.
		\label{0112a}
	\end{align}
Therefore any $z\in H^-$, $w \in B$ satisfy
$\pi_d(z) <  \pi_d(w)$, so that  
$H^-\subset (B\oplus B_r(o)) \setminus B$. 

We can bound above the volume difference of a
half-ball and $H^-$ by the volume of a cylinder of thickness $K\rho^2 r^2$
with base of radius $r$. 
Using this and the union bound, we obtain that
\begin{align}
	|B_r(x_0)\cap A|\le (\theta_d/2) r^d
	+ |H^-\cap A|
	+ \theta_{d-1}r^{d-1} K\rho^2 r^2.
	\label{0115a}
\end{align}
	%(so $\pi_d$ is the same as $h$). 

	 For at least one $i \in \{1,\ldots,d\}$,
	we must have 
	 $\diam (\pi_i(B)) \geq \eps r/d$.
	 We distinguish the cases where this  holds for
	 $i=d$, and where it holds for some $i \leq d-1$.

	 %\begin{figure}[p]
	 \begin{figure}
\center
\includegraphics[width=0.49\textwidth, trim= 0 10 5 20]{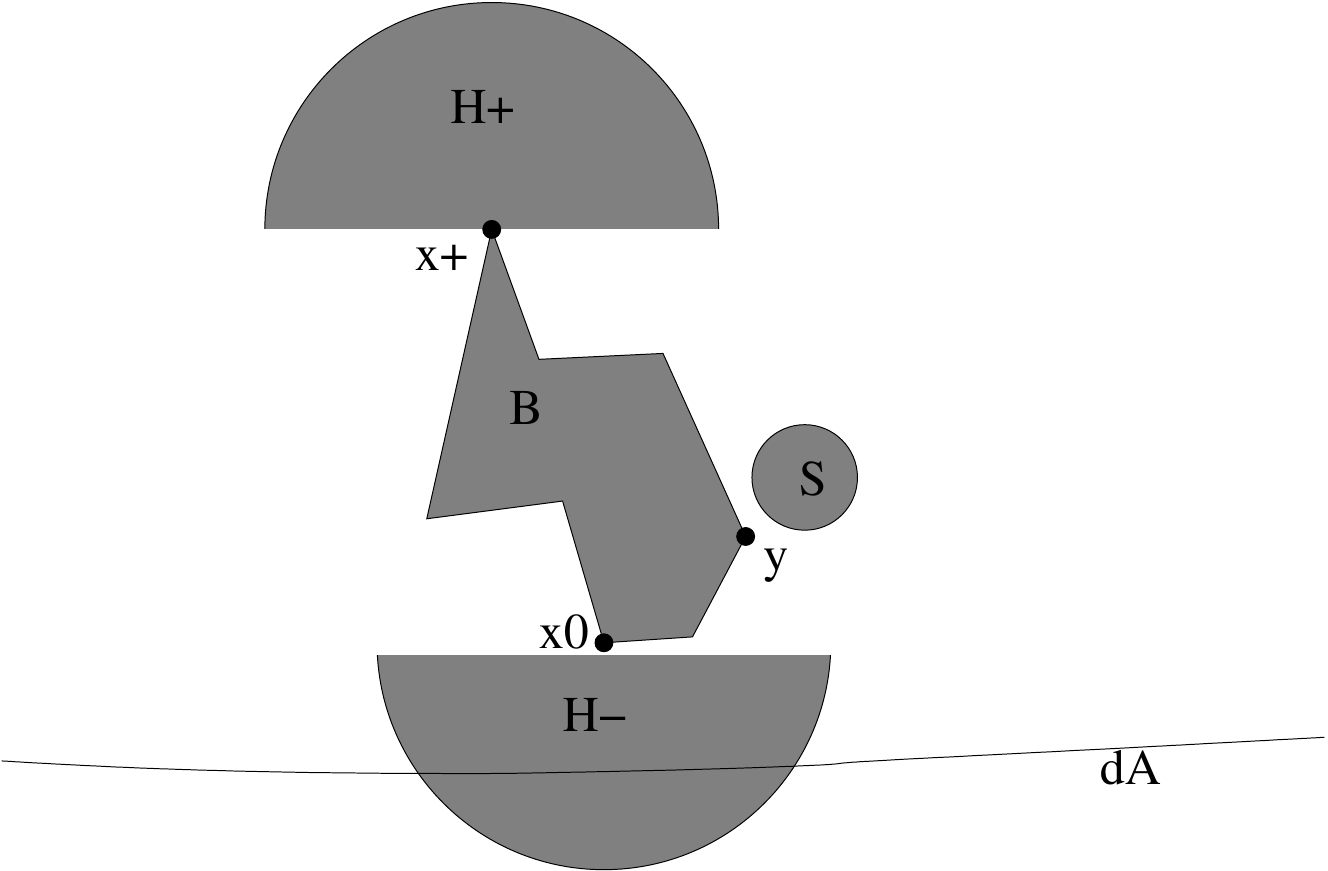}
 \includegraphics[width=0.49\textwidth, trim= 0 10 5 20, clip]{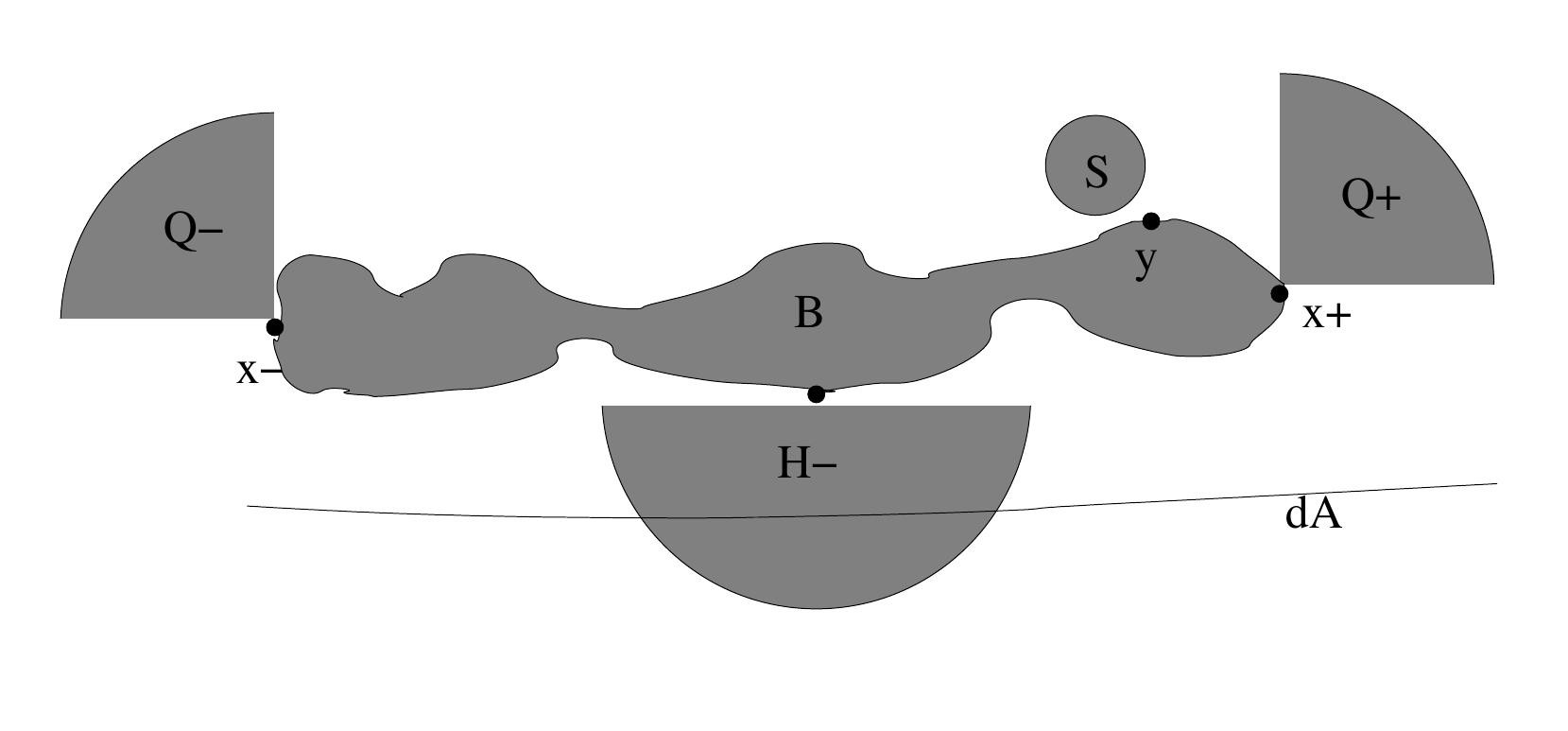}
\caption{\label{Usteer}
		 Illustration of the proof of Lemma \ref{l:Med2}.\\
		 Left: when $\diam (\pi_d(B) ) \geq \eps r/d$,
		 the sets $B, H^+, H^-, S$ are disjoint. 
		 The point
		 $x^-$ (not indicated) could be the same
		 as $x_0$; if not, it is only slightly lower
		 than $x_0$. 
		 \\
		 Right:
		 when $\diam (\pi_1(B) ) \geq \eps r/d$,
		 the sets
		 $B, H^-, Q^+, Q^-, S$
		 are  disjoint.
		 }
\end{figure}

 First suppose  $\diam (\pi_d(B)) \geq \eps r/d$.
Choose $x^+ \in B$ of maximal height (i.e., maximal $d$-coordinate),
 $x^- \in B$ of minimal height, and $y \in B$ of maximal $1$-coordinate
		 (see Figure \ref{Usteer} (Left)).

For all $z \in B \oplus B_r(o)$ we
	have $\|\pi(z)\| \leq \|z - x_0\| \leq 2 \rho r$ so
	by (\ref{0127b3}) we have $|\phi(\pi(z))| \leq K \rho^2r^2$. 
	Applying this in the case $z= x^-$, we deduce that
	\bean
	\pi_d(x^+) \ge \pi_d(x^-) + \eps  r/d \geq (\eps  r/d) - K \rho^2 r^2
	\geq (8/9)\eps  r/d,
	\eean
	and thus
	\bean
	\pi_d(x^+) - \phi(\pi(z)) \geq (8/9)\eps  r/d - K \rho^2 r^2
	\geq (7/9)\eps  r/d,
	~~~~
	%\label{0201a}
	\forall z \in B \oplus B_r(o).
	\eean
	Let $H^+ : = \{z \in B_r(x^+): \pi_d(z) > \pi_d(x^+) \}$
		 (see Figure \ref{Usteer}).
	Then $H^+  \cap B = \emptyset$, since $x^+$
	is a point of maximal height in $B$.
	Also $H^+ \subset B \oplus B_r(o)$, so
	%by (\ref{0201a}), for
	for all $z \in H^+ $ we have
	$$
	\phi(\pi(z)) \leq K \rho^2r^2 \leq
	(\eps/(9d))r
	\leq
	\pi_d(x^+) < \pi_d(z),
	$$
	so $z \in A$.
	Also, since $ \pi_d (z) > \pi_d(x^+) \geq \pi_d(x_0)$ we have
	from \eqref{e:S0def} that $z \notin H^-$. Hence
	 $H^+ \subset A \setminus H^-$.

	Now consider $y$. For $1 \leq i \leq d$ let
	$e_i$ denote the $i$th unit coordinate vector.
	Define a point $\ytil$ slightly to the right of
	$y$ by
	\begin{align*}
		\ytil := \begin{cases}
			y 
			+ (\eps r/(8d)) e_1
			- (\eps r/(8d)) e_d
			& {\rm if} ~
			\pi_d(y) \geq (\pi_d(x^+) + \pi_d(x^-))/2
			\\
			y + (\eps r/(8d)) e_1
			+ (\eps r/(8d)) e_d
			& {\rm if} ~
			\pi_d(y) < (\pi_d(x^+) + \pi_d(x^-))/2,
		\end{cases}
	\end{align*}
	and define the small ball $S:= B_{\eps r/(9d)}(\ytil).$
	Then $|S|= \delta_1 r^d$, where we set 
	$\delta_1 := \theta_d(\eps/(9d))^d$.
%
%	Define the half-ball
%	$$
%	H := \{z \in B_r(y): \pi_1(z) > \pi_1(y)\}.
%	$$
%	Then $H \cap B = \emptyset$. We aim to
%	find a horizontal slice
%	$S$ of $H$ of non-vanishing proportionate volume,
%	that is contained in $A$ and
%	disjoint from $B^+$.
	%If $\pi_d(y) \geq \eps  r/(2d)$, take 
%	If $\pi_d(y) \geq \eps  r/(2d)$, take 
%	\begin{align}\label{e:S_+}
%	S = \{z \in H: \pi_d(y) - (\eps  r/(4d)) \leq \pi_d(z) \leq \pi_d(y)\}.
%	\end{align}
	
	Suppose $\pi_d(y) \geq (\pi_d(x^+) + \pi_d(x^-))/2$ (as
	well as $\diam(\pi_d(B)) \geq \eps r/d$).

	Then for all $z \in S$ we have
	$\pi_d(z) \leq \pi_d(y) \leq \pi_d(x^+)$ so $z \notin H^+$.
%	Also $\pi_d(z) \geq \eps  r/(4d) \geq 2K \rho^2r^2$
%	so $\pi_d(z) \geq \phi(\pi(z))$ and $z \in A$.
Moreover
%provided $r$ is small enough
$$
\pi_d(z) \geq \pi_d(y) - \eps r/(4d) \geq \pi_d (x^-) +
\eps r/(4d) \geq \pi_d(x_0) + \eps r/(8d)
%= K \rho^2 r^d
$$
by \eqref{0112a}, applied to $w = x^-$.
Therefore $z \notin H^-$ by \eqref{e:S0def}, and also (by \eqref{0127b3})
$\pi_d(z) \geq \phi(z)$ so $z \in A $. Thus 
 $S \subset 	A \setminus ( H^+ \cup H^-)$ in this case.

	Now suppose $\pi_d(y) < (\pi_d(x^+) + \pi_d(x^-))/2$ (as
	well as $\diam(\pi_d(B)) \geq \eps r/d$).

	%If $\pi_d(y) < \eps  r/(2d)$, then instead of
	%\eqref{e:S_+} take 
	%\begin{align}
	%\label{e:S_-}
	%S := \{z \in H: 
	%\pi_d(y) + 2K \rho^2 r^2 \leq \pi_d(z) \leq \pi_d(y) + \eps  r/(4d)\}.
	%\end{align}
	Then for all $z \in S$ we have
	$\pi_d(z) \leq \pi_d(y) + \eps r/(4d)
	%\eps  r/(2d)+  \eps  r/(4d)  
	\leq \pi_d(x^+)$, so $z \notin H^+$. Also,
since $\pi_d(y)\ge  \phi(\pi(y)) \geq -  K \rho^2 r^2$,
%provided $r$ is small enough
we have
$$
\pi_d(z) \geq \pi_d(y) + \eps r/ (72 d) \geq  K \rho^2 r^2 
\geq \phi(\pi(z)),
$$
so $z \in A$,
and also by \eqref{0112a} applied to $w =y$, we
have $\pi_d(z) \geq \pi_d(x_0)$, so  $z \notin H^-$.
Therefore $S \subset A \setminus (H^+ \cup H^-)$ in this case too. 
Thus, whenever $\diam(\pi_d(B)) \geq \eps r/d$, we have
\begin{align}
| (B \oplus B_r(o)) \cap A| \geq |B| + |H^+| + |S| + |H^-\cap A|.
	\label{0115b}
\end{align}

%
%\begin{align}
%\label{e:2103}
%\end{align}
%
Combining \eqref{0115b} and \eqref{0115a},
 provided $r \leq \delta_1/(2K \rho^2 \theta_{d-1})$ we have  
	\begin{align}
	%|(B \oplus B_r(o)) \cap A| \geq |B| + |B_r^+| &
		%+|S| \geq |B| + \theta_d
	%r^d/2 + \delta_1 r^d,
		| (B \oplus B_r(o)) \cap A|  \ge |B| & + |B_r(x_0)\cap A| + (\delta_1/2) r^d, 
		%\nonumber \\
		& {\rm if}~ 
		%\pi_d(y) \geq \eps  r/(2d)~{\rm and}~
		\diam(\pi_d(B)) \geq \eps r/d.
	%\label{0201b}
\label{e:2103}
	\end{align}

	Now suppose $\diam \pi_i(B) \geq \eps  r/d$ for some 
	$i \in \{1,\ldots,d-1\}$. We shall consider here the
	case where this holds for $i=1$; the other cases may be treated
	similarly.

	Let $x^-, x^+, y$ be points in $B$ of minimal $1$-coordinate,
	maximal $1$-coordinate, and maximum height respectively.
	Let $\delta_2 := \delta_1/(2 \theta_{d-1})$.
	Define the sets $Q^-$ and $Q^+$ (slightly less than quarter-balls
	of radius $r$: see Figure \ref{Usteer} (Right)) by
	\bean
	Q^- : = \{z \in B_r(x^-): \pi_d(z) \geq \pi_d(x^-) + \delta_2 r,
	\pi_1(z) < \pi_1(x^-) \}; \\
	Q^+ : = \{z \in B_r(x^+): \pi_d(z) \geq \pi_d(x^+) + \delta_2 r,
	\pi_1(z) > \pi_1(x^+) \}.
	\eean
	By \eqref{0127b3}, for $z \in Q^-$ we have
	$|\phi(\pi(z))| \leq K\rho^2 r^2$,
	so 
	provided $r < \delta_2/(2K\rho^2)$,
	for all $z \in Q^-$ we have
	$$
	\pi_d(z) \geq \pi_d(x^-) + \delta_2 r \geq \delta_2 r - K \rho^2 r^2
	\geq K \rho^2 r^2 \geq \phi(\pi(z)),
	$$
	so that $z \in A$. Also by \eqref{0112a} applied to $w= x^-$ we have
	$\pi_d(z) \geq \pi_d(x^-) \geq \pi_d(x_0)-K \rho^2 r^2$,
	so $z \notin H^-$ by \eqref{e:S0def}.
	Thus
	$Q^- \subset A \setminus H^-$, and similarly $Q^+ \subset A \setminus
	H^-$. Also
	$(Q^- \cup Q^+) \cap B = \emptyset$, and
	$|Q^- \cup Q^+| \geq (\theta_d - 2 \delta_2 \theta_{d-1}) r^d/2
	= (\theta_d - \delta_1) r^d/2$.

	We define a point $\ytil$ slightly above $y$ by
	\begin{align*}
		\ytil := \begin{cases}
			y + (\eps r/(8d)) e_d 
			 + (\eps r/(8d)) e_1 &{\rm if~} 
			 \pi_1(y) \leq (\pi_1(x^-) + \pi_1(x^+) )/2
			\\
			y + (\eps r/(8d)) e_d 
			 - (\eps r/(8d)) e_1 &{\rm if~} 
			 \pi_1(y) > (\pi_1(x^-) + \pi_1(x^+) )/2,
		\end{cases}
	\end{align*}
	and set $S := B_{\eps r/(9d)}(\ytil)$: see Figure \ref{Usteer} (Right).
	Then $|S| = \delta_1 r^d$
	as before.

Then for all $z \in S$ we have
$\pi_d(z) > \pi_d(y) \geq \pi_d(x_0)$, so that $z \notin B$ and
$z \notin H^-$.
Also
	$\phi(\pi(z)) \leq K \rho^2 r^2 < \eps r/(72 d) \leq \pi_d(z)$, 
	so $z \in A$. Moreover,  
	if $\pi_1(y) > (\pi_1(x^-) + \pi_1(x^+))/2$
	then 
	$$
	\pi_1(z) - \pi_1(x^-)  = \pi_1(y) - \pi_1(x^-)
	+ (\pi_1(z) - \pi_1(y)) \geq
	(\eps r/(2d))
	- (\eps r/(4d))  >0 ,
	$$
	while 
	if $\pi_1(y) \leq (\pi_1(x^-) + \pi_1(x^+))/2$ then
	$\pi_1(z) - \pi_1(x^-) > \pi_1(y) - \pi_1(x^-) \geq 0$
	so in both cases  $z \notin Q^- $.
	Similarly  $z \notin Q^+$.
	%$\pi_1(x^+) - pi_1(z)  > 0$, so that
	%\cup Q^+$. 
	Thus $S \subset A \setminus(Q^+ \cup Q^-
	\cup B \cup H^-)$.
	%
	%For any $z\in H^-$, 
	%$\pi_d(z)\le \min \pi_d(B) \le \min( \min(\pi_d(Q^+)), 
	%\min(\pi_d(Q^-)), \min(\pi_d(S)))$. This shows
	%$H^-\cap Q^+ = H^-\cap Q^- = H^- \cap S=\emptyset$.
	Combining all
	of this and using \eqref{0115a} in the third line below yields
	\begin{align*}
		| (B \oplus B_r(o)) \cap A| &  \geq |B| +
	|Q^- \cup Q^+| + |S|  + |H^-\cap A| \\
		& \geq  |B| +  ((\theta_d + \delta_1)/2) r^d + |H^-\cap A| \\
		& \ge  |B| + |B_r(x_0)\cap A| - \theta_{d-1}r^{d-1} K\rho^2 r^2 + (\delta_1/2) r^d  \\ 
		& \ge  |B| + |B_r(x_0)\cap A| + (\delta_1/4) r^d
		~~~{\rm if}~ \diam (\pi_1(B)) \geq \eps r/d,
	\end{align*}
provided $r \leq \delta_1/(4K\rho^2\theta_{d-1})$. 
Combined with \eqref{e:2103}, 
this shows that \eqref{e:MedLB2} holds for $r$ small if we take
$\delta = \delta_1/8$.
\end{proof}

\section{Poisson approximation for the $k$-isolated vertices}

Fix $k \in \NN$.
We say a vertex is $k$-isolated if its degree is at most $k-1$.
Given $n,r>0$ let $\xxi_{n,r}$ denote
the number of $k$-isolated vertices in $G(\eta_n,r)$:
\begin{align}
	\label{e:def_isoVer}
	\xxi_{n,r} := \sum_{x\in \eta_n} \1\{\eta_n(B(x,r)) \le k\}.
\end{align}
The
goal of this section is to prove (in Proposition \ref{p:poisson2d} below) 
Poisson approximation for $\xxi_{n,r}$ when $n$ is large and $r$ is small.

Throughout this section we adopt our working assumption on $\nu$. Moreover
we assume   either that $d \geq 2$ and
the support $A$ of $\nu$  
is compact with $C^2$
%todo: need A=\bar{A^o}?
boundary, or  that $d=2$ and $A$ is
a polygon. 
%todo:{\bf (hope to drop convexity eventually)}.
We do {\em not} assume in this section that $\nu$ is necessarily 
uniform on $A$.
Recall that $\eta_n$ is the Poisson process in $\RR^d$ with intensity
measure $n \nu$. 

A fundamental identity used throughout this paper is the Mecke
 equation which basically says that the law of a Poisson
 process $\eta$ conditioned on having a point mass at $x$ is
 that of $\eta \cup \{x\}$.
 More precisely, let $\eta$ be a Poisson process on $\RR^d$ with
 diffuse intensity measure $\la$
 (that is, $\la$ does not charge atoms).
 The Mecke equation says that
\begin{align}\label{e:mecke}
	\EE\Big[ \sum_{x \in \eta} f(x,\eta) \Big] = 
	\int
	\EE\left[ 
	f(x,\eta \cup \{x\}) \right]\la(dx),
\end{align} 
for all $f: \RR^d\times \bN(\RR^d)\to \RR$ such that both sides of the identity are finite, 
where $\bN(\RR^d)$ denotes the space of all locally finite
subsets of 
$\RR^d$ - see \cite[Chapter 4]{LP18} for a more general statement. 

By the Mecke equation, given $n, r >0$ we have
\begin{align}
	\EE[\xxi_{n,r}] = n\int_A p_{n,r}(x) \nu(dx),
	\label{e:EEF}
\end{align}
where for each $x \in A$ we set
\begin{align}
	\label{e:pnj}
	p_{n,r}(x) : = \PP[\eta_n(B(x,r))\le k-1]
	= \sum_{j=0}^{k-1}( n (\nu(B_r(x)))^j/j!) \exp(-n \nu(B_r(x))). 
\end{align}

Given random variables $X,Z$ taking values in $\NN_0:= \NN\cup\{0\}$, define
the total variation distance
\begin{align*}
\dtv(X, Z) := \sup_{B\subset \NN_0}|\PP[X\in B] - \PP[ Z \in B]|.
\end{align*}
Given $\al >0$, 
let $\PRV_\al$ be Poisson distributed with mean $\al$.

\begin{proposition}[Poisson approximation]
	\label{p:poisson2d} 
	\label{p:poisson3d+} 
	Let $\beta' >0$.
	Let $(r_n)_{n \geq 1}$ be chosen
	so that $r_n \geq 0$ for all $n$ and
	\begin{align}
		\lim_{n \to \infty} \EE[\xxi_{n,r_n}] = \beta'.
	\label{e:Econv}
	\end{align}
	Then
	we have   
\begin{align}
	\dtv(\xxi_{n,r_n}, \PRV_{\EE [\xxi_{n,r_n}]}) 
	= O((\log n)^{1-d}) ~~~{\rm as}~ n \to \infty.
	\label{e:TVrate}
\end{align}
	In particular, with $L_{n,k} = L_k(\eta_n)$ defined at
	\eqref{e:defL_n},  
\begin{align}
	\PP[ L_{n,k} \leq r_n ] -   \exp(-\EE [\xxi_{n,r_n}]) 
	= O((\log n)^{1-d}) ~~~{\rm as}~ n \to \infty.
	\label{e:Lrate}
\end{align}
\end{proposition}

We prepare for proving this with three lemmas, the first of 
which is used repeatedly later on.
\begin{lemma}
	Under the WA, assuming either that $ \partial A \in C^2$
	or $d=2$ and $A$ is polygonal, there exists
	a constant $\delta_0>0$ (depending on $A$ and $f$)
	such that
	\begin{align}
		2 \delta_0 r^d \leq \nu(B_r(x)) \leq
		\theta_d r^d \fmax,  ~~~~~ \forall ~ x \in A, r \in (0,1].
		\label{e:ballbds}
	\end{align}
\end{lemma}
\begin{proof} The second inequality is clear. The first inequality
	follows from
		 Lemma \ref{lemgeom1a}-(i)
		in the case where $ \partial A \in C^2$,
		and can be seen directly when $A$ is polygonal.
\end{proof}

	\begin{lemma}
		\label{l:rlog}
		Let $\beta' \in (0,\infty)$ and suppose
		that $(r_n)_{n \geq 1}$ satisfies
		\eqref{e:Econv}.  Then  we have that
		$\liminf_{n \to \infty}( nr_n^d/\log n)  \geq 1/(\fmax \theta_d)$,
		and
		$\limsup_{n \to \infty}( nr_n^d/\log n) < \infty$.
	\end{lemma}
	\begin{proof}
		Let $\alpha \in (0,1/ (\fmax \theta_d) )$.
	If $nr_n^d < \alpha \log n$, then for all $x \in A$ we have
	$n \nu(B_{r_n}(x) ) \leq n \theta_d \fmax r_n^d \leq \alpha \theta_d \fmax
		\log n$. Therefore by \eqref{e:EEF}, 
		$
		\EE[ \xi_{n,r}] \geq n \int e^{-n \nu(B_{r_n}(x))} \nu(dx) 
	\geq n^{1- \alpha \theta_d \fmax },
		$
	so the condition (\ref{e:Econv}) implies that
	$n r_n^d \geq \alpha \log n$ for all large enough $n$.
		The first claim follows.

		For the second claim,
		let $\delta_0 >0$ be as in 
		\eqref{e:ballbds}. 
		Take $s_n >0$ so that $ns_n^d = \delta_0^{-1} \log n$.
		Using \eqref{e:ballbds}, 
		for some constant $c$ we have
		\begin{align*}
		n \int_A
			p_{n,s_n}(x)
			\nu(dx)
		\leq c n (\log n)^{k-1} \exp(-2 n \delta_0 s_n^d)
		= c (\log n)^{k-1}n^{-1},
		\end{align*}
		which tends to zero.
		Hence  by \eqref{e:Econv} 
		we have $r_n \leq s_n$ for $n $ large,
		and hence the second claim.
	\end{proof}

Given $x,y \in \R^d$ and $n, r  >0$,
setting $\eta_n^x := \eta_n\cup\{x\}$,
we define
the  quantity
\begin{align*}
	q_{n,r}(x,y) : = \PP[\eta^y_n(B(x,r))\le k-1, \eta^x_n(B(y,r))\le k-1].
\end{align*}
Our proof of Proposition \ref{p:poisson2d} 
is based on the following estimate which was proved in \cite{Pen}
by the local dependence approach of Stein's method. 
\begin{lemma}[{\cite[Theorem 6.7]{Pen}}]
	\label{l:Stein}
	Let $n, r>0$. Then
	\begin{align*}
	\dtv(\xxi_{n,r},\PRV_{\EE[\xxi_{n,r}]})\le 3(I_1(n,r)+I_2(n,r))		
	\end{align*}
where 
	\begin{align*}
		I_1(n,r) &= n^2 \int 
		\1\{\|x-y\|\le 3r\} p_{n,r}(x)p_{n,r}(y) \nu^2(d(x,y))\\
		I_2(n,r) &= n^2 \int \1\{\|x-y\|\le 3r\} q_{n,r}(x,y)
		\nu^2(d(x,y)).
	\end{align*}
\end{lemma}

\begin{proof}[Proof of Proposition \ref{p:poisson2d}] 
	Observe first that whenever $|\cP_n| \geq k+1$, the statement
		$L_{n,k} \leq r_n$ is equivalent to $\xxi_{n,r_n} =0$, so that
		$
		|\PP[L_{n,k} \leq r_n] - \PP[\xxi_{n,r_n} =0]| \leq 
		\PP[|\cP_n| \leq k] = O(n^k e^{-n}).
		$
		Therefore \eqref{e:TVrate} will imply \eqref{e:Lrate},
		so it suffices to prove \eqref{e:TVrate}.
%todo:	{\bf [Compare with similar proof in two sample paper]}
		
		By \eqref{e:ballbds},
		provided $n$ is large enough,
		for all $y \in A$ we have 
		\begin{align*}
			p_{n,r_n}(y) \le  k
			( nf_\mathrm{max} \theta_d r_n^d)^{k-1} 
			\exp(-n\nu(B(y,r_n)))
			\le \exp(- \delta_0 nr_n^d).
		\end{align*}
	 Therefore, using \eqref{e:EEF} in the second line below
	 we have
	\begin{align}
		\label{e:I1bound}
		I_1(n,r_n) & \le n (3^d \fmax \theta_d r_n^d) \exp(-
		\delta_0 nr_n^d)
		n \int  p_{n,r_n}(x) \nu(dx)
		\nonumber \\
		& \le \exp(- (\delta_0/2)
		nr_n^d) \EE[\xxi_{n,r_n}].
	\end{align}
		
	Now we estimate $I_2:= I_2(n,r_n)$.
	Since the integrand of $I_2$ is symmetric in
	$x$ and $y$,
\begin{align*}
	I_2 \leq 2 n^2 \int \1\{\|x-y\|\le 3r_n, \dist(x,\partial A)\le 
	\dist(y,\partial A)\} q_{n,r_n}(x,y) \nu^2(d(x,y)).
		\end{align*}
	To further simplify the integral, writing $B_x = B(x,r_n)$ and likewise for $B(y,r_n)$,
	we have
	\begin{align*}
		q_{n,r_n}(x,y) &
		\le \PP[\eta_n(B_x)\le k-1,
		\eta_n(B_y\setminus B_x)\le k-1] \\
		& = p_{n,r_n}(x) \PP[\eta_n(B_y\setminus B_x)\le k-1]. 
	\end{align*}

Consider first the case where $A$ has a $C^2$ boundary.
	If $\dist(x,\partial A) \leq \dist(y,\partial A)$,
	setting $\kappa_d := 2^{-3d-1}\theta_{d-1}$ and 
	using Lemma \ref{lemgeom1a}-(ii) for the lower bound and Fubini's
	theorem for the upper bound below, we have
	\begin{align*}
		f_0 \kappa_d \|y-x\|r_n^{d-1} \leq 
		\nu(B_y \setminus B_x) \leq \fmax \theta_{d-1}r_n^{d-1}
		\|y-x\|,
	\end{align*}
	and hence
	\begin{align*}
		q_{n,r_n}(x,y) &
		 \le p_{n,r_n}(x)
		\sum_{j=0}^{k-1}
		(n \fmax \theta_{d-1} r_n^{d-1} \|y-x\| )^j
		\exp(- \kappa_d f_0 \|y-x\| n r_n^{d-1}).
	\end{align*}
Therefore, we have
\begin{align*}
	I_2 & \le 2 \max(\fmax \theta_{d-1},1)^{k-1} n^2 \\
	& \times	\int_A
	%\int_{\|x-y\|\le 3r}
	\left( \int_{B(x,3r_n)}
	\sum_{j=0}^{k-1}
	(n r_n^{d-1} \|y-x\|)^j
	\exp(- \kappa_d f_0 \|y-x\| n r_n^{d-1}) \nu(dy) \right)
	p_{n,r_n}(x)  \nu(dx)   
	.
\end{align*}
A change of variables $z= nr_n^{d-1} (y-x)$ shows that the 
	inner
	integral is bounded by $c' r_n^d (nr_n^d)^{-d}$ for some finite constant $c'$. Together with \eqref{e:EEF}, this yields for
	some further constant $c''$ that 
\begin{align}
	I_2 \le 2c'' (nr_n^d)^{1-d} \EE[\xxi_{n,r_n}].
	\label{e:I2le}
\end{align}
	This, together with \eqref{e:I1bound} and \eqref{e:Econv},
	shows that $I_1+I_2 = O((nr_n^d)^{1-d})$; applying
	 Lemmas \ref{l:Stein} and \ref{l:rlog}
	 proves \eqref{e:TVrate} as required for this case.

	 Now consider the other case, where
$d=2$ and $A$ is polygonal. 
Let $x,y \in A$ with
$\|y-x\| \leq 3r_n$ and $\dist(x,\partial A) \leq \dist (y,\partial A)$.  
	By Lemma \ref{l:poly_diff_ball}, there exists 
	$\delta_1>0$ such that $\nu(B(y,r_n) \setminus B(x,r_n))
	\geq \delta_1 \|x-y\| r_n$.  
	Using this,
	we can estimate the contribution to $I_2$ from $x,y$ not too close to the corners similarly to how we estimated $I_2$ at \eqref{e:I2le} in  the
	previous case.
	
	Suppose instead that $x$ is close to a corner of $A$ and
	$\|x-y\|\le 3r_n$. By \eqref{e:ballbds}
	the contribution to $I_2$ from such pairs $(x,y)$ is at most 
	$c''' n^2 r_n^4  \exp(-\delta_2 nr_n^2)$ for suitable constants 
	$c''' <\infty, \delta_2>0$.
	Hence by Lemma \ref{l:rlog}, this
	contribution is $O(n^{-\delta_3})$ for some $\delta_3 >0$.
	The proof is now complete. 
\end{proof}

\section{Relating $L_{n,k}$ to $M_{n,k}$} \label{ss:lifted}
\label{s:RelateLM}
Throughout this section we assume that $\partial A \in C^2$ or
that $A$ is a convex polygon. We adopt our WA but do not
assume $f$ is necessarily constant on $A$.

Fix $k\in\NN$.
Recall that $L_{n,k}$ and $M_{n,k}$ were defined at 
(\ref{e:defL_n}) and (\ref{e:defM_n}).
While Proposition \ref{p:poisson2d} provides an understanding
of $\PP[L_{n,k} \leq r_n]$ for suitable $r_n$, for
Theorems \ref{t:smooth} and \ref{t:nonunif} we also
need to understand the limiting behaviour
of $\PP[M_{n,k} \leq r_n]$
and $\PP[M_{k}(\cX_n) \leq r_n]$. In this section, we work towards
this by showing (in Proposition \ref{p:L<M}) that 
$\PP[L_{n,k}\le r < M_{n,k}]$
and
$\PP[L_{k}(\cX_n) \le r < M_{k}(\cX_n) ]$
are small for $n$ large and $r$ small.

Suppose $\cX \subset \R^d$ is finite, and $r >0$.
%with $|\cX |\geq k+2$.
We adapt terminology from \cite[p. 282]{Pen}.
For $j \in \NN_0 = \NN \cup \{0\}$
a $j$-{\em separating pair} for the geometric
graph $G (\cX,r)$ means a pair of disjoint non-empty subsets
$\cY,\cY'$ of $\cX$ such that $G(\cY,r)$ and
$G(\cY', r)$ are both connected, $G(\cY\cup\cY', r)$ is not, and $\cX\setminus (\cY \cup \cY')$ contains at most $j$
points within distance $r$ of $\cY\cup \cY'$.  

When we need to refer to an individual set in a separating pair, we use the terminology {\em separating set}. That is, for $j \in \NN_0$
a $j$-separating set for the graph $G(\cX,r)$ is a set
$\cY \subset \cX$ such
that $G(\cY,r)$ is connected, and with $\Delta \cY$ denoting the set of
sites in $\cX \setminus \cY$ adjacent to $\cY$, we have $|\Delta \cY| \leq
j$ and $\cX \setminus (\cY \cup \Delta \cY) \neq \emptyset$.

\begin{lemma}
	\label{l:LrM}
Suppose $\cX \subset \R^d$ is finite with $|\cX |\geq k+2$. Let $r >0$, and 
suppose $L_k(\cX) \leq r < M_k(\cX)$. Then there exists a $(k-1)$-separating
	pair $(\cY,\cY')$ for $G(\cX,r)$  such that neither $\cY$ nor
	$\cY'$ is a singleton.
\end{lemma}
%Let $r>0$.
\begin{proof}
	Since $M_k(\cX) >r$, 
	the graph
	$G(\cX,r)$ is not $k$-connected.
	Therefore
by \cite[Lemma 13.1]{Pen}, 
	it has a $(k-1)$-separating pair
	%but not both.
%When $M_{n,k}>r$,  there exists a $(k-1)$-separating pair 
	$\cY,\cY'\subset\cX$. Since also  $L_{k} (\cX) \le r$,
	every vertex $x\in \cX$ 
has degree at least $k$, which implies that neither $\cY$ nor $\cY'$ can be a singleton.
\end{proof}
Our strategy in this section is to estimate
the probability that there exists a pair of non-singleton
separating sets for $G(\Po_n,r)$ or $G(\cX_n,r)$. We do this in
stages, according to the size of the separating sets.

For $x,y \in A$,
we write $x < y$ if 
$x$ precedes $y$ in the lexicographic ordering.
We 
define the following ordering $\prec $ on $A$, that we shall use repeatedly:
\begin{align}
	x \prec y \Leftrightarrow
	(\dist(x,\partial A) < \dist(y,\partial A))
	{\rm ~or~}
	(\dist(x,\partial A) = \dist(y,\partial A) {\rm ~and ~}
	x < y).
	\label{e:defprec}
%todo: in the lexicographic ordering {\bf define this ordering earlier?}.
\end{align}

\subsection{Small separating sets}

The goal of this section is to prove that for any fixed vertex $x\in A$, the probability that $x$ belongs to a non-singleton $(k-1)$-separating set of 
`small' diameter in $G(\eta_n \cup \{x\}, r)$ is negligible compared to the probability that it has degree at most $k-1$, provided that $x \prec y$ 
for all
other
$y$ in the separating set containing $x$, where  the ordering $\prec$ was
defined at \eqref{e:defprec}.

We introduce further notation. With $k$ fixed,
for $r >0$ and finite $\cX \subset \R^d$, $x \in \R^d$,
let $\cC_{r}(x,\cX)$ denote the collection
of $(k-1)$-separating sets $\cY$ for $G(\cX \cup \{x\},r)$ 
containing $x$ such that moreover $x \prec y$ for all
$y \in \cY \setminus \{x\}$. Given
also $\rho >0$, we are interested in the event 
\begin{align}
	\label{e:defE}
E_{x,\rho,r}(\cX) 
:= \{\exists \cY\in\cC_{r}(x,\cX), 0<\diam(\cY) \leq \rho r
	\}.
\end{align}

\begin{lemma} \label{l:smallK}
	(i) Suppose $d \geq 2$ and $A$ has $C^2$ boundary.
	Then there exist 
	$\delta, r_0\in (0,1)$ and 	$c <\infty$ such that for all 
	$n\ge k+2$,  any $x\in A$ 
	and any $r \in (0,r_0)$ we have
	\begin{align}\label{e:Zub}
		\PP[E_{x,\delta,r}(\eta_n)] \leq cp_{n,r}(x) (nr^d)^{1-d}; \\
		\label{e:Zbinub}
		\PP[E_{x,\delta,r}(\cX_{n-1})] \leq cp_{n,r}(x) 
		(nr^d)^{1-d},
	\end{align}
	where $p_{n,r}(x)$ was defined at \eqref{e:pnj}
	
	(ii) Suppose $d=2$ and $A$ is polygonal. 
	Then there exist $K \in (0,\infty)$, 
	and $\delta, r_0 \in (0,1)$ and $c <\infty$ such that 
	for all $n \geq 3$, $x\in A \setminus \Cor^{(Kr)}$, and $r \in (0,r_0)$,
	\eqref{e:Zub} and \eqref{e:Zbinub} hold.
\end{lemma}

\begin{proof}
	(i) Let $\delta \in (0,1)$. Suppose that $E_{x,\delta,r}(\eta_n)$
	occurs with some $\cY$. Then by considering the vertex furthest from 
	$x$ in $\cY$, we see that there exists $y\in \eta_n$ such that
	$\cY\subset B_{\|y-x\|}(x)$ and $\|y-x\|\le \delta r$.
	Moreover, setting $D_{x,y} := (B_r(x)\cup B_r(y))\setminus
	B_{\|y-x\|}(x)$ we have that $\eta_n ( D_{x,y}) \leq k-1$.
	By Markov's inequality,
	$\PP[E_{x,\delta,r}(\eta_n) ]$ is bounded above by the expected number
	of $y\in \eta_n \cap B_{\delta r}(x)$ satisfying
	$\Po_n(D_{x,y}) \leq k-1$,
%	the aforementioned properties,
	and hence by
	the Mecke formula
	\begin{align}
		\PP[E_{x,\delta,r}(\Po_n)] \leq
		n \int_{B(x,\delta r)} \PP[\Po_n(D_{x,y}) \leq k-1] \nu(dy).
		\label{e:0926a}
	\end{align}
	
	To proceed, we need to bound the volume of $D_{x,y}$ from below.
	By Lemma \ref{lemgeom1a}-(ii), there exists
	$r_0>0$ such that for all $r \in (0,r_0 )$ and $x, y\in A$ with $\|x-y\|\le r$
	and $x\prec y$, setting $\kappa_d := 2^{-3d-1} \theta_{d-1}$
	we have
	$$
	\nu(B_r(x) \cup B_r(y) ) \ge \nu(B_r(x)) +
	2 \kappa_d f_0 r^{d-1}\|y-x\|.
	$$
	Hence, for $r < r_0$,  for $x,y \in A$ with  $\|y-x\|\le \delta r $ and 
	$x\prec y$,
	\begin{align*}
	\nu(D_{x,y})
	\geq \nu(B_r(x)) +  2 \kappa_d  f_0 r^{d-1} \|y-x\|
	- \fmax \theta_d \|y-x\|^d. 
	\end{align*}
	Now provided $\delta \leq (\kappa_df_0/(\fmax \theta_d))^{1/(d-1)}$,
	we have  $\fmax \theta_d \|y-x\|^d\le \kappa_d  f_0 r^{d-1} \|y-x\|$, yielding 
	\begin{equation}
		\label{e:nulb}
			\nu(D_{x,y})
			\geq \nu(B_r(x)) +  \kappa_d  f_0 r^{d-1} \|y-x\|.
	\end{equation}
	By \eqref{e:ballbds},
	there is also a bound the other way, namely $\nu(D_{x,y}) \leq
	\nu(B_r(x) \cup B_r(y))
	\leq c_0 \nu(B_r(x))$ for some constant $c_0
	\in [1,\infty)$.  
	Using \eqref{e:0926a}
	 and the preceding upper and lower bounds on
	 $\nu(D_{x,y})$,
	 we have
	\begin{align}\label{e:Exrho}
		\PP[E_{x,\delta,r}(\eta_n)]\le c_0^{k-1} 
		n 
		\int_{B(x,\delta r)} 
		\sum_{j=0}^{k-1} 
		((n\nu(B_r(x)))^j/j!)
		e^{-n \nu(B_r(x)) - n r^{d-1} f_0 \kappa_d \|y-x\|}
		\nu(dy).
	\end{align}
	Recall that $p_{n,r}(x)=\sum_{j=0}^{k-1}
	((n \nu(B_r(x)))^j/j!) e^{-n \nu(B_r(x))}$.
	Changing variable to $y'=y-x$, and then to 
	$z= nr^{d-1}y'$ leads to 
	\begin{align*}
		\PP[E_{x,\delta,r}(\eta_n)] 
		&\le  c_0^{k-1} f_{\max} n p_{n,r}(x)
		\int_{\|y'\|\le \delta r} e^{-n f_0 \kappa_d  r^{d-1} \|y'\|}
		dy'
		\\
		&\le  c' p_{n,r}(x) n (nr^{d-1})^{-d} \int_{\RR^d} 
		e^{- f_0 \kappa_d \|z\|} dz,
	\end{align*}
	for a suitable  positive constant $c'$, not
	depending on $r$ or $n$.
	This proves \eqref{e:Zub}.
	
	To prove \eqref{e:Zbinub}, we
	use similar reasoning to before, now
	using the union bound (instead of the Mecke formula) and the
	binomial distribution, to obtain that
	\begin{align*}
		\PP[E_{x,\delta,r}(\cX_{n-1}) ]
		\leq (n-1) 
		\int_{B(x,\delta r)}
		\sum_{j=0}^{k-1} 
		\binom{n-2}{j} 
		\nu(D_{x,y})^j(1-\nu(D_{x,y}))^{n-2-j}
		\nu(dy).
	\end{align*}
	As before we bound $\nu(D_{x,y})^j $ from above by
	$c_0^{k-1}\nu(B_r(x))^{j-1}$. Provided $r$ is sufficiently small,
	we have for all $x,y$ and all $ j \leq k-1$
	that $(1-\nu(D_{x,y}))^{-2-j} \leq 2$. Also we can bound the
	binomial coefficient from above by $n^j/j!$. Combining all of
	these and also using the bound $(1-t)\leq e^{-t} $ we obtain that
	\begin{align*}
		\PP[E_{x,\delta ,r}(\cX_{n-1}) ]
		\leq  2c_0^{k-1} n \int_{B(x,\delta r)}
		\sum_{j=0}^{k-1} \frac{(n \nu(B_r(x)))^j}{j!}
		\exp(-n \nu(D_{x,y})) \nu(dy);
	\end{align*}
	then using \eqref{e:nulb} and arguing similarly
	to the Poisson case, we obtain \eqref{e:Zbinub}.
	
	(ii) Suppose $d=2$ and $A$ is polygonal. We use Lemma \ref{l:poly_diff_ball} in place of Lemma \ref{lemgeom1a} to get lower bound \eqref{e:nulb}. This together with the simple upper bound $\nu(D_{x,y})\le c_0\nu(B_r(x))$ and the same reasoning as in part (i) lead to 	\eqref{e:Zub} and \eqref{e:Zbinub}  in this case too.
\end{proof}

\subsection{Medium sized separating sets}

\newcommand{\medevent}{F_{x,\eps,\rho,r}}

Recall the definition of $\cC_{r}(x,\cX)$ before the previous lemma, and
$p_{n,r} (x)$ at \eqref{e:pnj}.
Given $\eps, \rho$ with $ 0 < \eps < \rho < \infty$,
define the event
\begin{align}
	\label{e:defF}
\medevent(\Po_n) :=
	\{\exists \cY\in\cC_{r}(x,\eta_n), ~ \eps r<\diam(\cY)<\rho r
	\}.
\end{align}
The next lemma helps us
bound the probability of having a medium-sized separating set. 

\begin{lemma} \label{l:medK}
	(i) Suppose $ \partial A \in C^2$.
	Given $\rho, \eps \in \R$ with $0<\eps <\rho$,
	there exist $\delta, r_0, c >0$ such that for all $n \geq 2+k$,
	$r \in (0,r_0)$
	and  all $x\in A$, we have
	\begin{align}
		\PP[\medevent(\Po_n)] \leq c p_{n,r}(x) e^{-\delta nr^d};
		\label{e:Zmed}\\
		\PP[\medevent(\cX_{n-1})] \leq cp_{n,r}(x)
		e^{-\delta nr^d}.
		\label{e:Zbimed}
	\end{align}
	
	(ii) Suppose $d=2$ and $A$ is polygonal. Given $0 < \eps  < \rho <
	\infty$, there exists $K \in (0,\infty)$ and $\delta, r_0  >0$ such
	that for
	all $n \geq 3$, $ r \in (0,r_0)$
	and all $x \in A \setminus \Cor^{(Kr)}$, we have
	\eqref{e:Zmed}
	and \eqref{e:Zbimed}.
\end{lemma}

\begin{proof}
	Later in the proof we shall use the fact that since
	we assume $A$ is compact and $f$ is continuous on $A$ with $f_0>0$,
	\bea
	\lim_{s \downarrow 0} \big(
	\sup\{f(y)/f(x): x,y \in A, \|y-x\|\leq s\} \big)
	=1.
	\label{0703c}
	\eea
	
	(i) Suppose $\partial A \in C^2$.
	Without loss of generality, 
	we can and do assume $\eps <1$.  
	Let $\delta_1 := \delta(d,\rho,\eps )$ be as in Lemma
	\ref{l:Med2}.
	With $e_1$ denoting an arbitrary unit vector in $\R^d$,
	choose 
	$\delta_2  \in (0, 1/(99 \sqrt{d}))$
	such that
	\begin{align}
		|B_1(o)\setminus B_{1-\sqrt{d}\delta_2  }(o)|\le \delta_1.
		%~~~~~~~~~~~~~~
		%|B_1(o)\setminus B_1(2\sqrt{d}\delta_2   e_1)|\le \delta_1/16.
		\label{e:del2}
	\end{align}
	Partition $\RR^d$ into cubes of side length $\delta_2   r$.  Given $\cY \in \cC_{r}(x,\Po_n)$, 
	denote by $\A_{\delta_2}
	(\cY)$ the closure of the union of all the cubes in
	the partition that intersect $\cY$. Here
	$\A$ stands for ``animal'' and is unrelated to our
	underlying domain $A$.
	If
	$\diam \cY \in (\eps  r, \rho r]$, then
	$\A_{\delta_2}(\cY)  \subset B(x, \rho r + \delta_2 d^{1/2} r)$ and $\A_{\delta_2}(\cY) $
	can take at most
	$c := 2^{{(2\lceil (\rho /{\delta_2})  +\sqrt{d}\rceil
	)^d}}$ different possible shapes. 

	If the event  $\medevent(\Po_n)$ occurs there is at least one
	set $\cY \in \cC_{r}(x,\Po_n)$ with $\eps  r < \diam \cY
	\le \rho r$. If there are several such sets $\cY$, choose one
	of these according to some deterministic rule, and
	denote it by $\cY^*(\Po_n)$.
	
	Fix a possible shape $\sigma$ that might arise
	as $\A_{\delta_2}(\cY)$ for some $\cY \in \cC_{r}(x,\Po_n)$
	with $\diam \cY \in (\eps r,\rho r]$,
	and suppose
	the event $\medevent(\Po_n)\cap \{\A_{\delta_2}(\cY^*(\Po_n))
	= \sigma\}$ occurs. 
	Let $\sigma^*:= \{z\in \sigma: x\prec z \} \cup \{x\}$. 
	Set $H:=H(\sigma)= (\sigma^* \oplus 
	B_{(1-\sqrt{d}{\delta_2}) r}(o)) \setminus \sigma^*$. By the
	triangle inequality, $H\subset \cY^*(\Po_n)
	\oplus B_r(o)$. We claim that $\eta_n(H)\le k-1$. 
	Indeed, if there are $k$ or more points in $\eta_n\cap H$, then since
	$\cY^*(\Po_n) $ is $(k-1)$-separating,
	necessarily one of these points, denoted by $y$, belongs to
	$\cY^*(\Po_n) $.
	Hence  $y\in \eta_n\cap H\cap \cY^*(\Po_n)$, implying $y\in \sigma$ and therefore $y\in \sigma \setminus \sigma^*$ (since $y\in H$), but this would contradict the assumption that $x\prec y$ for all $y\in\cY^*(\Po_n) \setminus \{x\}$.

	Now we estimate from below the volume of $H \cap A$. 
	Recall that $\delta_1 = \delta(d,\rho,\eps)$ is as
	in Lemma \ref{l:Med2}. 
	Applying \eqref{e:MedLB2} from there
	leads to 
	\begin{align*}
		|H\cap A| \ge |B_{r(1-\sqrt{d}{\delta_2})}(x)\cap A| + 
		2 \delta_1 r^d. 
	\end{align*}
	By \eqref{e:del2}, $|(B_r(x) \setminus B_{r(1- \sqrt{d}\delta_2)}(x))
	\cap A | \leq \delta_1 r^d$ and hence
	$$
	|B_{r(1-\sqrt{d}{\delta_2})}(x)\cap A| 
	\ge |B_r(x)\cap A|- \delta_1 r^d.
	$$
	Let $\delta_3 \in (0,1/2)$ be such  that
	$ \delta_4:= (1-2\delta_3)(1 + \delta_1/(\fmax \theta_d)) - 1 >0 $.
	By the preceding estimates,
	and \eqref{0703c}, provided $r$ is small we have
	that
	\begin{align*}
		\nu(H) & \geq (1-\delta_3) f(x) \left(
		|B_r(x) \cap A| +  \delta_1r^d \right)
		\\
		& \geq (1-2 \delta_3) \nu(B_r(x)) 
		\left( 1 + \frac{  \delta_1 r^d}{ \fmax
			\theta_d r^d
		} \right) = (1+ \delta_4) \nu(B_r(x)).
	\end{align*}
	Let $\delta_5 = \delta_0 \delta_4$, with $\delta_0$ 
	given at \eqref{e:ballbds}.
	Then
	\begin{align}
		\label{e:0811-3}
		\nu(H) \ge
		\nu(B_r(x)) + \delta_5  r^d. 
	\end{align}
	Also, because of the upper bound on diameters and
	\eqref{e:ballbds},
	there is a constant $c_1 \in [1,\infty)$ such that
	$\nu(H) \leq c_1 \nu(B_r(x))$ uniformly over all possible
	$x$,  all small $r$, and all possible $\sigma$.

	Using these upper and lower bounds on $\nu(H)$, 
we can deduce that
	\begin{align*}
		\PP[\medevent(\Po_n)  \cap \{\A_{\delta_2}
		(\cY^*(\Po_n)) = \sigma\}]
		& \leq \PP[\Po_n(H) \leq k-1] \\
		& \leq c_1^{k-1} \sum_{j=0}^{k-1} 
		((n\nu(B_r(x)))^j/j!)
		e^{- n \nu(B_r(x)) - n \delta_5 r^d}
		\\
		& = c_1^{k-1} p_{n,r}(x) e^{-n \delta_5 r^d}.
	\end{align*}
	This, together with the union bound over the choice of possible shapes $\sigma$, gives us \eqref{e:Zmed}.
	
	To prove the result \eqref{e:Zbimed} for the binomial case
	we use
	the volume estimates \eqref{e:0811-3} and $\nu(H) \leq c_1 \nu(B_r(x))$
	once more. For $n$ large, we have
	\begin{align*}
		\PP[ \medevent(\cX_{n-1}) \cap \{ \A_{\delta_2}
		(\cY^*(\cX_{n-1})) 
		= \sigma\}] & \leq 
		\PP[\cX_{n-1}(H) \leq k-1]
		\\
		& = \sum_{j=0}^{k-1} \binom{n-1}{j} \nu(H)^j
		(1-\nu(H))^{n-1-j}
		\\
		& \leq 2 c_1^{k-1} \sum_{j=0}^{k-1}
		(n^j /j!) \nu(B_r(x))^j \exp(- n \nu(H))
		\\
		& \leq
		2  c_1^{k-1} p_{n,r}(x) e^{-n \delta_5 r^d},
	\end{align*}
	and hence \eqref{e:Zbimed}.
	
	(ii) Suppose $d=2$ and $A$ is polygonal. Let $0 < \eps  < \rho < \infty$.
	Choose $K$ such that for
	all $r$ and all $x \in A \setminus \Cor^{(Kr)}$, the ball $B_{(\rho +9)r}(x)$
	intersects at most one edge of $A$. We can choose such a $K$ by a similar argument to the proof of Lemma \ref{l:poly_diff_ball}.
	Then for $x \in A \setminus 
	\Cor^{(Kr)}$
	we can deduce \eqref{e:Zmed} and \eqref{e:Zbimed}
	in the same manner as in the proof of part (i).
\end{proof}

Next we consider the probability of having a small or medium-sized separating set near the corner of a polygon in dimension 2. 
We do not attempt to compare this probability with $p_{n,k-1}(x)$ because the corners have negligible area and we can get by with a less precise estimate. 

\begin{lemma}
	\label{l:polycor}
	Suppose that $d=2$ and $A$ is a convex polygon. 
	Given $\rho, K \in (0,\infty)$,
	 there exist constants  $c, \delta, r_0 \in (0,\infty)$
	depending only on $A$, $f_0$, $\rho$ and $K$, such that
	if $n \geq k+2$ 
	then
	\begin{align}
		\sup_{x \in \Cor^{(Kr)}, r \in (0,r_0)}
		\PP[\exists \cY\in\cC_{r}(x,\eta_n), \diam(\cY)\le \rho r ]
		\leq c \exp(-\delta n r^2);
		\label{e:CorMed}
		\\
		\sup_{x \in \Cor^{(Kr)}, r \in (0,r_0)}
		\PP[\exists \cY\in\cC_{r}(x,\cX_{n-1}), \diam(\cY)\le \rho r ]
		\leq c \exp(-\delta n r^2).
		\label{e:CorMedBin}
	\end{align} 
\end{lemma}

\begin{proof}
	Fix  $\rho,K \in (0,\infty)$. 
	Let $\alpha_{\rm min}$ be the smallest angle of the corners of $A$.
	Assume without loss of generality that one of the corners of $A$ lies at the origin, and moreover one of the edges of $A$ incident to
	the origin is in the direction of the positive $x$-axis, while the other edge is the in the anti-clockwise direction from the positive $x$-axis,
	and therefore lies in the upper  half-plane since $A$ is
	assumed convex.  
	%which can be either in the upper (convex polygon) or lower (non-convex polygon) half-plane.  
%todo: {\bf [Do we use the non-convex case?]}

	Let $\delta_2 :=  1/4$. 
	Define $\A_{\delta_2}(\cY^*(\Po_n))$ as in the proof of
	Lemma \ref{l:medK}.
	As argued there, if there exists 
	$\cY \in \cC_{r}(x,\Po_n)$ with $\diam(\cY) \leq \rho r$,
	then
	$\A_{\delta_2}(\cY^*(\Po_n)) $
	can take at most $c$ different possible shapes for some finite $c$ not depending on $r$.
	
	Suppose $x \in B_{Kr}(o)$.
	Fix a possible shape $\sigma$ that might arise 
	when there exists 
	$\cY \in \cC_{r}(x,\Po_n)$ with $\diam(\cY) \leq \rho r$,
	and suppose the event $\{\A_{\delta_2}
	(\cY^*(\Po_n)) = \sigma\}$ occurs. 
	Let $y_{\rm max}$ be the largest point of $\sigma$
	in the lexicographic ordering (i.e., the highest rightmost point).
	Let $S$ be a sector centred on $y_{\rm max}$
	of radius $r/2$ and with one straight edge from $y$ in
	the direction of the positive $x$-axis, while the other edge is
	in the anti-clockwise direction with 
	angle $ \min(\alpha_{\rm min},\pi/2)$ from the first edge.
	
	Then provided $r$ is small enough, $S \subset A $ and
	the interior of $S$ is disjoint from $\sigma$.
	Also 
	the squares making
	up $\sigma$ have diameter less than $r/2$, so $S$
	is contained in $\Po_n \oplus B_r(o)$; hence
	$\Po_n(S) \leq k-1$.
	Also 
	$
	f_0 \min(\alpha_{\min},\pi/2) r^2/2 \leq \nu(S) \leq
	(\pi/4) \fmax r^2.
	$
	 Therefore
	\begin{align*}
		\PP[\{\exists \cY\in\cC_{r}(x,\eta_n),
		\diam(\cY)\le \rho r\}\cap \{\A_{\delta_2}
		(\cY^*(\Po_n)) = \sigma\}] 
		\leq 
		k (n \fmax (\pi/4) r^2)^{k-1}
		\\
		 \times
	\exp(-n f_0 \min(\alpha_{\min},\pi/2) r^2/2). 
	\end{align*}
	Summing over all possible $\sigma$ and treating other corners
	similarly, we obtain (\ref{e:CorMed}).
	Also
	\begin{align*}
		\PP[\{\exists \cY\in\cC_{r}(x,\cX_{n-1}), 
		\diam(\cY)\le \rho r\}& \cap \{\A_\delta (\cY^*(\cX_{n-1})) =
		\sigma\}] 
		\\
		& \leq \sum_{j=0}^{k-1} 
		\binom{n-1}{j} \nu(S)^j(1-\nu(S))^{n-1-j}
		\\
		& \leq \sum_{j=0}^n 2((n \nu(S))^j/j!)\exp(-n \nu(S)),
	\end{align*}
	and using the same  upper and lower bounds on $\nu(S)$
	as before gives us \eqref{e:CorMedBin}.
\end{proof}

\subsection{Large separating pairs}

Given $r>0$, $\rho >0$,
recall  that if $L_{n,k}\le r< M_{n,k}$ then
there is a $(k-1)$-separating pair for $G(\Po_n,r)$,
and each individual set in the pair is non-singleton.
Then, 
either there exists a non-singleton $(k-1)$-separating set with diameter at most $\rho r$, or both sets in the pair have diameter greater than
$\rho r$.
Our next lemma deals with the latter possibility. 
Given $\rho >0,r\geq 0$, define the event
$$
H_{r,\rho}(\cX) =
\{ \exists \text{ a } (k-1)\text{-separating pair } \cY, \cY'
\text{ for } G(\cX,r), \min(\diam(\cY), \diam(\cY'))> \rho r \}.
$$

\begin{lemma}\label{l:2large}
	Suppose $(r_n)_{n >0}$
	satisfies \eqref{e:Econv} for some $\beta' \in (0,\infty)$.
	Then there exists $\rho \in (0,\infty)$
	such that 
	$\PP[H_{r_n,\rho}(\eta_n)] = O(n^{-2})$ and 
	$\PP[H_{r_n,\rho}(\cX_n)] = O(n^{-2})$
	as $n\to \infty$.
%todo:	{\bf [Maybe only need binomial version of this but this might
%	need a lot of checking]}
\end{lemma}

\begin{proof}
	{\bf Case 1: $\partial A \in C^2$.} 
	See \cite[Equation (3.14)]{Pen99b}.
	That result is formulated only for $\cX_n$, not for $\Po_n$,
	and also only for the case $k=0$.
	However, it uses only the probability bound that if
	$X $ is binomial with mean $\mu$ then
	$\PP[X =0 ] \leq e^{-\mu}$. Using a standard Chernoff bound,
	e.g. \cite[Lemmas 1.1 and 1.2]{Pen},
	we have that if $X$ is either binomial or Poisson 
	distributed with mean $\mu$ for $\mu$ sufficiently large,
	we have $\PP[X < k] \leq e^{-\mu/2}$, and using
	this we can readily adapt the argument in
	\cite{Pen99b} to the generality required here.

	{\bf Case 2: $d=2$ and $A$ is polygonal.}
	In this case we use the proof of \cite[Lemma 3.12]{PYpoly}.
	Our $r_n$ is not quite the same as there, but the argument
	works for our $r_n$ too; the properties of $r_n$ given in
	Lemma \ref{l:rlog} are sufficient. Again, the
	proof in \cite{PYpoly} is only for $\cX_n$,  but it
	relies only on Chernoff probability bounds for 
	a binomial random variable, which also apply
	for a Poisson random variable with the same mean and
	therefore the result holds for $\Po_n$
	as well as for $\cX_n$.
%
	%LATER: reinstate the next paragraph for non-convex polygons?
	%That said, there is one more subtle point when 
	%applying \cite[Lemma 3.12]{PYpoly}. Without a convexity assumption on the polygon, the lower bound in \cite[Equation (3.13)]{PYpoly} does not hold naively for a candidate optimal $a_j$ value as would appear in the general lower bound for the largest nearest neighbour link (hence also the connectivity threshold under WA) in \cite[Lemma 3.5]{PYpoly}. In fact, the proof of \cite[Lemma 3.12]{PYpoly} does not rely on the fine estimate of volume, but rather the condition (G) there. This condition was checked in \cite[Proposition 3.15]{PYpoly} and can readily be extended to the non-convex polygon case. 
	%
%	todo: is this good enough? The reason why we don not need a precise volume lower bound in the 2d polygonal case is that small or medium sized separating sets are unlikely to appear near the corners. 
%	Note we assumed the convexity of the polytope to verify the
%	conditions on $A$ for \cite[Lemma 3.12]{PYpoly}.  See
%	\cite[Lemma 3.13]{PYpoly}. So if we want more general
%	polygons, would need to generalise that result to more
%	general polygons (eg simply connected but not convex?)
%	which might be possible.
%	{\bf [Do this! Or restrict to convex polygons]}
\end{proof}

We can now bound  $\PP[L_{n,k} \leq r_n < M_{n,k}]$,
which is the result we have been leading up to in this whole section.

\begin{proposition}\label{p:L<M}
	Let $\beta' \in (0,\infty)$ and suppose $(r_n)_{ n\geq 1}$ satisfies \eqref{e:Econv}. Then as $n \to \infty$,
	\begin{align}
		\PP[ L_{n,k}\le r_n< M_{n,k}] = O( (\log n)^{1-d});
		\label{e:0414a}
		\\
		\PP[ L_{k}(\cX_n) \le r_n< M_{k}(\cX_n)] = O( (\log n)^{1-d}).
		\label{e:0414abin}
	\end{align}
\end{proposition}

\begin{proof}
	Given $ r, \rho \in (0,\infty)$, and finite
	$\cX \subset \R^d$, define the event 
	\begin{align*}
		J_{r,\rho}(\cX)	 :=\{ \exists x \in \cX, \cY\in 
		\cC_{r}(x,\cX), 0<\diam(\cY)\le \rho r \}.
	\end{align*}
	By Lemma \ref{l:LrM},
%	As discussed at the start of this section,
	if $L_k(\cX)
	\leq r_n < M_k(\cX)$, then either $J_{r_n,\rho}(\cX)$
	or $H_{r_n,\rho}(\cX)$ occurs. Hence
	by Lemma \ref{l:2large}, it suffices to prove that for any
	$\rho \in (0,\infty)$, the events $J_{r_n,\rho}(\cP_n)$
	and $J_{r_n,\rho}(\cX_n)$
	occur with probability $O((\log n)^{1-d})$ as $n\to\infty$.
	
	{\bf Case 1: $\partial A \in C^2$.}
	Fix $\rho \in (0,\infty)$.
	Let $N_n$ denote the (random) number of $x \in \Po_n$ such
	that there exists a  $\cY \in \cC_{r_n}(x,\Po_n)$
	with 
	$0 < \diam (\cY) \leq \rho r_n$.
	By Markov's inequality $\PP[J_{\rho,r_n}(\Po_n)]=
	\PP[N_n\ge 1]\le \EE[N_n]$.
	
	Let $\delta$ be as in
	Lemma \ref{l:smallK} (i) and assume without loss of generality that
	 $0 < \delta <  \rho $. Then by the Mecke equation and the definitions
	of $E_{x,\rho,r}$ and $F_{x,\eps,\rho,r}$
	at \eqref{e:defE} and \eqref{e:defF},
	and the union bound,
	\begin{align}
		\EE[N_n] &\le n \int_A \PP[E_{x,\delta,r_n}(x,\eta_n)] \nu(dx) 
		+ n \int_A \PP[F_{x,\delta,\rho,r_n}(x,\eta_n)] \nu(dx).
		\label{e:MarkMeck}
	\end{align}
	Using Lemma \ref{l:smallK} for the first integral and
	Lemma \ref{l:medK} for the second integral,
	and \eqref{e:EEF},
	we can find $c, \delta_2 \in (0,\infty)$ such that
	for large enough $n$ we have that
	\begin{align*}
		\EE[N_n]\leq c(  (nr_n^d)^{1-d}
		+ e^{-\delta_2 n r_n^d} ) \EE[\xxi_{n,r_n}],
	\end{align*}
	where $\xxi_{n,r}$ was defined at \eqref{e:def_isoVer}.
	%(not related to $F_{x,\delta,\rho,r}$). 
%	todo:{\bf [can we avoid this near-clash]}
	From this, we obtain the claimed 
	estimate
	\eqref{e:0414a}
	by \eqref{e:Econv} and Lemma \ref{l:rlog}. 
	
	{\bf Case 2: $d=2$ and $A$ is polygonal.}
	Fix $\rho \in (0,\infty)$.
	Let $\delta $ be as in Lemma \ref{l:smallK} (ii);
	assume without loss of generality that $0 < \delta < \rho$.
	By a similar argument to (\ref{e:MarkMeck}) we have
	\begin{align*}
		\PP[J_{\rho,r_n}(\Po_n)]
		& \leq  \int_{A \setminus \Cor^{(Kr_n)}} \PP[E_{x,\delta,r_n}(x,\eta_n)] n \nu(dx).
		\nonumber  \\
		& +  \int_{A \setminus \Cor^{(Kr_n)}} 
		\PP[F_{x,\delta,\rho,r_n}(x,\eta_n)] n \nu(dx).
		\nonumber  \\
		& +  \int_{\Cor^{(Kr_n)}} 
		\PP[\{\exists \cY\in\cC_{r_n}(x,\eta_n), \diam(\cY)\le \rho r_n \}] n \nu(dx).
	\end{align*}
	We can deal with the first two integrals just as we did in
	Case 1. By Lemma \ref{l:polycor}, there is a constant
	$\delta'$ such that the third integral is 
	$O(nr_n^2\exp(-\delta' nr_n^2))$,
	which completes the proof of \eqref{e:0414a}
	for this case.
	
	The proof of \eqref{e:0414abin} is identical, now
	relying on the binomial parts of Lemmas \ref{l:smallK}--\ref{l:2large}. We omit the details. 
\end{proof}

\section{Proof of Theorem \ref{t:nonunif}}

Throughout this section we assume that $\partial A \in C^2$ or
that $A$ is a convex polygon. We adopt our WA but do not
assume $f$ is necessarily constant on $A$.

\subsection{Proof of parts (i) and (ii)}

%LATER: maybe change $n$ to $t$ round here.

Given $\beta \in \R$, choose $n_0(\beta)$ such that
$n_0(\beta) > e^{-\beta}$ and
$e^{-n} \sum_{j=0}^{k-1}(n^{j+1})/j! < e^{-\beta}$ for all $n \in [n_0(\beta),\infty)$.
%For example we could take $n_0(\beta):=\max(2e^{-\beta},2\beta,9)$.
Recall the definition of $p_{n,r}(x)$ at \eqref{e:pnj}.
Given $n  \geq n_0(\beta)$,
define $r_n (\beta) \in (0,\infty)$ by
%to be the unique solution in $(0,\infty)$ to 
	\begin{align}
		r= r_n(\beta) \Longleftrightarrow
		n \int_A  p_{n,r}(x) 
	%	\sum_{j=0}^{k-1} ((n \nu(B_r(x)))^j/j!)
	%	\exp(-n \nu(B_{r}(x)) )
		\nu(dx) = e^{-\beta}.
		\label{e:Eeq}
	\end{align}
	By the Intermediate Value theorem, such an $r_n(\beta)$ exists
	and is unique (note that the integrand is
	nonincreasing in $r$ because the Poisson$(\lambda)$
	distribution is stochastically monotone in $\lambda$).
	Moreover, for $- \infty < \beta < \gamma < \infty$ and
	$n \geq \max(n_0(\beta),n_0(\gamma) )$,
	we have $r_n(\beta) < r_n(\gamma)$.

	We first determine the first-order limiting behaviour of
	$r_n(\beta)$.
\begin{lemma}
	\label{l:rfirstorder}
	Let $\beta \in \R$ and let $ r_n(\beta)$ satisfy
	\eqref{e:Eeq} for all $n > n_0(\beta)$.
	Then 
	\begin{align}
		\lim_{n \to \infty}(n\theta_d r_n(\beta)^d/\log n)
		= \max( 1/f_0, (2-2/d)/f_1).
		\label{e:limr}
	\end{align}
	If also $\gamma \in \R$ with $\beta < \gamma$, then
	\begin{align}
		\lim_{n \to \infty} \sup_{x \in A}
		\big(
		\nu(B(x,r_n(\gamma)))/\nu(B(x,r_n(\beta))) \big)
		= 1.
		\label{e:limratio}
	\end{align}
\end{lemma}
\begin{proof}
	By Proposition \ref{p:poisson2d}, 
		as $t \to \infty$ (through $\R$),
	\begin{align}
		\PP[L_{t,k} \leq r_t(\beta)] \to
	 \exp(-e^{-\beta}).
		\label{e:Lnkwk}
	\end{align}

	 On the other hand, we claim that
	 $t \theta_d L_{t,k}^d/\log t \toP \max(1/f_0,(2-2/d)/f_1)$
	 as $t \to \infty$.  Indeed, writing
	 $N_t$ for the number of points of the Poisson
	 process $\eta_t$, 
	 we have
	 \begin{align}
	 t \theta_d L_{t,k}^d/\log t 
		 =  (N_t \theta_d L_{t,k}^d/\log N_t )  \times
		 (t/N_t) 
		 \times 
		 ( \log N_t/ \log t). 
		 \label{e:mixbin}
	 \end{align}
	 Since the conditional distribution
	 of $L_{t,k} $, given $N_t =n$, is that of $L_{k}(\cX_n)$,
	 we have from \eqref{e:SLLN_smooth} when $A$ has
	 a $C^2$ boundary, and 
	 from \eqref{e:SLLN_cube} when $A$ is a convex polygon,
	 %LATER: see what else in the paper is needs  this, which
	 %requires convexity.
	 that 
	 the first factor in the right hand side of \eqref{e:mixbin}
	 tends to 
	 $\max(1/f_0,(2-2/d)/f_1)$
	 in probability, and by Chebyshev's inequality the 
	 second factor also tends to 
	 $1$
	 in probability, from
	 which we can deduce the third factor also tends
	 to 1 in probability.
	 Combining these  gives us the claim.

	 Let $\alpha < \max(1/f_0,(2-2/d)/f_1)$.
	 By the preceding claim we have
	 $\PP[t \theta_d L_{t,k}^d/\log t < \alpha] \to 0$,
	 and hence
	 by \eqref{e:Lnkwk},
	 $t \theta_d r_t^d/ \log t \geq \alpha
	 $ for $t $ large.  Similarly, if $\alpha' > \max(1/f_0,(2-2/d)/f_1)$
	 then 
	 $\PP[t \theta L_{t,k}^d/\log t \leq \alpha] \to 1$, so
	  $t \theta_d r_t^d/ \log t \leq \alpha'$ for $t$ large.
	  Combining these assertions gives us \eqref{e:limr}.
	  %{\bf [Only for convex polygons so far]}. 

	  For \eqref{e:limratio}, note that
	  %there exists $ r_0 >0, \delta_1 >0$ such that 
	  %$\nu(B_r(x)) \geq \delta_1 r^d$ for all
	  %$r \in (0,r_0), x \in A$.
	  %This can be seen using
	  % Lemma \ref{lemgeom1a}-(i) if $\partial A$ is smooth,
	  % or directly if $A$ is polygonal {\bf [Maybe we use
	  % this several times - give as separate lemma?]}
	  %Moreover
	  $\nu(B_s(x) \setminus B_r(x)) \leq \fmax \theta_d
	  (s^d-r^d)$ for $x \in A$ and $0<r<s $.
	  Therefore using \eqref{e:ballbds},
	  for all $x \in A$ and all large enough $n$ we have
	  \begin{align*}
		  \frac{\nu(B_{r_n(\gamma)}(x) \setminus
		  B_{r_n(\beta)}(x))}{\nu(B_{r_n(\beta)}(x))}
		  \leq
		  \frac{
		  \fmax \theta_d (r_n(\gamma)^d - r_n(\beta)^d)
		  }{ \delta_0 r_n(\beta)^d,
		  }
	  \end{align*}
	  which tends to zero by \eqref{e:limr}, and \eqref{e:limratio}
	  follows. 
 \end{proof}

Recall that $\mu(X)$ denotes the median of a random variable $X$.
For non-uniform $\nu$, it seems to be hard in general to 
find a formula for $r_n(\beta)$ satisfying (\ref{e:Eeq}) (even if 
the equality is replaced by convergence). However, if we can determine
a limit for $nr_n(\gamma)^d- nr_n(\beta)^d$ for all $\beta < \gamma$,
then we can still obtain a weak limiting distribution for
$nM_{n,k}^d- n\mu(M_{n,k})^d$ without
giving an explicit sequence for $\mu(M_{n,k})$. The next lemma
spells out this argument.

\begin{lemma}
	\label{l:Gum}
Suppose there exists $\alpha \in (0,\infty)$ such that for all
$\beta,\gamma \in \R $ with $\beta < \gamma$, we have as
$n \to \infty$ that
\begin{align}
	\lim_{n \to \infty} n (r_n(\gamma)^d - r_n(\beta)^d)
	= \alpha (\gamma - \beta), 
		\label{e:0422a2}
\end{align}
where $r_n(\beta)$ is defined by \eqref{e:Eeq}.
%Let $\beta \in \R$ and
	Suppose  that $(X_n)_{n>0}$ are random variables
	satisfying
	\begin{align}
	\lim_{n \to \infty}
	\PP[X_n \leq r_n(\beta)] = \exp(-e^{-\beta}), ~~~
	\forall~ \beta \in \R.
		\label{e:Xr}
	\end{align}
	Then
	\begin{align}
	n X_{n}^d - n \mu(X_{n})^d \tod \alpha(\Gum + \log \log 2)
	~~~~~{\rm as}~n \to \infty.
		\label{e:wkcentred}
	\end{align}
\end{lemma}
\begin{proof}
Set $\beta_0 = -\log \log 2$.  Let $r_n = r_n(\beta_0)$ and let
	$-\infty  < y < x < y' < \infty $.   Set
	$ s_n := r_n(\beta_0+y/\alpha)$ and
	$ s'_n := r_n(\beta_0+y'/\alpha)$.
	 Then by (\ref{e:0422a2}), 
	$n(s_n^d -r_n^d) \to y$ and
	$n((s'_n)^d -r_n^d) \to y'$.
	 Hence for $n$ large we have $ns_n^d < x + nr_n^d$ and
	  $n(s'_n)^d > x + nr_n^d$, so that
	  by \eqref{e:Xr}, setting $F(x):=\exp(-e^{-x})$ we have
$$
	\PP[n X_{n}^d -n r_n^d \leq x ] \geq
	\PP[n X_{n}^d  \leq n s_n^d ] 
	\to F( \beta_0 +  y/\alpha),
	$$
	and 
	similarly
%	taking $s'_n :=r_n(\gamma')$ with $\gamma' := \beta_0+y'/
%	\alpha$, for $n$ large  we have
	$
	\PP[nX_{n}^d - nr_n^d \leq x] \leq \PP[nX_{n}^d \leq n(s'_n)^d]
	\to F(\beta_0 + y'/\alpha).
	$
	Since we can take $y$ and $y'$ arbitrarily close to $x$
	and $F(\cdot)$ is continuous, we can deduce that
	$$
	\PP[nX_{n}^d - nr_n^d \leq x]  \to F(\beta_0 + x/\alpha),
	%~~~~~ x >0.
	~~~~~ x \in \R.
	$$

%Now let $u_n := r_n(\beta_0-y/\alpha)$ and $u'_n := r_n(\beta_0-y'/\alpha)$.
%	 Then by (\ref{e:0422a2}), $n(r_n^d -u_n^d) \to y$
%	and  $n(r_n^d -(u'_n)^d) \to y'$. Hence for $n$ large
%	$nu_n^d >  nr_n^d -x$ and
%	$n(u'_n)^d <  nr_n^d -x$, so that by \eqref{e:Xr},
%$$
%	\PP[n X_{n}^d -n r_n^d \leq -x ] \leq 
%	\PP[n X_{n}^d  \leq n u_n^d ] 
%	\to F( \beta_0 -  y/\alpha),
%	$$
%	while
%$
%	\PP[n X_{n}^d -n r_n^d \leq -x ] \geq 
%	\PP[n M_{n}^d  \leq n (u'_n)^d ] 
%	\to F( \beta_0 -  y'/\alpha).
%	$
%	Taking $y \uparrow x$ and $y' \downarrow x$ yields
%$$
%\PP[n X_{n}^d -n r_n^d \leq -x ] \to
%	 F( \beta_0 -  x/\alpha), ~~~~~ x >0.
%	$$
Since $F(\beta_0+ z/\alpha) = \PP[\alpha (\Gum + \log \log 2) \leq z]$, we
thus have 
\begin{align}
	n X_{n}^d - n r_n^d \tod \alpha (\Gum+\log \log 2).
	\label{e:limX}
\end{align}

Finally we need to check that $n\mu(X_{n})^d - n r_n^d \to 0$.  Let $\eps >0$.
By \eqref{e:limX}, as $n \to \infty$ we have
\begin{align*}
	\PP[nX_n^d- nr_n^d \leq \eps ] & \to \PP [ \alpha(\Gum + \log \log 2) \leq \eps]
> 1/2;
\\
	\PP[nX_n^d- nr_n^d \leq - \eps ] & \to \PP [ \alpha(\Gum + \log \log 2) \leq -
	\eps] < 1/2,
\end{align*}
so for large $n$ we have 
$$
\mu(X_n^d) - nr_n^d = \mu(n X_n^d - nr_n^d) \in [-\eps, \eps],
$$
so $\mu(X_n^d) - nr_n^d \to 0$ as $n \to \infty$, and
then \eqref{e:wkcentred} follows from \eqref{e:limX} and the continuity
of the Gumbel cdf.
\end{proof}

To use Lemma	\ref{l:Gum}, we need to show that \eqref{e:0422a2}
holds for some $\alpha $. We do this first for the case
where $f_1 > f_0(2-2/d)$.

\begin{lemma}
	\label{l:betas}
Let $ \beta , \gamma \in \R$ with $\beta < \gamma$.
Define $r_n(\beta)$ by \eqref{e:Eeq}.
	If $f_1 > f_0(2-2/d)$, then  
\begin{align}
	n(r_n(\gamma)^d -r_n(\beta)^d) \to (\gamma - \beta)/(\theta_df_0)
	~~~~~{\rm as}~ n \to \infty.
	\label{e:0422a}
\end{align}
\end{lemma}
\begin{proof}
	Given $n$, set $r=r_n(\beta), s=r_n(\gamma)$.
	For all $x \in A^{(-s)}$,
$
\nu(B_s (x) \setminus B_r(x)) \geq \theta_d f_0 (s^d - r^d),
$
	and hence using \eqref{e:Eeq}, we have
\begin{align}
	e^{- \beta } & \geq \int_{A^{(-s)}} 
	\sum_{j=0}^{k-1} ((n \nu(B_r(x)))^j/j!)
	e^{-n \nu(B_s(x))} 
	e^{n \nu(B_s(x) \setminus B_r(x))} n \nu(dx)
\nonumber
\\
	& \geq e^{n\theta_d f_0 (s^d - r^d)} 
	\int_{A^{(-s)}} \sum_{j=1}^{k-1} 
	((n \nu(B_r(x)))^j/j!)e^{-n \nu(B_s(x))} n \nu(dx). 
	%
	% + \int_{(\partial A)^{(s)}} \exp(-n \nu(B_r(x)) ) n\nu(dx).
	\label{e:0421a}
\end{align}

Suppose $\partial A \in C^2$.
	Using our  assumption $f_1 > f_0(2-2/d)$, 
	choose $\delta >0$ such that $(2-2/d)(f_1 - 2\delta)^{-1}
	< f_0^{-1}$.
	%Given $f_1^- < f_1$, 
	Using Lemma \ref{lemgeom1a}-(i) and the continuity
	of $f|_A$ we find for
all large enough $n$ and all $x \in (\partial A)^{(s)}$
	that $\nu(B_r(x)) \geq  \theta_d (f_1 -  \delta) r^d/2$.
	Then
	using Lemma \ref{l:rfirstorder}
	and our assumption  
	$f_1 > f_0(2-2/d)$, 
	we have that
	\begin{align}
		\lim_{n \to \infty} \big(
		n\theta_d r^d/ \log n \big) = f_0^{-1}
		> (2-2/d)(f_1- 2\delta)^{-1}.
		\label{e:0908c}
	\end{align}
Hence there are constants $c, c'$ such that for $n$ large
\begin{align*}
	\int_{(\partial A)^{(s)}} 
	\sum_{j=0}^{k-1}((n\nu(B_r(x)))^j/j!)e^{-n \nu(B_r(x)) }
	n \nu(dx)
	 \leq c (\log n)^{k-1}  n s e^{-n \theta_d (f_1 - \delta) r^d/2}
\\
	 \leq c' (\log n)^{k-1+1/d}
	 n^{1-1/d} \exp (- (f_1- \delta) (1- 1/d) (f_1- 
	 2 \delta)^{-1} 
	\log n),
\end{align*}
which tends to zero.

Suppose instead that $d=2$ and $A$ is polygonal. The preceding estimate
	shows  $\int_{(\partial A)^{(s)} \setminus \Cor^{(s)}} 
	\sum_{j=0}^{k-1} ((n \nu(B_r(x)))^j/j!)
	\exp(-n \nu(B_r(x)) ) n \nu(dx) $ tends to zero. Moreover,
	by \eqref{e:ballbds}
	there exist $c,\delta_0 \in (0,\infty)$ such that
	$$
	\int_{\Cor^{(s)}} 
	\sum_{j=0}^{k-1} 
	\left(
	(n \nu(B_r(x)))^j/j!
	\right)
	e^{-n \nu(B_r(x))}n \nu(dx) 
	\leq c (n r^2)^{k} e^{- \delta_0 nr^2} $$ 
	which tends to zero since $nr^2 \to \infty$ by
	\eqref{e:limr}. Thus in both cases we have that
	  \begin{align}
		  \int_{(\partial A)^{(s)} } 
	  \sum_{j=0}^{k-1} ((n \nu(B_s(x)))^j/j!)
	\exp(-n \nu(B_r(x)) ) n \nu(dx) \to 0.
		  \label{e:0908d}
	  \end{align}
Therefore using (\ref{e:Eeq}) and  \eqref{e:limratio}
we obtain that
\begin{align*}
	e^{-\gamma}
	& = \lim_{n \to \infty} \int_{A^{(-s)}}
	\sum_{j=0}^{k-1} ((n \nu(B_s(x)))^j/j!)
	%\exp(-n \nu(B_s(x)))
	e^{-n \nu(B_s(x))}
	n \nu(dx),
	\\
	& = \lim_{n \to \infty} \int_{A^{(-s)}}
	\sum_{j=0}^{k-1} ((n \nu(B_r(x)))^j/j!)
	%\exp(-n \nu(B_s(x))) 
	e^{-n \nu(B_s(x))} 
	n \nu(dx),
\end{align*}
and then taking $n \to \infty$ in (\ref{e:0421a}) we obtain that
$
	e^{-\beta}  \geq
	e^{-\gamma} \limsup_{n \to \infty}
	\left( e^{n\theta_d f_0 (s^d - r^d)} \right),
	$
so that
\begin{align}
	\limsup(n (s^d-r^d)) \leq (\gamma-\beta)/(\theta_d f_0).
	\label{0421c}
\end{align}

For an inequality the other way, let $\eps >0$ and let
$A_\eps := \{x \in A: f(x) \leq f_0 + 4 \eps\}$.
By the assumed continuity of $f$ on $A$,
for all $n$ large
enough, and all $x \in A_\eps$, we have
$
\nu (B_s(x) \setminus B_r(x)) \leq \theta_d (f_0+ 5 \eps) (s^d-r^d).
$
Therefore by \eqref{e:Eeq},
\begin{align}
	e^{-\beta} & \leq 
	\int_{A_\eps }
	\sum_{j=0}^{k-1} ((n \nu(B_r(x)))^j/j!)
	e^{n \theta_d (f_0+ 5 \eps)(s^d-r^d)}
	e^{-n \nu(B_s(x))} n \nu(dx)
	\nonumber
	\\
	& + \int_{A^{(-r)} \setminus A_\eps} 
	\sum_{j=0}^{k-1} ((n \nu(B_r(x)))^j/j!)
	e^{-n\nu(B_r(x))} n \nu(dx) 
	\nonumber
	\\
	& + \int_{(\partial A)^{(r)} \setminus A_\eps}
	\sum_{j=0}^{k-1} ((n \nu(B_r(x)))^j/j!)
	e^{-n\nu(B_r(x))} n \nu(dx). 
	\label{e:0421b}
\end{align}
The third integral on the right tends to zero by
%as at 
\eqref{e:0908d}. 
For $n$ large enough, and all $x \in A^{(-r)} \setminus A_\eps$,
using \eqref{e:0908c}
we have $n \nu(B_r(x)) \geq n (f_0+3 \eps) \theta_d r^d \geq (f_0+ 3 \eps)
(\log n)/(f_0+ \eps)$, and hence the second integral in \eqref{e:0421b}
tends to zero.  Therefore
\begin{align*}
	\liminf\left( 
	\int_{A_\eps } 
	\sum_{j=0}^{k-1} ((n \nu(B_r(x)))^j/j!)
	e^{n \theta_d (f_0+ 5 \eps)(s^d-r^d)}
	e^{-n \nu(B_s(x))} n \nu(dx) \right)
	\geq e^{-\beta}.
\end{align*}
Also, by \eqref{e:limratio}
the second and third integrals in \eqref{e:0421b}
still tend to zero if we change $B_r(x)$ to $B_s(x)$,
so
\begin{align*}
	e^{-\gamma} 
	= 
	\lim \int_{A_\eps}
	\sum_{j=0}^{k-1} ((n \nu(B_s(x)))^j/j!)
	e^{-n \nu(B_s(x))} n \nu(dx).
\end{align*}
Hence, using \eqref{e:limratio} again we obtain that
\begin{align*}
	%\\ &
	e^{-\gamma} 
	= \lim \int_{A_\eps}
	\sum_{j=0}^{k-1} ((n \nu(B_r(x)))^j/j!)
	e^{-n \nu(B_s(x))} n \nu(dx).
\end{align*}
Hence $\liminf(e^{n \theta_d(f_0+ 5 \eps) (s^d -r^d) }) \times e^{-\gamma}
\geq e^{-\beta}$, so that
$$
\liminf (n(s^d -r^d)) \geq (\gamma- \beta)/(\theta_d (f_0+ 5 \eps)).
$$
Since $\eps >0$ is arbitrary, combining this
with (\ref{0421c})
yields (\ref{e:0422a}).
\end{proof}
Next we show that \eqref{e:0422a2} holds for a different $\alpha$
in the case where $f_1 < f_0(2-2/d)$.
%the inequality involving $f_0$ and $f_1$ is reversed.
\begin{lemma}
	\label{l:bdybetas}
	Let $\beta, \gamma \in \R$ with $\beta < \gamma$.
%	$-\infty < \beta < \gamma < \infty$
%and define $r_n := r_n(\beta),$ $ s_n = r_n(\gamma)$ using
%and define $ r_n(\beta)$ 
	%$ s_n = r_n(\gamma)$ 
%	using \eqref{e:Eeq}.
	If $f_1 < f_0(2-2/d)$, then
\begin{align}
	\lim_{n \to \infty}(n (r_n(\gamma)^d -r_n(\beta)^d)) =
	%\frac{2 (\gamma - \beta)}{ \theta_d f_1}.
	2 (\gamma - \beta)/( \theta_d f_1).
\label{e:0422g} 
\end{align}
	Also, if $f_1 = f_0(2-2/d)$, then
\begin{align}
	\limsup_{n \to \infty}(n (r_n(\gamma)^d -r_n(\beta)^d)) \leq
	%\frac{2 (\gamma - \beta)}{ \theta_d f_1}.
	2 (\gamma - \beta)/ (\theta_d f_1).
\label{e:0427a} 
\end{align}
\end{lemma}
\begin{proof}
	%By \eqref{e:SLLN_smooth}, there exists $\delta >0$ such that
%\begin{align}
%\lim_{n \to \infty} \frac{n \theta_d r_n^d}{\log n} = \frac{2-2/d}{f_1}
%> \frac{1+ \delta}{f_0}.
%	\label{e:0422c}
%\end{align}
%Hence for $n$ large
%\begin{align}
%%\int_{A^{(-r)} } e^{- n \nu(B_r(x)) } n \nu(dx)
%\leq n e^{-n \theta_d f_0 r^d} = n e^{-(1+\delta) \log n}
%	\label{e:0422e}
%\end{align}
%which tends to zero. Therefore also $\int_{A^{(-r)}} e^{-n \nu(B_s(x))}
%n\nu(dx) \to 0$ as $n \to \infty$, and hence, as $n \to \infty$,
%\begin{align}
%\int_{(\partial A)^{(r)}} e^{-n \nu(B_s(x))} n \nu(dx) \to e^{-\gamma}  
%	\label{e:0423b}
%\\
%\int_{(\partial A)^{(r)}} e^{-n \nu(B_r(x))} n \nu(dx) \to  e^{-\beta}.
%	\label{e:0423c}
%\end{align}
	Assume $f_1 \leq f_0(2-2/d)$.
	For each $n$
	set $r := r_n(\beta)$, $s:= r_n(\gamma)$.
Suppose $\partial A \in C^2$.
	By Lemma \ref{lemgeom1a}-(iii),
	given $\eps >0$, for $n $ large enough and all
	$x \in (\partial A)^{(-s)}$,
we have $\nu(B_s(x) \setminus B_r(x)) \geq (f_1 -\eps) \theta_d
	(s^d - r^d)/2$.
	%{\bf [For this we need something like
	%Lemma \ref{lemgeom3}-(ii) with the inequality the other way]}.
	Also $\nu(B_s(x) \setminus B_r(x)) \geq nf_0 \theta_d
	(s^d-r^d)$ for $x \in A^{(-s)}$. Hence for $n$ large enough
	%Therefore
\begin{align}
	e^{-\beta}  
	%& = 
	%\int_{A} \sum_{j=0}^{k-1} ((\nu(B_r(x)))^j/j!)
	%e^{n \nu(B_s(x) \setminus B_r(x))}
	%e^{-n\nu(B_s(x))} n \nu(dx) 
	%\nonumber \\
	& \geq  
	e^{n (f_1-\eps) \theta_d(s^d -r^d)/2}
	\int_{(\partial A)^{(s)}}
	\sum_{j=0}^{k-1} ((\nu(B_r(x)))^j/j!)
	e^{-n \nu(B_s(x))} n \nu(dx) 
	\nonumber \\
	& ~~~~~~ +
	e^{n f_0 \theta_d (s^d-r^d)} 
	\int_{A^{(-s)}} 
	\sum_{j=0}^{k-1} ((\nu(B_r(x)))^j/j!)
	e^{-n \nu(B_s(x))} n \nu(dx) 
	\nonumber \\
	& \geq  
	e^{n (f_1-\eps) \theta_d(s^d -r^d)/2}
	\Big(
	\int_{(\partial A)^{(s)}}
	\sum_{j=0}^{k-1} ((\nu(B_r(x)))^j/j!)
	e^{-n \nu(B_s(x))} n \nu(dx) 
	\nonumber \\
	& ~~~~~~ +
	\int_{A^{(-s)}} 
	\sum_{j=0}^{k-1} ((\nu(B_r(x)))^j/j!)
	e^{-n \nu(B_s(x))} n \nu(dx) \Big), 
	\label{e:0427c}
\end{align}
since the assumption $f_1 \leq f_0(2-2/d)$ implies $f_0 \geq f_1/2$.
Hence, 
for $n$ large enough,
setting $\psi_n:=  \inf_{x \in A} \big( \nu(B_r(x))/\nu(B_s(x)) \big)$ 
we have
$$
e^{-\beta} \geq e^{n (f_1-\eps) \theta_d(s^d-r^d)/2} 
	e^{-\gamma} \psi_n^{k-1}.
	%\Big( \inf_{x \in A} \nu(B_r(x))/\nu(B_s(x)) \Big)^{k-1},
	$$
By \eqref{e:limratio} we have $\psi_n \to 1$ as $n \to \infty$,
and thus
\begin{align}
\limsup(n(s^d -r^d)) \leq 2 (\gamma -\beta)/( \theta_d (f_1- \eps)).
	\label{e:0427b} 
\end{align}

In the other case with $d=2$ and $A$ polygonal, 
	on choosing a suitable large $K$ (dependent
	on the smallest angle of $A$) we can obtain, similarly
	to (\ref{e:0427c}), that
	\begin{align}
	e^{-\beta} 
		& 	\geq
		e^{n (f_1-\eps) \pi (s^2 -r^2)/2}
		\psi_n^{k-1}
		\int_{A \setminus \Cor^{(Kr)} }
		\sum_{j=0}^{k-1} ((n \nu(B_s(x)))^j/j!)
	e^{-n \nu(B_s(x))} n \nu(dx) 
		\nonumber \\
		& =
		e^{n (f_1-\eps) \pi (s^2 -r^2)/2}
		\psi_n^{k-1}
		\left( e^{-\gamma} -
		\int_{\Cor^{(Kr)}}
		\sum_{j=0}^{k-1} ((n \nu(B_s(x)))^j/j!)
		e^{-n \nu(B_s(x))} n \nu(dx) \right) . 
		\label{e:0427d}
	\end{align}
	%Also there exists $\delta_1 >0$ that
	%for all $n$ large and all $x \in A$ we have
	%$\nu(B_s(x) \setminus B_r(x)) \geq \delta_1(s^2-r^2)$,
	%and hence
	%$$
	%e^{-\beta} \geq \int_{A} e^{n \delta_1 (s^2 -r^2)}
	%e^{-n \nu(B_s(x))} n \nu(dx) = e^{n\delta_1(s^2 -r^2) }
	%e^{-\gamma},
	%$$
	%so that $n(s^2-r^2) \leq (\gamma- \beta)/\delta_1$, so
	%the first factor in the right hand side of (\ref{e:0427d})
	%remains bounded as $n \to \infty$. 
	Since the integral over $\Cor^{(Kr)}$ tends to zero
	by \eqref{e:ballbds},
%	in the second term of the second factor tends to zero,
	we therefore obtain from (\ref{e:0427d})
	that  \eqref{e:0427b} holds in this case too.
%
%we can derive the
	%same conclusion 
	%and integrating over
	%$(\partial A)^{(r)} \setminus \Cor_{Kr}$ instead of
	%$(\partial A)^{(r)}$. In this case (\ref{e:0423b}) and
	%(\ref{e:0423c}) still hold with the integrals restricted
	%to $(\partial A)^{(r)} \setminus \Cor_{Kr}$, since
	%$\int_{\Cor_{Kr}} e^{-n\nu(B_r(x))}n \nu(dx) \to 0$.
%
Taking $\eps \downarrow 0$, we deduce in both cases that
\eqref{e:0427a} holds whenever $f_1 \leq f_0(2-2/d)$.
%\begin{align}
%\limsup_{n \to \infty} (n(s^d-r^d)) \leq \frac{2(\gamma - \beta)}{f_1 \theta_d}.
%	\label{e:0422b}
%\end{align}

Now assume the strict inequality 
$f_1 < f_0(2-2/d)$. Again set $r=r_n(\beta)$, $s= r_n(\gamma)$.
	Using  \eqref{e:limr} from Lemma \ref{l:rfirstorder},
	choose $\delta_2 >0$ such that
\begin{align}
\lim_{n \to \infty} \frac{n \theta_d r^d}{\log n} = \frac{2-2/d}{f_1}
> \frac{1+ 2 \delta_2}{f_0}.
	\label{e:0422c}
\end{align}
Then there exists a constant $c$ such that 
for $n$ large,
\begin{align}
	\int_{A^{(-r)} } \sum_{j=0}^{k-1}
	((n \nu(B_r(x)))^j/j!)
	e^{- n \nu(B_r(x)) } n \nu(dx)
	& \leq c n (\log n)^{k-1}  e^{-n \theta_d f_0 r^d}
\nonumber	\\
	& \leq n c( \log n)^{k-1}
	e^{-(1+\delta_2) \log n},
	\label{e:0422e}
\end{align}
which tends to zero. 

Take a new $\eps >0$. Let $A_\eps := \{x \in A: f(x) \leq f_1 + 3 \eps \}$.
Using Lemma \ref{lemgeom3}-(ii) in the case $\partial A \in C^2$,
take $\delta > 0$ such that for all large enough $n$ and
all $x \in (\partial A)^{(\delta r)}$,
%LATER add the next bit if non convex.
%(or in the non-convex polygonal case, all
%$x \in (\partial A)^{(\delta r)} \setminus \Cor^{(Kr)}$ for
%a suitable choice of $K$ {\bf [It's all $x \in
%\partial A^{(\delta r)}$ for convex polygons!]}) 
we have
$$
|A \cap B_s(x) \setminus  B_r(x)|
\leq \theta_d (s^d -r^d) (f_1 + 5 \eps)/(2(f_1+ 4 \eps)). 
$$
Such $\delta$ can also be found in the other case  where $A$ is a convex polygon.

Then for $n$ large and $x \in (\partial A)^{(\delta r)} \cap A_\eps$ we
have $\sup_{B_s(x) \cap A} f \leq f_1 + 4 \eps$, and hence  
%for $\partial A \in C^2$ or for $A$ a convex polygon
\begin{align}
	\nu(B_s(x) \setminus B_r(x))
	\leq \theta_d (s^d -r^d) (f_1+ 5 \eps)/2,  ~~~~~~ \forall ~ 
x \in (\partial A)^{(\delta r)} \cap A_\eps.
	\label{e:0422d}
\end{align}
%LATER: maybe add the following for non-convex case:
%In the {\bf [non-convex]}
%polygonal case this inequality holds for all
%$x \in A_\eps \cap (\partial A)^{(\delta r)} \setminus \Cor^{(Kr)}$.

Using Lemma \ref{lemgeom1a}-(i),
in the case where $\partial A \in C^2$ we have
for large $n$, and all $x \in A \setminus A_\eps$,
that  $\nu (B_r(x)) \geq ( f_1 + 2 \eps) \theta_dr^d/2$. 
Hence for $n$ large
\begin{align*}
	\int_{(\partial A)^{(r)} \setminus A_\eps}
	\sum_{j=0}^{k-1} ( ( n \nu(B_r(x)))^j/j!)
	e^{-n \nu(B_r(x))} n \nu(dx)
	 \leq c (\log n )^k n r \exp(-n (f_1+2 \eps) \theta_d r^d/2)
	\\
	 \leq c n^{1-1/d} (\log n)^{k+1/d} \exp(-(f_1+ \eps) (1-1/d)
	 (\log n)/f_1)
	\to 0,
\end{align*}
where for the second inequality we have used the equality in (\ref{e:0422c}).
%LATER: maybe add the next bit.
%In the {\bf [non-convex]}
%polygonal case we can draw the same conclusion
%for the integral over $(\partial A)^{(r)} \setminus (A_\eps \cup \Cor^{(Kr)})$.

By Lemma \ref{lemgeom3}-(i),
there is a constant $\delta_1 >0$ such that for $x \in (\partial A)^{(r)} 
\setminus (\partial A)^{(\delta r)}$
(in the case where $\partial A \in C^2$) or for
$x \in (\partial A)^{(r)} \setminus (\partial A)^{(\delta r)}
\setminus \Cor^{(Kr)}$ (in the
case where $A$ is polygonal)
we have $\nu(B_r(x)) \geq 
(f_1 + 2\delta_1) \theta_d r^d/2$. 
%{\bf (need to prove this somewhere)}. 
Thus if $\partial A \in C^2$  then 
\begin{align*}
	\int_{(\partial A)^{(r)} \setminus (\partial A)^{(\delta r)}}
	p_{n,r}(x)
%	\sum_{j=0}^{k-1} ((n \nu(B_r(x)))^j/j!) e^{-n \nu(B_r(x))} 
	n \nu(dx) 
%	\\
	\leq
	c (\log n)^k n r \exp(-n (f_1 + 2 \delta_1) \theta_d r^d/2),
\end{align*}
which tends to zero by \eqref{e:0422c}.
%LATER: maybe add the following
 In the polygonal case we get the same conclusion
using also the fact that the integral over $\Cor^{(Kr)}$ tends to zero.
Combining the last two estimates with
(\ref{e:0422e}) shows that
\begin{align}
\int_{A \setminus ((\partial A)^{(\delta r)} \cap A_\eps) }
	\sum_{j=0}^{k-1} ((n \nu(B_r(x)))^j/j!)
e^{-n \nu(B_r(x))} n \nu(dx) \to 0,
	\label{e:0422f}
\end{align}
and therefore using \eqref{e:Eeq} followed by (\ref{e:0422d}) we have
\begin{align}
	e^{-\beta} & = \lim_{n \to \infty}
	\int_{(\partial A)^{(\delta r)} \cap A_\eps}
	\sum_{j=0}^{k-1} ((n \nu(B_r(x)))^j/j!)
	e^{-n \nu(B_r(x))} n \nu(dx)
\nonumber \\
	& \leq \liminf_{n \to \infty}
	\left( e^{n \theta_d (s^d -r^d) (f_1 + 5 \eps)/2}
	\int_{(\partial A)^{(\delta r)} \cap A_\eps}
	\sum_{j=0}^{k-1} ((n \nu(B_r(x)))^j/j!)
	e^{-n \nu(B_s(x))} n \nu(dx) \right).
	\label{e:0908e}
\end{align}
By \eqref{e:0422f} 
%and \eqref{e:limr}
we have $
\int_{A \setminus ((\partial A)^{(\delta r)} \cap A_\eps) }
	\sum_{j=0}^{k-1} ((n \nu(B_r(x)))^j/j!)
e^{-n \nu(B_s(x))} n \nu(dx) \to 0$, so 
using \eqref{e:limratio} from Lemma \ref{l:rfirstorder}
we have
$$
\int_{(\partial A)^{(\delta r)} \cap A_\eps }
	\sum_{j=0}^{k-1} ((n \nu(B_r(x)))^j/j!)
e^{-n \nu(B_s(x))} n \nu(dx) \to e^{-\gamma},
$$
and hence by \eqref{e:0908e}, 
$$
e^{-\beta} \leq \liminf_{n \to \infty} \left( e^{n \theta_d(s^d-r^d)
(f_1+ 5 \eps)/2} \right) \times e^{-\gamma},
$$
so that 
$$
\liminf_{n \to \infty}(n (s^d -r^d)) \geq 2 (\gamma - \beta)/ (\theta_d
(f_1+ 5 \eps)).
$$
Taking $\eps \downarrow 0$ and combining with (\ref{e:0427a}) shows 
that  (\ref{e:0422g}) holds. 
\end{proof}

\begin{lemma}[De-Poissonization]
	\label{l:dePo}
	Let $\beta \in \R$ and suppose $r_n = r_n(\beta)$ is given by 
	\eqref{e:Eeq} for $n$ sufficiently large. Then
	\begin{align}
		\lim_{n \to \infty} \PP[L_{k}(\cX_n) \leq r_n(\beta)]
		=
	\lim_{n \to \infty} \PP[L_{n,k} \leq r_n(\beta)] = \exp(-e^{-\beta}).
		 \label{e:Lklim}
	\end{align}
\end{lemma}
\begin{proof}
	The statement about $L_{n,k}$ in \eqref{e:Lklim} follows from
	 Proposition \ref{p:poisson2d}. It remains to prove the statement
	 about $L_k(\cX_n)$.

	 Given $n >0,r >0$ define $\phi_{n,r} := \EE[\xxi_{n,r}]$,
	 i.e. by \eqref{e:EEF},
	 %$I_{n,r}$ (nothing to do with $I_i(n,r)$ in Lemma
	 %\ref{l:Stein} earlier) by
	 \begin{align}
	 \phi_{n,r} := \int_A p_{n,r}(x) n \nu(dx)
	 = \int_A \sum_{j=0}^{k-1} ((n \nu(B_r(x)))^j/j!)
	 \exp(-n \nu(B_r(x)) ) n \nu(dx),
		 \label{e:Inrdef}
	 \end{align}
	 which is decreasing in $r$.
	 Set $n^-:= n- n^{3/4}$,
	  $n^+:= n + n^{3/4}$,
	 and let $\beta \in \R$. Then
	 $$
	 \frac{\phi_{n^-,r_n(\beta)}}{\phi_{n,r_n(\beta)}}\geq \big( \frac{n^{-}}{n}
	 \big)^{k-1}= 1+ O(n^{-1/4}),
	 $$
	 and 
	 $$ \frac{\phi_{n^-,r_n(\beta)}}{\phi_{n,r_n(\beta)}}
	 \leq \exp(n^{3/4} \fmax \theta_d r_n(\beta)^d)
	 =1 + O((\log n) n^{-1/4}),
	 $$
	 so that $\phi_{n^-,r_n(\beta)}/\phi_{n,r_n(\beta)} \to 1$ as $n \to \infty$,
	 and thus $\phi_{n^-,r_n(\beta)} \to e^{-\beta}$ as
	 $n \to \infty$.
	 Therefore 
	 using Proposition
	\ref{p:poisson2d}, we have 
	 \begin{align}
		 \PP[L_{n^-,k} \leq r_n(\beta) ] \to \exp (-e^{-\beta}).
		 \label{e:n-vn}
	 \end{align}

	 Now, following the proof of \cite[Theorem 8.1]{Pen},
	 we note that with $\eta_{n^-}$, $\eta_{n^+}$ and $\cX_n$ coupled
	 in the usual way (as described in \cite{Pen}), we have
	 $$
	 \{L_{n^-,k} \leq r_n(\beta) \} \triangle 
	 \{L_k(\cX_n) \leq r_n(\beta) \} \subset E_n \cup F_n \cup G_n
	 $$
	 where, setting
	 % $n^+: = n+ n^{3/4}$, and
	  $N_t = \eta_t(A)$ for all $t$,
	  we set
	 \begin{align*}
		 E_n &:= \{ \exists x \in \eta_{n^+} \setminus \eta_{n^-}:
		 \eta_{n^-}(B_{r_n(\beta)}(x)) \leq k-1\};
		 \\
		 F_n  &:= \{ \exists
		 x \in \eta_{n^-},
		 y \in \eta_{n^+} \setminus \eta_{n^-}
		 : \eta_{n^-}(B_{r_n(\beta)}(x)) \leq k,
		 \| y -x \| \leq r_n(\beta)\};
		 \\
		 G_n  &:= \{ N_{n^-} \leq n \leq N_{n^+} \}^c.
	 \end{align*}

	 By Chebyshev's inequality $\PP[G_n] = O(n^{-1/2})$.
	 By Markov's inequality,
	 \begin{align*}
		 \PP[E_n] & \leq 2n^{3/4} \int_A \PP[\Po_{n^-}(B_{r_n(\beta)}(x))
	 \leq k-1] \nu(dx) \\
		 & = (2 n^{3/4}/n^-)
		 \phi_{n^-,r_n(\beta)},
	 \end{align*}
	 which tends to zero.
	 Finally, by the Mecke formula,
	 \begin{align*}
		 \PP[F_n] & \leq 2n^{3/4} \fmax \theta_d r_n(\beta)^d 
		 \int_A \PP[\Po_{n^-}(B_{r_n(\beta)}(x))
	 \leq k-1] n^- \nu(dx) \\
		  & = 2 \fmax \theta_d  n^{3/4} r_n(\beta)^d
		 \phi_{n^-,r_n(\beta)}
		 %\\
		  = O((\log n)n^{-1/4}). 
	 \end{align*}
	 Therefore using \eqref{e:n-vn} we obtain
	 \eqref{e:Lklim}.
	 %that
	 %\begin{align}
	 %\PP[L_k(\cX_n) \leq r_n(\beta) ] \to  \exp(- e^{-\beta}).
	%	 \label{e:Lklim}
	% \end{align}
\end{proof}

\begin{proof}[Proof of Theorem \ref{t:nonunif} parts (i) and (ii)]
	For part (i) we assume $f_1 > f_0(2- 2/d)$; in this
	case set $\alpha = 1/(\theta_d f_0)$.
	For part (ii) we assume $f_1 < f_0(2- 2/d)$; in this
	case set $\alpha = 2/(\theta_d f_1)$.

%Combining
	By Lemma \ref{l:betas} in the first case, or by
	Lemma \ref{l:bdybetas} in the second case,
	 for all $\beta,\gamma \in \R$ with
	$\beta < \gamma$ the
	condition \eqref{e:0422a2} holds.
	
	Let $\beta \in \R$ and suppose $r_n = r_n(\beta)$ is given by 
	(\ref{e:Eeq}) for $n$ sufficiently large.  
	By Lemma \ref{l:dePo},
%	Proposition \ref{p:poisson2d},
	$\PP[L_{n,k} \leq r_n(\beta)] \to F(\beta)$
	and $\PP[L_{k}(\cX_n) \leq r_n(\beta)] \to F(\beta)$,
	where we set $F(x) := \exp(-e^{-x})$.
	By Proposition \ref{p:L<M},
	$\PP[M_{n,k} \leq  r_n(\beta)] \to F(\beta)$ as $n \to \infty$.

	Then by Lemma \ref{l:Gum} (taking $X_n = M_{n,k}$), we obtain 
	 that
	 $n M_{n,k}^d - n \mu(M_{n,k})^d \tod \alpha(\Gum + \log \log 2),
	 $
	 i.e. \eqref{e:limnonu1} holds if $f_1 > f_0(2-2/d)$,
	 and \eqref{e:limnonu5} holds if $f_1 < f_0(2-2/d)$.
	 Also by taking $X_{n}= L_{n,k}$ in Lemma \ref{l:Gum} we obtain
	 \eqref{e:limnonu3} if $f_1 > f_0(2-2/d)$, and
	 \eqref{e:limnonu7} if $f_1 < f_0(2-2/d)$.
	 %Lemma \ref{l:Gum}, we have that
	 %$n L_{n,k}^d - n \mu_{n,k}^d \tod \alpha(\Gum + \log (\log 2)).
	 %$

	 It remains to demonstrate the results for the binomial model,
	 i.e. \eqref{e:limnonu2},
	 \eqref{e:limnonu4}, \eqref{e:limnonu6}
	 and \eqref{e:limnonu8}.
	 By Lemma \ref{l:dePo}, we can
	 use Lemma \ref{l:Gum} (now taking $X_n = L_{k} (\cX_n)$)
	 to deduce  that \eqref{e:limnonu4} holds if
	 $f_1 > f_0(2-2/d)$, and \eqref{e:limnonu8} holds
	 if  
	 $f_1 < f_0(2-2/d)$.

	 Using \eqref{e:Lklim}, and \eqref{e:0414abin} from Proposition
	\ref{p:L<M}, we obtain that
	 $$
	 \PP[M_k(\cX_n) \leq r_n(\beta) ] \to  \exp(- e^{-\beta}).
	 $$
	 Then using Lemma \ref{l:Gum} (now taking $X_n = M_{k}(\cX_n)$)
	 we can deduce that  \eqref{e:limnonu2} holds
	if $f_1 > f_0(2-2/d)$, and \eqref{e:limnonu6} holds
	if $f_1 < f_0(2-2/d)$.
\end{proof}

\subsection{Proof of Theorem \ref{t:nonunif}: conclusion}

It remains to prove part (iii) of Theorem \ref{t:nonunif}. We deal
first with the assertions there concerning
 tightness. Again in the next proof, set $F(x):= \exp(-e^{-x})$, $x \in \R$.
\begin{lemma}
	\label{l:tight}
	The collection of random variables 
	$\{nM_{n,k}^d - n \mu(M_{n,k})^d\}_{n \geq 1}$
	is tight.
	So is the collection of random variables 
	$\{nL_{n,k}^d - n \mu(L_{n,k})^d\}_{n \geq 1}$,
	and also the sequence $( nM_{k}(\cX_n)^d - n \mu(M_{k}(\cX_n))^d)_{n
	\in \NN}$, 
	and the
	sequence $( nL_{k}(\cX_n)^d - n \mu(L_{k}(\cX_n))^d)_{n \in \NN}$.
\end{lemma}
\begin{proof}
Let $\eps \in (0,1/6)$.
Choose $\beta < \beta'$ with $F(\beta) < \eps/3$ and $F(\beta') > 1- \eps/3$.
Set $r_n = r_n(\beta)$, $s_n = r_n(\beta')$
as given by (\ref{e:Eeq}).
	By Lemma \ref{l:bdybetas}
%	(\ref{e:0423a})
	there
exists a constant $K$ such that $n(s_n^d -r_n^d) \leq K$ for all large
enough $n$.  By 
	Proposition \ref{p:poisson2d},
	$\PP[L_{n,k} \leq r_n(\beta)] \to
	F(\beta)$. 
	By Proposition \ref{p:L<M},
	$\PP[M_{n,k} \leq r_n  ] \to F(\beta)$ as $n \to \infty$. 
	Similarly,
	$\PP[L_{n,k} \leq   s_n] \to F(\beta')$
	and
	$\PP[M_{n,k} \leq   s_n] \to F(\beta')$
	as $n \to \infty$. 
	Therefore since $F(\beta) < 1/2 < F(\beta')$, we have
	$r_n \leq  \mu(L_{n,k})  < s_n$
	and 
	$r_n \leq  \mu(M_{n,k})  < s_n$
	for $ n$ large.
	Then for $n$ large
	\begin{align*}
		\PP[n M_{n,k}^d \leq n \mu(M_{n,k})^d -K ]
		& \leq \PP[ n M_{n,k}^d \leq n s_n^d - K]
		\\
		& \leq \PP[M_{n,k} \leq r_n] < \eps/2,
	\end{align*}
	and likewise for $L_{n,k}$.
	Similarly for $n$ large
	\begin{align*}
		\PP[n M_{n,k}^d > n \mu(M_{n,k})^d +K ]
		& \leq \PP[ n M_{n,k}^d > n r_n^d + K]
		\\
		& \leq \PP[M_{n,k} > s_n] < \eps/2,
	\end{align*}
	and likewise for $L_{n,k}$.
	Thus $\PP[|n(M_{n,k}^d- \mu(M_{n,k})^d)| > K] \leq \eps$
	and $\PP[|n(L_{n,k}^d- \mu(L_{n,k})^d)| > K] \leq \eps$
	for all large enough $n$, 
	Also $\{ n(M_{n,k}^d-\mu(M_{n,k}^d))\}_{1 \leq  n \leq n_0}$
	and $\{ n(L_{n,k}^d-\mu(L_{n,k}^d))\}_{1 \leq n \leq n_0}$
	are uniformly bounded for any fixed $n_0 \in (0,\infty)$.
	This yields the asserted tightness of $(M_{n,k})_{n \geq 1}$
	and of $(L_{n,k})_{n \geq 1}$.

	The proof of tightness for $L_{k}(\cX_n)$ and
	of $M_k(\cX_n)$ is similar,
	except that instead of 
	Proposition \ref{p:poisson2d} we use \eqref{e:Lklim}.
	 Proposition \ref{p:L<M} still applies in the binomial setting.
\end{proof}

To prove \eqref{e:LM} we shall adapt
the `squeezing argument' from \cite{Pen97}.
For $- \infty < \beta < \gamma < \infty$ we define
the random variable
\begin{align}
U_n(\beta,\gamma) := \sum_{x \in \cX_n}
		\1 \{ \cX_n ( B_{r_n(\beta)}(x) ) \leq k , 
		\cX_n(  B_{r_n(\gamma)}(x) \setminus
		B_{r_n(\beta)}(x)) \geq 2  \}.
		\label{e:defunbg}
\end{align}
	\begin{lemma}
		\label{l:squeeze}
	Let $K >0$.
		Then there is a constant $c \in (0,\infty)$ such
		that for all $\beta, \gamma \in \R$ with
		 $-K \leq \beta < \gamma \leq K$,
	 $$
		\limsup_{n \to \infty}
		\PP[ U_n(\beta,\gamma) \geq 1 ] \leq c (\gamma-\beta)^2. 
	$$

\end{lemma}
	\begin{proof} Set $r_n:= r_n(\beta)$, and
		$s_n := r_n(\gamma)$. By the union bound,
		\begin{align}
	\PP[U_n(\beta,\gamma) \geq 1]
		%	& 
			\leq
			%\EE[U_n(\beta,\gamma)]
		%	\nonumber
	%\\ & =
			n \int_A \PP[ \cX_{n-1} ( B_{r_n}(x))
			< k ,  \cX_{n-1} ( B_{s_n}(x) \setminus B_{r_n}(x) )
			\geq 2] \nu(dx).
			%\nonumber \\
			%& 
			%\leq \sup_{x \in A}
			%(n \nu(B_{s_n}(x) \setminus B_{r_n}(x)))^2
			%\times \int_{A}
			%\sum_{j=0}^{k-1}
			%((n \nu(B_{r_n}(x) ))^j/j!)
			%e^{-n \nu(B_{r_n}(x))} n\nu(dx)
	%\nonumber \\
	%& = \sup_{x \in A} (n \nu(B_{s_n}(x) \setminus B_{r_n}(x)))^2
	%		\times e^{-\beta}.
			\label{e:annulus}
\end{align}

		Let $x \in A$ and set $Y := \cX_{n-1}(B_{r_n}(x))$,
		$Z : = \cX_{n-1}(B_{s_n}(x) \setminus
		B_{r_n}(x))$.
		Also set $v_n(x):= \nu(B_{r_n}(x))$ and
		 $w_n(x):= \nu(B_{s_n}(x) )$. Then
		\begin{align*}
			\PP[Y < k] & = \sum_{j=0}^{k-1}
			\binom{n-1}{j} v_n(x)^j(1-v_n(x))^{n-1-j}
			\\
			&	\leq (1-v_n(x))^{-k} \sum_{j=0}^{k-1}
			((n v_n(x))^j/j!)  (1-v_n(x))^n
			\\
			& \leq (1- \fmax \theta_d r_n^d)^{-k} 
			 \sum_{j=0}^{k-1}
			((n v_n(x))^j/j!)  e^{-nv_n(x)}.
		\end{align*}
Also, using the fact that $\PP[Z \geq 2|Y=j]$ is nonincreasing in $j$,
and  the fact that for any binomial
		 random variable $W$ with mean $\alpha$
		 we have $\EE[W(W-1)] \leq \alpha^2$,
		 we have
		\begin{align*}
			\PP[Z \geq 2|Y <k] & \leq
			\PP[ Z \geq 2|Y=0]
			\\
			& \leq (1/2) \EE[Z(Z-1)|Y=0]
			\\
			& \leq (1/2)  n^2 ((w_n(x)-v_n(x))/(1-v_n(x)))^2
			\\
			& \leq (1/2)  n^2 (1- \fmax \theta_d  r_n(\beta)^d)^{-2}
			(w_n(x) -v_n(x))^2.
			%(\nu(B_{s_n}(x) )
			%-\nu(B_{r_n}(x) )^2.
		\end{align*}
		If $n$ is taken to be large enough we have
		$(1- \fmax \theta_d r_n(\beta)^d)^{-k-2} \leq 2$, and
		hence 
		using \eqref{e:annulus} 
		followed by \eqref{e:Eeq},
		we have
		\begin{align}
	\PP[U_n(\beta,\gamma) \geq 1]
			& \leq n^3
			\sup_{y \in A} (w_n(y) - v_n(y))^2
			\int_A \sum_{j=0}^{k-1}
			((nv_n(x))^j/j!) e^{-n v_n(x)}
			\nu(dx) \nonumber \\
			& 
			=
			\big( n \sup_{y \in A} (w_n(y)-v_n(y)) \big)^2
			e^{-\beta} .
			\label{e:anninter}
		\end{align}
		
		By Lemmas \ref{l:betas} and \ref{l:bdybetas},
		\begin{align}
			\limsup_{n \to \infty}
			n \theta_d (s_n^d- r_n^d) \leq \left( \frac{1}{f_0}
			\vee \frac{2}{f_1} \right)
			%(f_0 \delta_2)^{-1}
			(\gamma-\beta),
			\label{e:0423a}
		\end{align}
		where $\vee$ denotes maximum, and hence 
		\begin{align*}
			\limsup_{n \to \infty}
			\sup_{x \in A}
			%n \nu(B_{s_n}(x) \setminus B_{r_n}(x))
			n (w_n(x) - v_{n}(x))
			\leq \limsup_{n \to \infty}
			n \fmax \theta_d(s_n^d - r_n^d) \leq
		\fmax \left( \frac{1}{f_0} \vee \frac{2}{f_1}\right)
	 (\gamma - \beta),
		\end{align*}
		so by (\ref{e:anninter}) we have
		$\limsup_{n \to \infty}
		\PP[U_n(\beta,\gamma) \geq 1] \leq  e^K (\fmax
		(\frac{1}{f_0} \vee \frac{2}{f_1}))^2 (\gamma -\beta)^2$.
\end{proof}

\begin{proof}[Proof of Theorem \ref{t:nonunif} (conclusion)]
	
	It remains to prove part (iii), and  by Lemma \ref{l:tight} it
	remains only to prove
	\eqref{e:LM}.
	Let $\eps > 0$. Choose  $K >0$
	such that $\exp(- e^{-K}) > 1- \eps$, and also $\exp(-e^{K}) < \eps$.
	Then let
	$c$ be as in Lemma \ref{l:squeeze}.
	Choose
	$\beta_0 < \ldots < \beta_m$ with $\beta_0=-K$ and $\beta_m= K$ 
	such that
	$c \sum_{i=1}^m (\beta_i - \beta_{i-1})^2 < \eps$.
	Write $L'_{n,k}$ for $L_k(\cX_n)$ and $M'_{n,k}$
	for $M_k(\cX_n)$.
	Since $L'_{n,k} \leq M'_{n,k}$, by the union bound
	\begin{align*}
		\PP[L'_{n,k} \neq M'_{n,k}] & = \PP[L'_{n,k} < M'_{n,k}]
		\\
		& \leq  \PP[L'_{n,k} \leq r_n(\beta_0)]
		+ \PP[ L'_{n,k} > r_n(\beta_m)] + 
		\sum_{i=1}^m \PP[L'_{n,k} \leq r_n(\beta_i) < M'_{n,k} ]
		\\
		& + \sum_{i=1}^m \PP[r_n(\beta_{i-1}) < L'_{n,k} <
		M'_{n,k} \leq r_n(\beta_i)]. 
	\end{align*}
	Using Lemma \ref{l:dePo}
	%for the first two terms,
	%we have $\PP[L'_{n,k} \leq r_n(\beta)] \to
	%\exp(- e^{-\beta})$ as $n \to \infty$.
	%Thus using also
	and Proposition \ref{p:L<M},
	%for the last sum, 
	we obtain that
	\begin{align*}
		\limsup_{n \to \infty} \PP[L'_{n,k} \neq M'_{n,k}] \leq
		2 \eps +
		 \sum_{i=1}^m \limsup_{n \to \infty}
		 \PP[r_n(\beta_{i-1}) < L'_{n,k} < M'_{n,k}
		 \leq r_n(\beta_i)]. 
	\end{align*}
	Suppose $\beta < \gamma$, and suppose
	 $r_n(\beta) < L'_{n,k} < M'_{n,k} \leq r_n(\gamma)$ and
	 all inter-point
	 distances in $\cX_n$ are distinct (the latter condition holds
	 almost surely).

	 Then there exist $x,y \in \cX_n$ with
	 $\|x-y\| =  M'_{n,k}$, and it is possible to remove $k$
	 vertices from $G(\cX_n,M'_{n,k})$ leaving the resulting
	 graph connected, but disconnected if the edge $\{x,y\}$
	 is also removed. Removing the same set of vertices
	 from
	 $G(\cX_n,r_n(\beta))$ leaves $x$ and $y$ in distinct
	 components, and if also for some fixed $\rho >0$,
	 events $H_{r_n(\beta),\rho}(\cX_n)$ (defined in Lemma \ref{l:2large})
	 and $J_{r_n(\beta),\rho}(\cX_n)$ (defined in the proof of
	 Proposition \ref{p:L<M}),
	 %$$
	 %J_{\rho,n}
	 %:= \{\exists w \in \cX_n,  \cY \in \cC_{r_n(\beta)}(w,\cX_n), 0
	 %< \diam (\cY) <\rho r_n(\beta) \}
	 %$$
	 fail to occur, then  $x$ or $y$ must have at most
	 $k-1$ other points of $\cX_n$  within distance
	 $r_n(\beta)$. But 
	 $\cX_n(B_{r_n(\gamma)}(x)\setminus \{x\}) \geq k+1$ 
	 since $L'_{n,k} < \|y-x\| \leq r_n(\gamma)$,
	 and similarly
	 $\cX_n(B_{r_n(\beta)}(y)\setminus \{y\}) \geq k+1$.
	 Thus we have the event inclusion
	 \begin{align*}
		 \{r_n(\beta) < L'_{n,k} < M'_{n,k} \leq r_n(\gamma) \}
		 \subset
		 H_{r_n(\beta),\rho}(\cX_n) \cup J_{r_n(\beta),\rho}(\cX_n) 
		 \cup U_n(\beta,\gamma),
	 \end{align*}
	 where $U_n(\beta,\gamma) $ was defined at
	 \eqref{e:defunbg}.

	 By Lemma \ref{l:2large}, we can choose $\rho $ so that
	 $\PP[H_{r_n(\beta),\rho}(\cX_n)] \to 0$,
	 and by  the proof of Proposition \ref{p:L<M}
	 %LATER:
	 %({\bf binomial version - can we arrange to refer to
	 %a lemma rather than a proof?}),
	 $\PP[J_{r_n(\beta),\rho}(\cX_n)] \to 0$. Therefore
	 using Lemma \ref{l:squeeze}
	 we obtain that
	 \begin{align*}
		 \limsup_{n \to \infty} \PP[ r_n(\beta ) < L'_{n,k} < M'_{n,k}
		 \leq r_n(\gamma) ] \leq \limsup_{n \to \infty}
		 \PP[ U_n(\beta,\gamma) \geq 1] \leq c (\gamma - \beta)^2.
	 \end{align*}
	 Thus
	 \begin{align*}
		 \limsup_{n \to \infty}
		 \PP[L'_{n,k} \neq M'_{n,k}] \leq 2 \eps + \sum_{i=1}^m c
		 (\beta_i - \beta_{i-1})^2 < 3 \eps,
	 \end{align*}
	 and since $\eps >0$ is arbitrary this gives us (\ref{e:LM}).
\end{proof}

\section{Proof of Theorem \ref{t:smooth}}
Now we specialise to the uniform case. We make the same
assumptions on $A$ as in the previous section, but now we take 
$ f \equiv f_0 {\bf 1}_A$ with $f_0 = | A|^{-1}$.
%on $A$.
Recall from \eqref{e:defcd} the definition
%Given $d,k$, define the constant
\begin{align}
c_{d,k} := \theta_{d-1}^{-1}  \theta_d^{1-1/d} (2- 2/d)^{k-2 + 1/d} 2^{1-k}
/(k-1)!
\label{e:defcdk2}
	\end{align}

Given $k \in \NN$ and
$\beta \in \R$, let $r_n=r_n(\beta)\ge 0$ be defined for all $n >0$
by
\begin{equation}
	f_0 n\theta_d r_n^d = \max \big( (2-2/d) \log n +
        ( 2k -4 +2/d ) \1_{\{d\ge 3
	\text{ or } k\ge 2\}} \log\log n + \beta, 0 \big).
        %~~~ n\ge n_0,
        \label{e:rboth}
\end{equation}
%where $n_0$ is chosen sufficiently large that the
%right hand side of (\ref{e:rboth}) is positive for
%all $n \geq n_0$.

%\subsection{Convergence of $\EE[F_{n,k}]$ for $d=2$}

We now show the convergence of $\EE[\xxi_{n,r_n}]$
	(with $\xxi_{n,r}$ defined at \eqref{e:def_isoVer}).
%which justifies the choice of $r_n$ at (\ref{e:rboth}).
That is, we show that this choice of $r_n$ 
satisfies (\ref{e:Econv}) for appropriate $\beta'$.
	Recall the definition of the
	isoperimetric ratio $\sigma_A$ at \eqref{e:defsigA}. 

\begin{proposition}[convergence of the expectation in the uniform case
	with $d=2$]
	\label{p:average2d}
	Suppose $f \equiv f_0{\bf 1}_A$, with $d=2$ and either $A$ compact
	with $C^2$ boundary, or $A$ polygonal.
	Fix $k \in \NN$, $\be\in\RR$,
	and let $r_n, \xxi_{n,r}$ be as given in 
	\eqref{e:rboth}
	and
	\eqref{e:def_isoVer}.
	Then as $n \to \infty$,
	\begin{align}
		\EE[\xxi_{n,r_n}] =  
		\begin{cases}
			e^{-\beta} + \sigma_A e^{-\beta/2} \frac{\sqrt{\pi}}{2} 
			(\log n)^{-1/2}
			%(\frac{1}{\sqrt{\log n}} 
			+ O((\log n)^{-3/2})  & \mbox{ if } k= 1 \\
			e^{-\beta} + \sigma_A e^{-\beta/2}
			\frac{\sqrt{\pi}}{4} \Big(1 +
			%|\partial A|e^{-\beta/2} \frac{\sqrt{f_0\pi}}{2}(
			\frac{\log\log n}{2 \log n} \Big)
			+ \frac{e^{-\beta} \log \log n}{\log n} 
			+ O((\log n)^{-1})  & \mbox{ if } k=2 \\
			\sigma_A e^{-\beta/2} \frac{\sqrt{\pi}}{(k-1)!2^k}
			\left(1+\frac{(2k-3)^2\log\log n 
			%+ (2k-3)\beta + \frac{\pi}{2} + 2k -2
			}{2 \log n}  
			\right)
			%O((\frac{\log \log  n}{\log n})^{2}) 
			+O\big((\log n)^{-1} \big)
			& \mbox{ if } k\ge 3.
		\end{cases}
%		e^{-\be}|A|? +  ?O((\log n)^{-1/2})?.
		\label{e:EFlim}
	\end{align}
\end{proposition}

\begin{proof} 
	Define the `$k$-vacant region'
	$V_{n,k} := \{x \in A: \Po_n(B_{r_n}(x)) < k\}$.
	Recall the definition of
	$p_{n,r}(x)$  at \eqref{e:pnj}.
	By \eqref{e:EEF}, we have
\begin{align}\label{e:E[F_n]}
	\EE[\xxi_{n,r_n}]= n f_0 \int_A p_{n,r_n}(x) dx = n |A|^{-1} 
	\EE[|V_{n,k}|].
\end{align} 
	Therefore the result follows from \cite[Proposition 5.1]{HPY23}.
\end{proof}
%\subsection{Convergence of $\EE[\xxi_{n,k}]$ for $d \geq 3$}

\begin{proposition}[convergence of the expectation in the uniform case with $d \geq 3$] \label{p:average3d+}
	Suppose $f \equiv f_0{\bf 1}_A$, with $d \geq 3$ and $A$ compact
	with $\partial A \in C^2$.
Fix $\be \in\RR$ and let
	$r_n(\beta), \xxi_{n,r}, c_{d,k}, \sigma_A$ be as given in
	\eqref{e:rboth},
	\eqref{e:def_isoVer}
	and 
	\eqref{e:defcdk2}. 
	%\eqref{e:defsigA}. 
	Let $\eps >0$.
	Then as $n \to \infty$,
\begin{align}
	\EE[\xxi_{n,r_n}] & =  e^{-\beta/2} c_{d,k} \sigma_A
	\Big(1 + \frac{ (k-2 + 1/d)^2  
	\log \log n}{(1-1/d) \log n}  
	\nonumber \\ 
	& + \frac{ (k-2 + 1/d)  \beta + 4k -4}{(2-2/d)
	\log n} 
	\Big)
	+ O\big( (\log n)^{\eps -2} \big). 
	\label{e:EFconv}
\end{align}
\end{proposition}
%\begin{proof}[Proof of Proposition \ref{p:average3d+}]
\begin{proof}
	Again using \eqref{e:E[F_n]}, we obtain this
	result from \cite[Proposition 5.2]{HPY23}.
\end{proof}

\begin{corollary}\label{c:nnlink2d}
	Let $d=2$, $\beta \in\RR$. Then  
		\eqref{e:Mnweak} holds, and also
\begin{align}
	\PP[n f_0 \pi L_{n,1}^2 - \log n\le \beta] =
	\exp \big( - \frac{\sigma_A \pi^{1/2} e^{-\beta/2}}{2 (\log n)^{1/2}} \big) 
	\exp(- e^{-\be}) + O( (\log n)^{-1}).
	\label{e:Lwk}
\end{align} 
	Moreover \eqref{e:limu22} holds, and
	\begin{align}
	\PP[n f_0 \pi L_{n,2}^2 - \log n - \log \log n \le \beta] =
		& \exp \Big(- \frac{\sigma_A \pi^{1/2} e^{-\beta/2}
		\log \log n}{ 8 	\log n} 
		- \frac{e^{-\beta} \log \log n}{\log n}
		\Big) 
	\nonumber	\\
		& \times \exp \Big(- e^{-\be} - \frac{\pi^{1/2} \sigma_A 
		e^{-\beta/2}}{ 4} \Big) + O\Big( \frac{1}{\log n} \Big).
	\label{e:Lwk22}
\end{align} 

\end{corollary}
\begin{proof}
	Let $r_n = r_n(\beta)$ be given by \eqref{e:rboth} with
	$d=2, k=1$; then
	$n f_0 \pi r_n^2 - \log n = \beta$
	for all large enough $n$.
	
	Let $\xxi_{n,r}$ be the number of isolated vertices of
	$G(\Po_n,r)$ as defined at \eqref{e:def_isoVer},
	 taking $k =1$.
	 By Proposition \ref{p:average2d}, 
	 \eqref{e:Econv} holds on taking  $\beta' = e^{-\beta}$. 
	Hence by Proposition \ref{p:poisson2d},  
	%Notice that $\PP[nf_0\pi L_n^2 - \log n\le \be] = \PP[ F_n=0]$.
	\begin{align*}
		\PP[nf_0 \pi L_{n,1}^2 - \log n \leq \beta]
		= \PP[ L_{n,1} \leq r_n]
		= \exp(- \EE[\xxi_{n,r_n}]) + 
		O( 1/(\log n)).
		%O\big( \frac{1}{\log n} \big)
	\end{align*}
	Then using
	 Proposition \ref{p:average2d}, 
	and the fact that %$\dtv(\PRV_\lambda,\PRV_{\lambda'}) 
	$|e^{-\la}- e^{-\la'}|
	\leq |\lambda - \lambda'|$
	for any $\lambda, \lambda' >0$,
	we obtain \eqref{e:Lwk}. We can then deduce
	\eqref{e:Mnweak} using Proposition \ref{p:L<M}.

	Next, let $r_n = r_n(\beta)$ be given by \eqref{e:rboth}
	again,
	but  now with $d=2, k=2$.
	 Then 
	$n f_0 \pi r_n^2 - \log n - \log \log n= \beta$
	for $n $ large. Repeating the previous argument
	gives us \eqref{e:Lwk22} and then \eqref{e:limu22}.
\end{proof}

\begin{corollary}\label{c:nnlink3d+}
	Suppose either $d\ge 3$, or $d=2, k \geq 3$. Let $\be\in\RR$.
	%and define  $L_n$ as in Corollary \ref{c:nnlink2d}. 
	Then \eqref{e:limudhi}
		 holds, and 
\begin{align*}
	\PP[n f_0 \theta_d L_{n,k}^d - (2-2/d)\log n
	+ (4-2k -2/d) \log \log n \le \beta] 
	\\
	%= 
	%\exp(-  e^{-\beta/2} c_d \sigma_A)  + O\left(  \frac{\log \log n}{\log n} \right).
	 = \exp \Big( - \frac{c_{d,k} \sigma_A e^{-\beta/2}
	(k-2+1/d)^2 \log \log n}{(1-1/d)\log n} 
	\Big) \exp(- c_{d,k} \sigma_A e^{-\beta/2} ) +
	O\big( \frac{1}{\log n} \big), 
%\\
%\PP[ nf_0 \theta_d M_n^d - (2-\frac{2}{d})( \log n- \log \log n) \le  \be ] =
\end{align*} 
\end{corollary}
\begin{proof}
The proof is the same as for Corollary \ref{c:nnlink2d},
	using Proposition
	\ref{p:average3d+} in place of Proposition \ref{p:average2d}
	when $d \geq 3$.
\end{proof}

We are now ready to finish the proof of Theorem \ref{t:smooth}. 

\begin{proof}[\it Proof of Theorem \ref{t:smooth}]
	We already showed \eqref{e:Mnweak},
	\eqref{e:limu22},
	\eqref{e:limudhi}
	and the corresponding
	results for $L_{n,k}$, in Corollaries 
	\ref{c:nnlink2d} and
	\ref{c:nnlink3d+}. 
	We already proved \eqref{e:LM} under weaker assumptions
	in Theorem \ref{t:nonunif}.
	Therefore it remains only to prove \eqref{e:MXweak} and
	the binomial versions of 
	\eqref{e:limu22} and
	\eqref{e:limudhi}, along with the corresponding results
	for $L_k(\cX_n)$.

	Let $\phi_{n,r}$ be as defined at
	\eqref{e:Inrdef}. Set $n^- := n- n^{3/4}$.
	As shown in the proof of Lemma
	\ref{l:dePo}, 
	%LATER maybe extract as lemma?
	given $\beta \in \R$ we have that
	$$
	\phi_{n^-,r_n(\beta)} = \Big( 1 + O \big( \frac{\log n}{n^{1/4}} \big)
	\Big) \phi_{n,r_n(\beta)}.
	$$
	Then by Proposition \ref{p:poisson2d},
	\begin{align*}
		\PP[L_{n^-,k} \leq r_n(\beta)]
		= \exp(-\phi_{n^-,r_n(\beta)}) + O((\log n)^{-1})
		\\
		= \exp(-\phi_{n,r_n(\beta)}) + O((\log n)^{-1}).
	\end{align*}
	By the proof of Lemma \ref{l:dePo},
	% LATER can we extract as lemma?,
	\begin{align*}
		\PP[L_{k} (\cX_n) \leq r_n(\beta) ] & =
	\PP[L_{n^-,k} \leq r_n(\beta) ]  + O((\log n)/n^{1/4}).
		\\
		& = \exp(-\phi_{n,r_n(\beta)}) + O((\log n)^{-1}).
	\end{align*}
	Plugging in the expressions for $\phi_{n,r_n(\beta)}
	= \EE[\xxi_{n,r_n(\beta)}]$ in Lemmas 
	 \ref{p:average2d} and \ref{p:average3d+} gives us the
	result \eqref{e:MXweak} for $L_k(\cX_n)$ and
	the binomial versions of \eqref{e:limu22} and \eqref{e:limudhi}
	for $L_k(\cX_n)$.
	Finally, applying Proposition \ref{p:L<M} gives the same results
	for $M_k(\cX_n)$.
\end{proof}

\bibliographystyle{plain}

\end{document}